\documentclass[11pt, reqno]{amsart}
\usepackage[margin=1in]{geometry}
\usepackage{hyperref}
\usepackage{comment}
\usepackage{amsmath}
\usepackage[T2A,T1]{fontenc}
\usepackage[utf8]{inputenc}
\usepackage[russian,english]{babel}
\usepackage{color}
\usepackage{mathtools}
\usepackage{amssymb}
\usepackage{amsthm}
\usepackage{amsfonts}
\usepackage{mathrsfs}
\usepackage{caption}
\usepackage{comment}
\usepackage[super]{nth}

\usepackage{graphicx}
\usepackage{subcaption}
\graphicspath{{pictures/}}

\hypersetup{
	colorlinks=true,
	linkcolor=blue,
	citecolor=blue,
	urlcolor=blue
}

\allowdisplaybreaks

\counterwithin{figure}{section}

\newtheorem{thm}{Theorem}
\newtheorem{prp}[thm]{Proposition}
\newtheorem{lem}[thm]{Lemma}
\newtheorem{cor}[thm]{Corollary}
\newtheorem{con}[thm]{Conjecture}

\theoremstyle{definition}

\newtheorem{rem}[thm]{Remark}

\numberwithin{thm}{section}

\newcommand{\N}{\mathbb{N}}

\newcommand{\uSdw}{\rotatebox[origin=c]{180}{$\Delta$}}
\newcommand{\uSdwB}{\rotatebox[origin=c]{180}{$\Delta$}_{\text{Bor}}}

\newcommand\floor[1]{\left\lfloor #1\right\rfloor}
\newcommand\ceil[1]{\left\lceil #1\right\rceil}

\def\im{\operatorname{in}}
\def\squash{\operatorname{FS}}
\def\gin{\operatorname{gin}}
\def\bl{\operatorname{bl}}
\def\fbl{\operatorname{ubl}}
\def\LPPT{\operatorname{LPPT}}
\def\comp{\mathcal{C}}
\def\Line{\mathcal{L}}
\def\Span{\operatorname{Span}}
\def\MB{\operatorname{MB}}
\def\Bez{{\normalfont \textrm{\foreignlanguage{russian}{\CYRB}}}}

\title{A Structural Property of Generic Initial Ideals}
\author{Nikola Kuzmanovski}

\begin{document}
\begin{abstract}
We prove an asymptotic structural property of generic initial ideals.
This single phenomenon yields results on Hilbert functions, persistence, hyperplane restriction, graded Betti numbers, combinatorial shadow minimization, and Lefschetz properties.
\end{abstract}

\maketitle
\tableofcontents\label{table}


\section{Introduction}
Our field $K$ will always have characteristic $0$ unless otherwise stated. 
Let $R=K[x_1,\dots, x_n]$ and suppose $n \geq 3$.
The corresponding results have already been proved or are folklore when $n=2$.
All ideals considered are homogeneous.
Our definitions and notation are standard in the literature \cite{EisenbudBook, Green1998, PeevaBook},
but for a reader unfamiliar with these sources,
see Section \ref{DefinitionsAndNotation} for a summary.

Let $I\subseteq R$ be an ideal.
For a given term order $<$ on $R$
we can construct $\im(I)$, the initial ideal of $I$.
If $I$ is in general coordinates then $\im(I)$ is called the generic initial ideal of $I$,
and it is denoted by $\gin(I)$.
Many authors have studied how information can be transferred between $I$ and $\im(I)$.
 
One direction that has been considered is to see how information about a property of $I$ can be transferred to $\im(I)$.
It has been known for over a century that $I$ and $\im(I)$ have the same Hilbert function.
Other properties like graded Betti numbers and regularity can only increase when constructing $\im(I)$.
However,
Bayer and Stillman \cite{BayerStillman1987b} proved that the regularity of $I$ and $\gin(I)$ are equal when one uses the revlex order.
Bayer, Charalambous, and Popescu \cite{BCP} generalized this equality to extremal Betti numbers.
More recently, 
Conca and Varbaro \cite{ConcaVarbaro} showed that the extremal Betti numbers of $I$ and $\im(I)$ agree when $\im(I)$ is square free.

In another direction,
authors have used the structure of $\im(I)$ to deduce information about $I$.
Green \cite{Green1998} proved classical results of Macaulay and Gotzmann by reducing the problems to $\gin(I)$.
The structure of $\im(I)$ encodes a lot of information,
as seen in the works by Fløystad \cite{Floystad}, Fløystad and Green \cite{FloystadGreen}, and Fløystad and Stillman \cite{FloystadStillman}.

In this paper we take up both directions in the context of regular sequences.
For the remainder of the paper we always use the revlex order when taking an initial ideal.
First, we transfer information about lengths and degrees of regular sequences in $I$ to structural properties of $\gin(I)$ (Theorems \ref{IntroadvancedLefschetzBase} and \ref{IntroMonster2}).
Afterwards, we use this structure of $\gin(I)$ to prove several properties of $I$,
see Theorems \ref{IntroEGHFirst}, \ref{IntroEGHMain}, \ref{IntroEGHGotzmann}, \ref{BFGG}, \ref{IntroEGHGreenHyper}, \ref{IntroBettiEGH}, \ref{IntroWLPCITheorem}, \ref{IntroWLPCIPersistence}.
Additionally, in Theorem \ref{IntroCIGinStructure} we establish uniqueness of $\gin(I)$ in a specified range of degrees when $I$ is a complete intersection.
We develop several tools in the process of proving our main results.
One of these tools is of independent interest.
It is a generalization of Green's Crystallization Principle \cite{Green1998},
see Theorem \ref{IntroGeneralGreenNumeric}.
Section \ref{CombinatoricsMini} and Theorem \ref{BFGG} can be read completely independently from the rest of the paper.
Theorem \ref{BFGG} is a purely combinatorial result proved using several algebraic techniques.

\subsection{Structural Properties of Generic Initial Ideals}\label{MainResultsSection}
In this section we establish that the degrees and lengths of regular sequences in $I$ are encoded in the structure of $\gin(I)$.
Although $I_d$ may contain a regular sequence of length $\ell$,
this regular sequence need not survive in $\gin(I)_d$.
Thus we need to define functions that allow us to transfer this information from $I$ to $\gin(I)$.

A monomial $x_1^{d-k}x_2^k\in \gin(I)_d$ with $k\geq 1$ is {\it uncovered} with respect to $I$ if $x_1^{d-k}x_2^{k-1}x_3 \not\in \gin(I)_d$.
The {\it uncovered bottom length} of $I$ in degree $d$ is $\fbl_I(d)$,
the number of uncovered monomials of the form $x_1^{d-k}x_2^k\in \gin(I)_d$ with $k\in \{1,2,\dots, d\}$. 
Figure \ref{fig:ubl-examples} shows strongly stable sets with different uncovered bottom lengths.
All figures will be for the case $n=3$.
In any figure,
the leftmost cube denotes $x_1^{d}$,
the rightmost cube denotes $x_2^d$,
and the topmost cube denotes $x_3^d$.

\begin{figure}[htbp]
	\centering
	
	\begin{subfigure}[b]{0.32\textwidth}
		\centering
		\includegraphics[width=\textwidth]{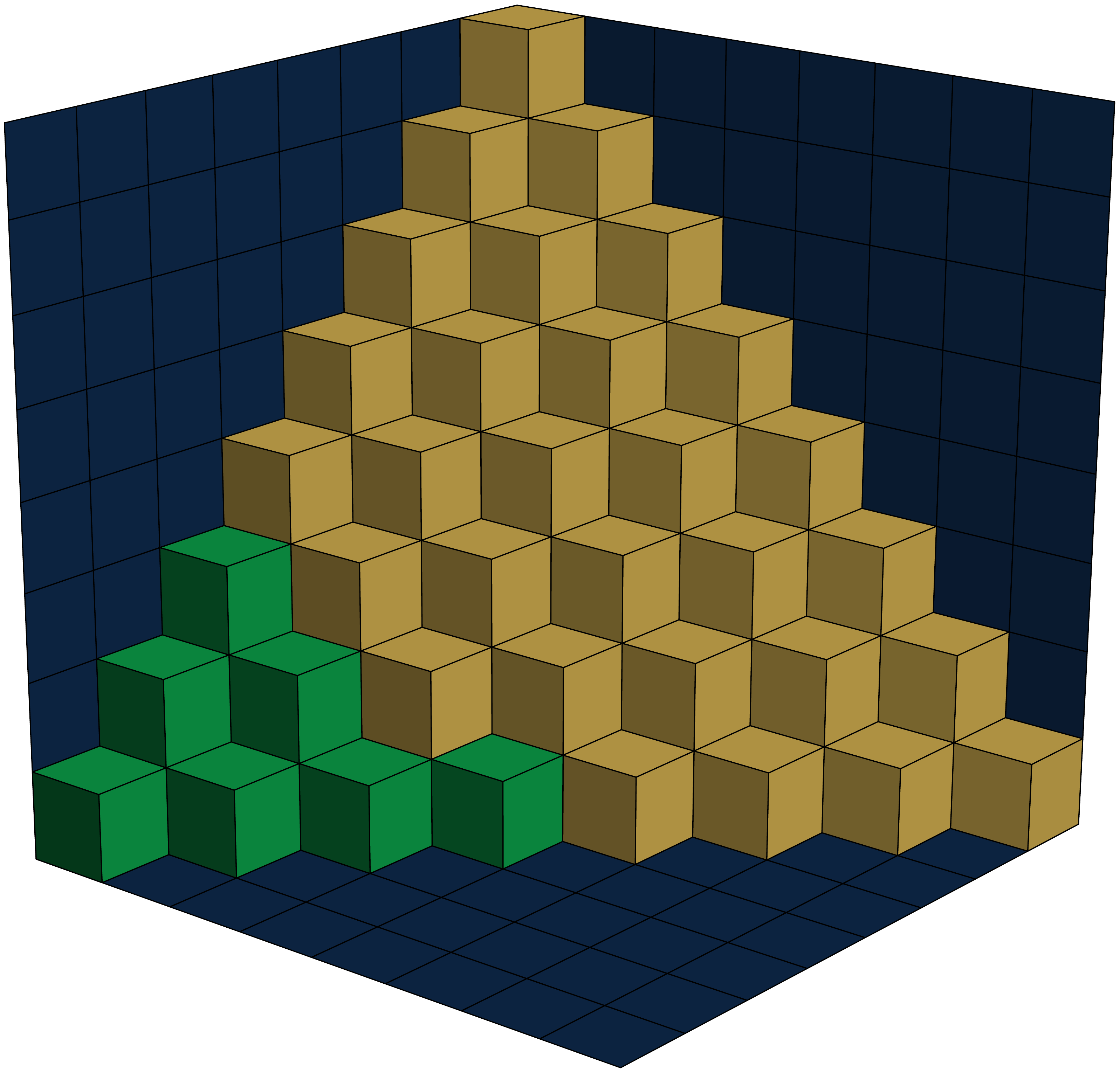}
		\caption{$\fbl_I(7) = 1$.}
		\label{fig:ubl1}
	\end{subfigure}
	\hfill
	\begin{subfigure}[b]{0.32\textwidth}
		\centering
		\includegraphics[width=\textwidth]{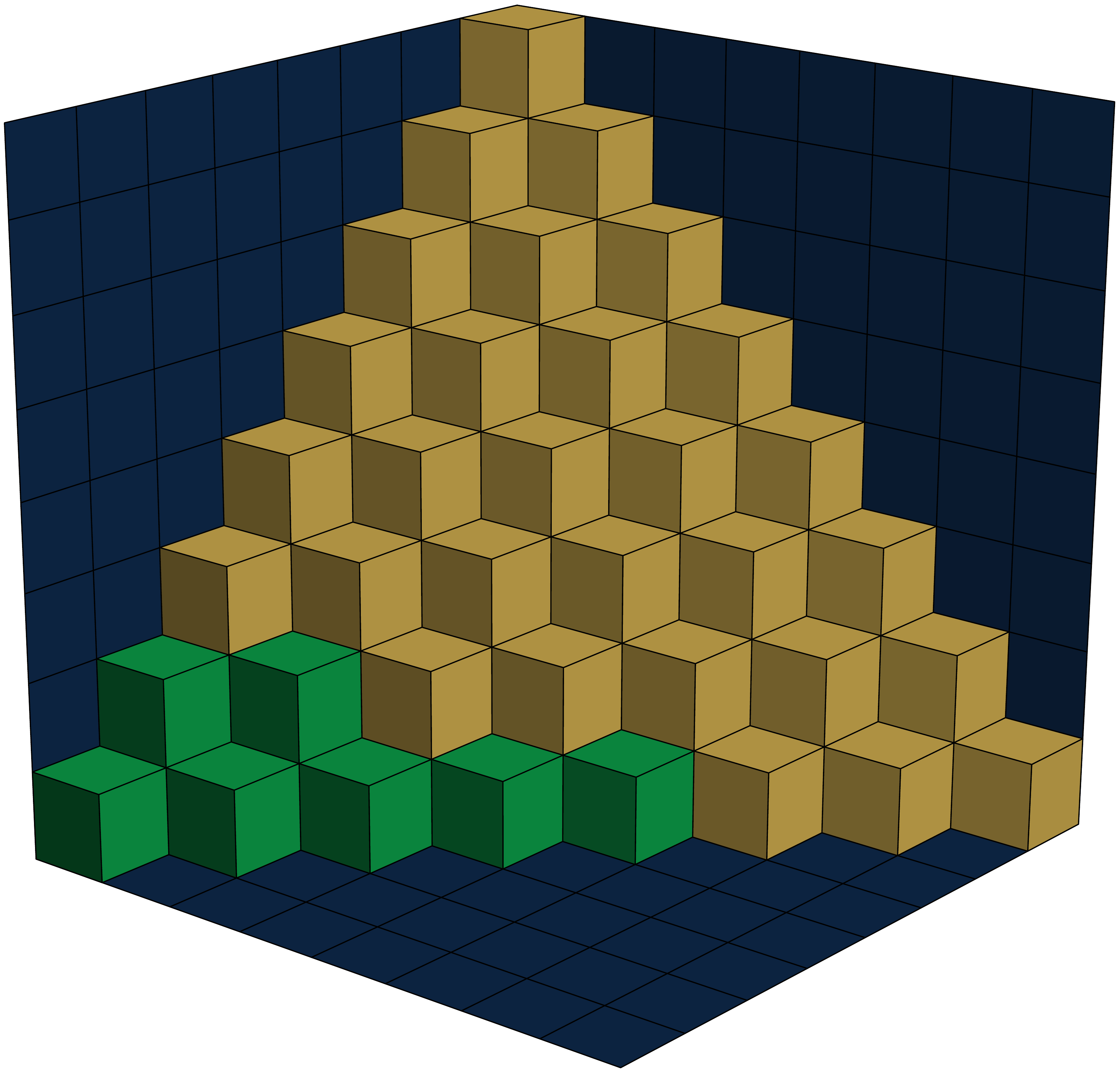}
		\caption{$\fbl_I(7) = 2$.}
		\label{fig:ubl2}
	\end{subfigure}
	\hfill
	\begin{subfigure}[b]{0.32\textwidth}
		\centering
		\includegraphics[width=\textwidth]{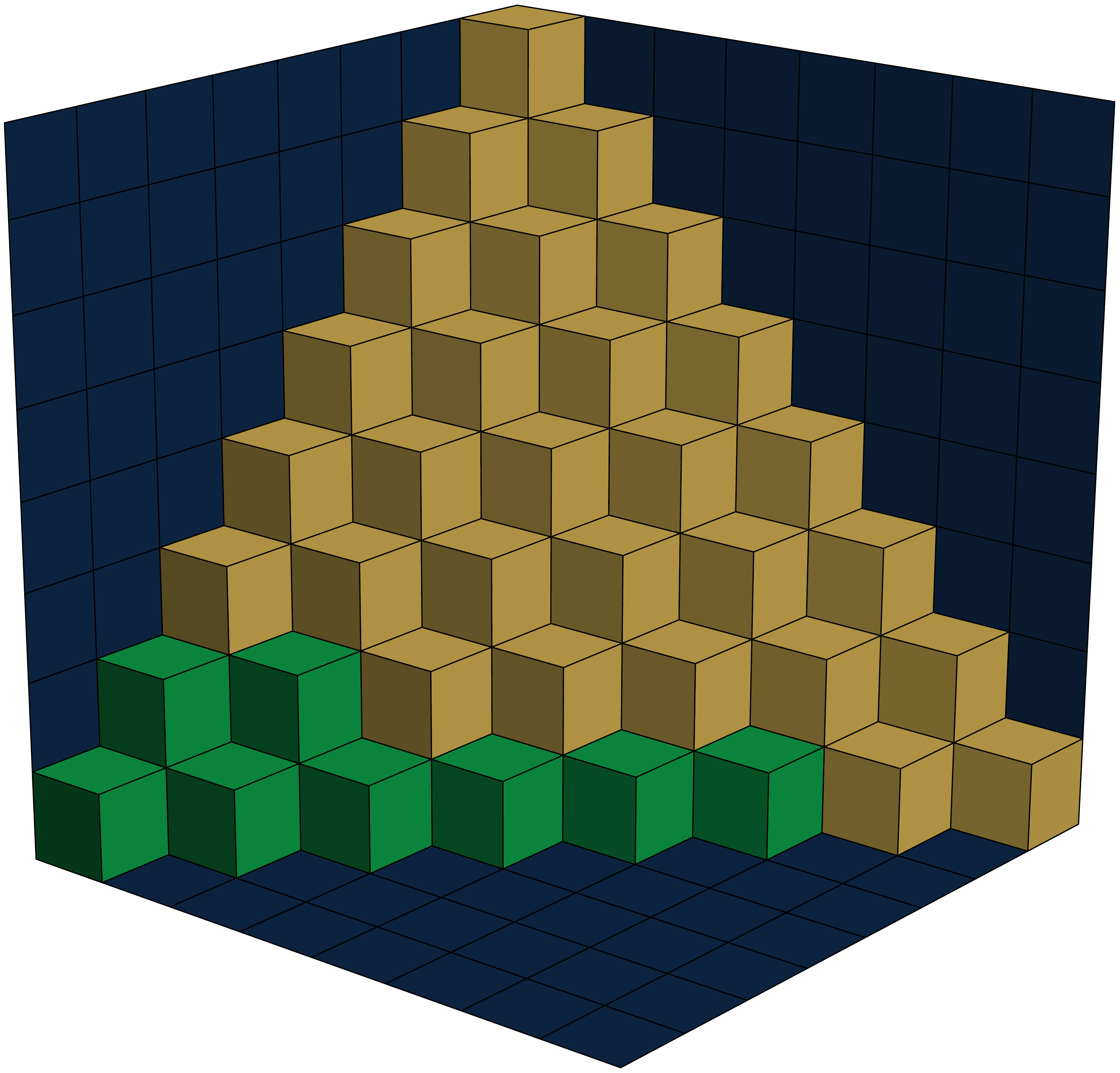}
		\caption{$\fbl_I(7) = 3$.}
		\label{fig:ubl3}
	\end{subfigure}
	
	\caption{Examples of uncovered bottom length.}
	\label{fig:ubl-examples}
\end{figure}

\begin{thm}\label{IntroadvancedLefschetzBase}
	Let $\delta, \ell \in \N$.
	There exists $D_{\delta,n}\in \N$ such that if $d\geq D_{\delta,n}$ and $I$ is equigenerated in degree $d$ with $\dim I_d=\delta$,
	and $I_d$ contains a regular sequence of length $\ell$,
	then $\fbl_I(d) \geq \ell-1$.
\end{thm}

One might be able to improve the techniques in this paper and obtain a sharp lower bound for $D_{\delta, n}$ in Theorem \ref{IntroadvancedLefschetzBase},
but the condition $d\geq D_{\delta,n}$ can't be completely removed.
Let $\alpha(I)$ denote the smallest degree of a minimal generator in $I$.
If $\alpha(I)\geq \ell$ and $\delta= \dim I_{\alpha(I)} \geq \binom{n-1+\alpha(I)-(\ell-1)}{n-1}+\ell+1$, 
then we always have $\fbl_I(\alpha(I))<\ell -1$ because $\gin(I)$ is strongly stable.
Hence the number of minimal generators and the length of the regular sequence matter.
So in order to extend Theorem \ref{IntroadvancedLefschetzBase} to a more general class of ideals,
we need to track the number of minimal generators and the length of a regular sequence.

We say that $I$ is of {\it type} $v_{\ell,t}=(a_1,\dots, a_{\ell},b_0,\dots, b_t)\in \N^{\ell+t+1}$ if there is a regular sequence $f_1,\dots, f_\ell\in I$ such that $a_i+\alpha(I)=\deg(f_i)$ and $\deg(f_i) \leq \deg(f_{i+1})$,
and for all $j\in \{0,\dots,t\}$, $I$ has $b_j$ minimal generators in degree $\alpha(I)+j$.
The ideal $I$ in Theorem \ref{IntroadvancedLefschetzBase} is of type $v_{\ell,0} = (0,\dots, 0, \delta)$ and $v_{\ell, t}=(0,\dots, 0, \delta, 0,\dots, 0)$.
Every nonzero ideal is of type $v_{1,t}$.
Also, every nonzero ideal is of some type $v_{\ell, t}$,
and if we make $t$ large enough,
then the tail of $v_{\ell, t}$ will be all zeros because every ideal is finitely generated.

The inequality $\fbl_I(d) \geq \ell-1$ also needs to be adjusted.
We want notation for the length of a regular sequence in a specific degree.
For a sequence of integers $2\leq e_1\leq e_2\leq \cdots \leq e_\ell$ and $p\in \N$,
when $e_1\leq p$ we define $\gamma (p,e_1,\dots, e_\ell)$ to be the largest integer $i\in \{1,\dots, \ell\}$ such that $e_i\leq p$,
and if $p<e_1$ then we define $\gamma (p,e_1,\dots, e_\ell) = 0$.
For $d\in \N$ and an ideal $I$ of type $v=v_{\ell,t}=(a_1,\dots, a_{\ell},b_0,\dots, b_t)\in \N^{\ell+t+1}$ we define $\gamma_{v,I}(d) = \gamma (d,\alpha(I)+a_1,\dots, \alpha(I)+a_\ell)$.
We will often abuse notation and write $\gamma_I(d)$ instead of $\gamma_{v,I}(d)$ because $v$ will be clear from context.

Finally,
we define a much more general version of uncovered bottom length.
Consider $r\geq 3$ with $r\leq n$ and $e = (e_3,e_4,\dots, e_r)\in \N^{r-2}$ such that $E=e_3+\cdots +e_r$ and $d-E\geq 1$. 
A monomial of the form $x_1^{d-j-E}x_2^jx_3^{e_3}x_4^{e_4}\cdots x_r^{e_r}\in \gin(I)_{d}$ with $j\geq 1$ is {\it $(r,e)$-uncovered} with respect to $I$,
if $x_1^{d-j-E}x_2^{j-1}x_3^{e_3}x_4^{e_4}\cdots x_r^{e_r+1} \not\in \gin(I)_d$.
The {\it $(r,e)$-uncovered bottom length} of $I$ in degree $d$ is $\fbl_I^{(r,e)}(d)$,
the number of uncovered monomials of the form $x_1^{d-j-E}x_2^jx_3^{e_3}x_4^{e_4}\cdots x_r^{e_r}\in \gin(I)_{d}$ with $j\in \{1,2,\dots, d-E\}$.
Note that $\fbl_I(d) = \fbl_I^{(3,0)}(d)$.
Figure \ref{fig:generalized-ubl} gives examples of $\fbl_I^{(r,e)}(d)$.

\begin{figure}[htbp]
	\centering
	
	\begin{subfigure}[b]{0.32\textwidth}
		\centering
		\includegraphics[width=\textwidth]{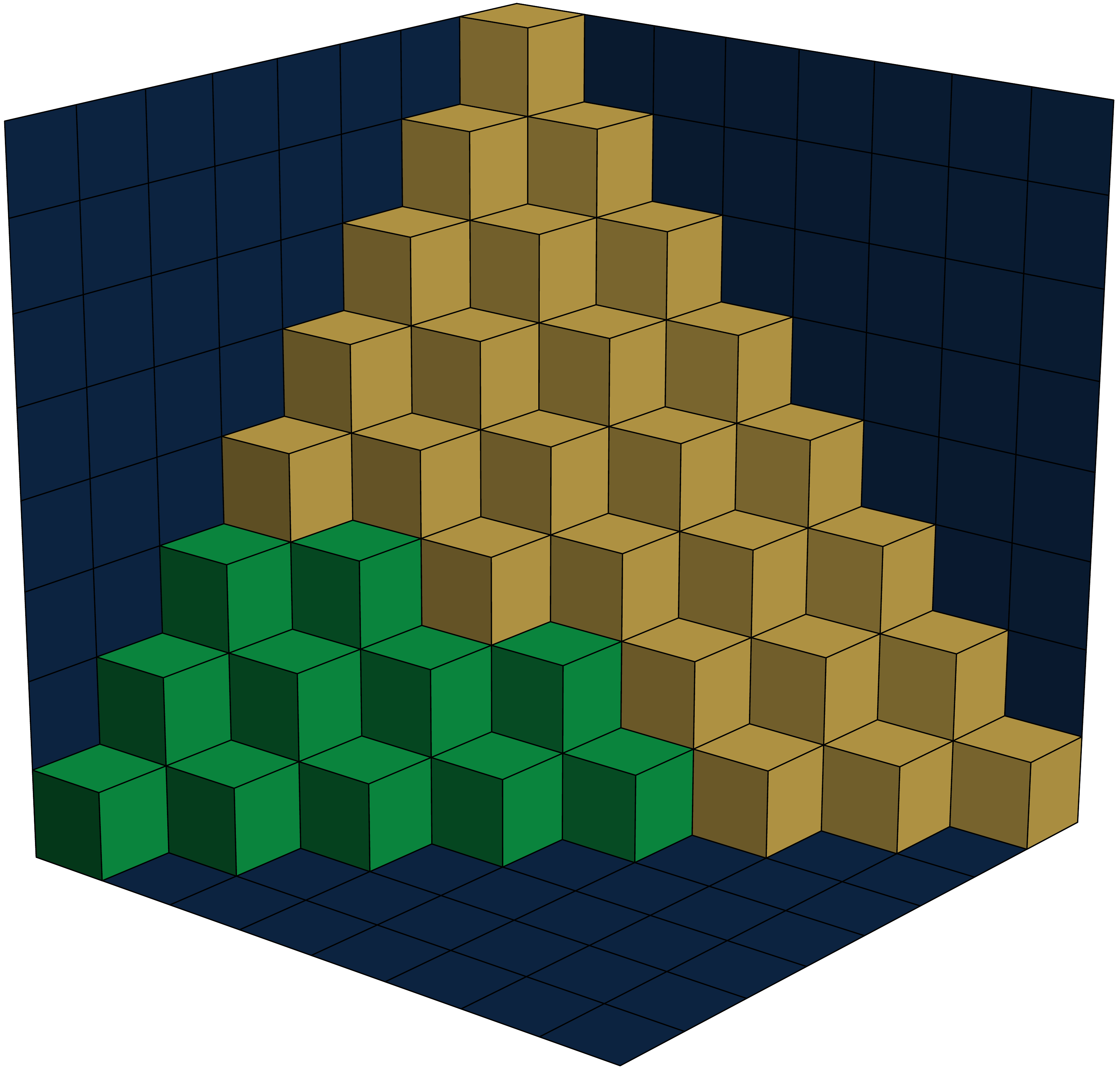}
		\caption{$\operatorname{ubl}_I^{(3,0)}(7)=0$, $\operatorname{ubl}_I^{(3,1)}(7)=1$, and $\operatorname{ubl}_I^{(3,2)}(7)=1$.}
		\label{fig:gubl1}
	\end{subfigure}
	\hfill
	\begin{subfigure}[b]{0.32\textwidth}
		\centering
		\includegraphics[width=\textwidth]{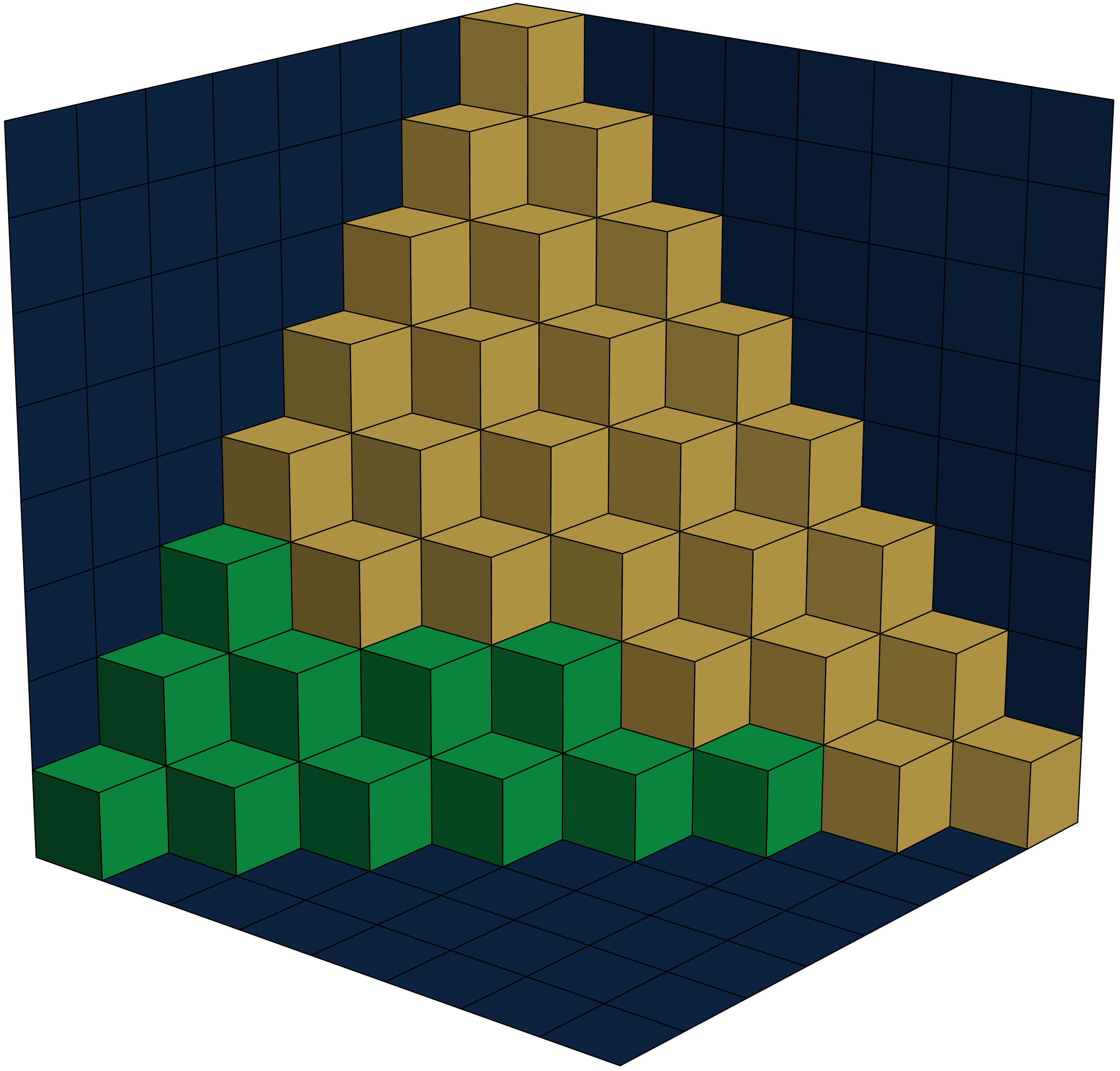}
		\caption{$\operatorname{ubl}_I^{(3,0)}(7)=1$, $\operatorname{ubl}_I^{(3,1)}(7)=2$, and $\operatorname{ubl}_I^{(3,2)}(7)=0$.}
		\label{fig:gubl2}
	\end{subfigure}
	\hfill
	\begin{subfigure}[b]{0.32\textwidth}
		\centering
		\includegraphics[width=\textwidth]{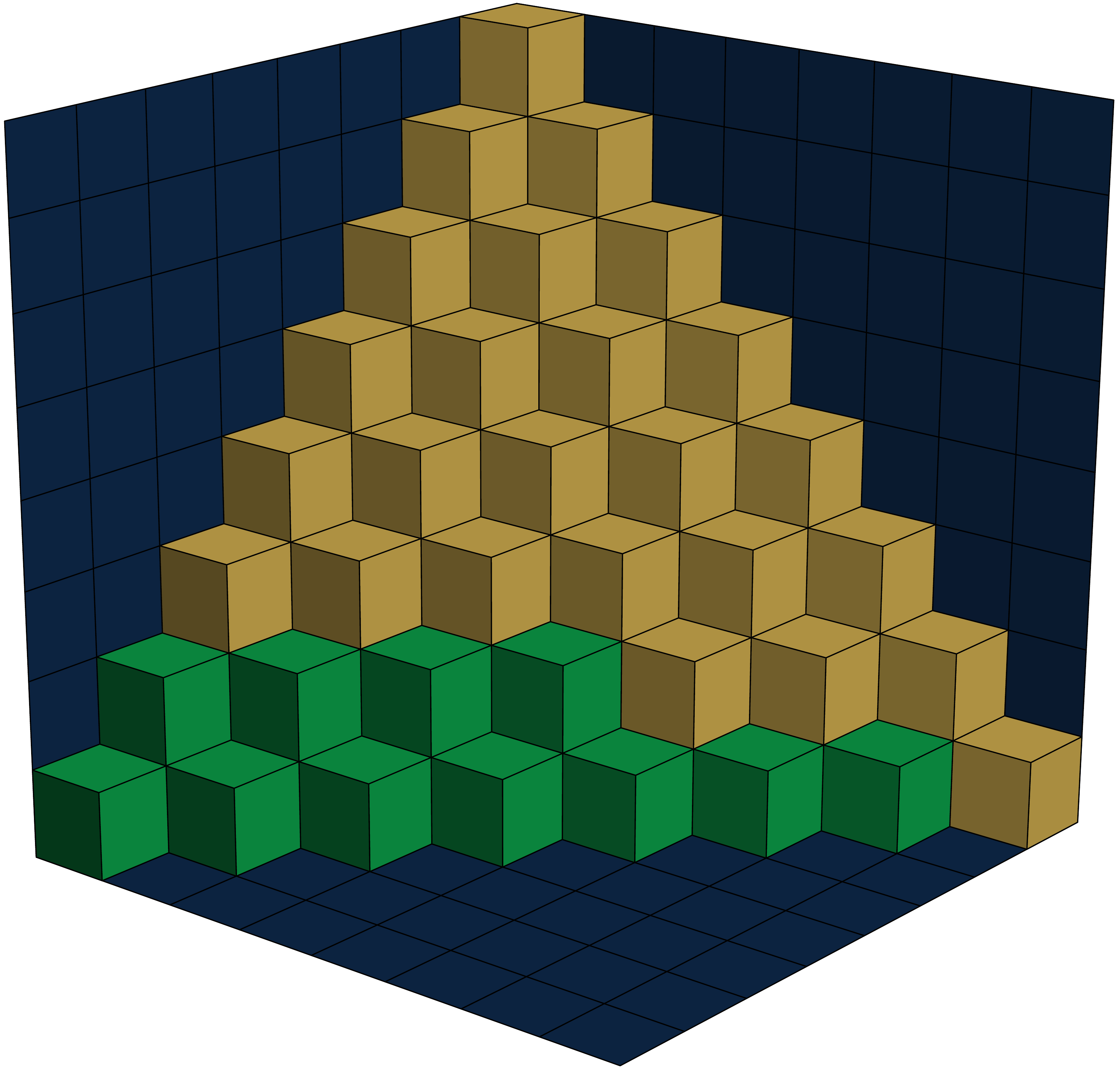}
		\caption{$\operatorname{ubl}_I^{(3,0)}(7)=2$, $\operatorname{ubl}_I^{(3,1)}(7)=3$, and $\operatorname{ubl}_I^{(3,2)}(7)=0$.}
		\label{fig:gubl3}
	\end{subfigure}
	
	\caption{Examples of generalized uncovered bottom length.}
	\label{fig:generalized-ubl}
\end{figure}

\begin{thm}\label{IntroMonster2}
	Let $v\in \N^{\ell+t+1}$.
	There exists $D_{v,n}\in \N$ such that for any ideal $I$ of type $v$ with $\alpha(I) \geq D_{v,n}$,
	for all $k\in \{0,1,\dots, t\}$, $r\in \{3,4,\dots, n\}$, and $e = (e_3,e_4,\dots, e_r)\in \N^{r-2}$ with $E=e_3+\cdots +e_r \leq k$,
	we have $\fbl_I^{(r,e)}(\alpha(I)+k) \geq \gamma_I (\alpha(I)+k-E)-1$.
\end{thm}

\subsection{Results on the Eisenbud--Green--Harris Conjecture}
In 1927, Macaulay \cite{Macaulay} proved that for every ideal in $R$,
there exists a lex ideal with the same Hilbert function.
In 1969, Clements and Lindström \cite{ClementsLindstrom} generalized Macaulay's Theorem.
They proved that for every ideal $I$ which contains an ideal $P$ that is generated by $x_1^{e_1},\dots, x_\ell^{e_\ell}$ with $e_1\leq \cdots \leq e_\ell$ and $\ell\leq n$,
there exists a lex ideal $L$ such that $L+P$ has the same Hilbert function as $I$.
We call $L+P$ a {\it lex plus powers (LPP)} ideal.
In 1993, Eisenbud, Green, and Harris \cite{EGHMain} conjectured that the same should hold if the monomial complete intersection is replaced with any complete intersection.

\begin{con}[Eisenbud--Green--Harris \cite{EGHMain}]\label{EGHConjecture}
	Let $I\subseteq R$ be an ideal that contains a regular sequence $f_1,\dots, f_\ell$ with degrees $e_1\leq \dots \leq e_\ell$.
	Then there exists a lex ideal $L$ such that $L+P$ has the same Hilbert function as $I$,
	where $P$ is the ideal generated by $x_1^{e_1},\dots, x_\ell^{e_\ell}$.
\end{con}

We will refer to Conjecture \ref{EGHConjecture} as the {\it EGH Conjecture}.
We will often state that there exists an LPP ideal without specifying the degrees for the powers when they are clear from context.
Following Güntürkün and Hochster \cite{HoschsterSema},
a {\it defect} $\delta$ ideal is an ideal that is minimally generated by $n+\delta$ forms and contains a regular sequence of length $n$.
Francisco \cite{Francisco} proved that the EGH Conjecture holds in degrees $\leq d+1$ for defect $1$ ideals,
where $d$ is the degree of the minimal generator not belonging to the complete intersection.
Güntürkün and Hochster \cite{HoschsterSema} proved that the EGH Conjecture holds in degrees $\leq 3$,
for defect $2$ ideals generated by quadrics.
They also proved the same statement for defect $3$ ideals,
but only when $n=5$.
We prove that the EGH Conjecture holds in degrees $\leq d+1$ for defect $\delta$ ideals generated in degree $d$,
except for finitely many $d$.

\begin{thm}\label{IntroEGHFirst}
	Let $\delta\in \N$.
	There exists $D_{\delta,n}\in \N$ such that for any equigenerated defect $\delta$ ideal $I$ with $\alpha(I)\geq D_{\delta,n}$,
	there exists an LPP ideal with the same Hilbert function in degrees $\leq \alpha(I)+1$.
\end{thm}

Note that equigenerated defect $\delta$ ideals are of type $v_{n,0} = (0,\dots, 0, n+\delta)$ and $v_{n, t}=(0,\dots, 0, n+\delta, 0,\dots, 0)$.
One might be interested in a bound for $D_{\delta, n}$.
This can be obtained by following the proofs of results used to prove Theorem \ref{EGHTheorem}.
However, the arguments can be specialized to obtain better bounds for $D_{\delta,n}$ than what the current proofs produce.
For example, when $n=3$ we can get
\begin{align*}
	\delta \leq \binom{\ceil{\frac{D_{\delta,n}+3}{2}}-2}{2} -1.
\end{align*}
So, if $\delta = 5$ we have $D_{\delta,n}\geq 8$.
Computing a bound for $D_{\delta,n}$ gets significantly more complicated when $n\geq 4$.
All the results in the paper are obtained by a unified approach.
Optimizing individual results would make the paper much longer.

Another direction one can take is to prove that EGH holds in a larger range of degrees beyond $d$.
Caviglia and De Stefani \cite{CavigliaDeStefani2} proved that the EGH Conjecture holds for defect $1$ ideals in $3$ variables.
Our next result extends Theorem \ref{IntroEGHFirst} by relaxing the equigeneration and range requirements.

\begin{thm}\label{IntroEGHMain}
	Let $v\in \N^{\ell+t+1}$.
	There exists $D_{v,n}\in \N$ such that for any ideal $I$ of type $v$ with $\alpha(I) \geq D_{v,n}$,
	there exists an LPP ideal with the same Hilbert function as $I$ in degrees $\leq \alpha(I)+t$.
\end{thm}

Theorem \ref{IntroEGHMain} establishes infinitely many cases of the EGH Conjecture, allowing arbitrary natural numbers for the fixed number of generators, the range over which the numbers of generators are prescribed, the range over which the LPP ideal has the same Hilbert function, and the number of variables.
As noted above, previous results assuming that the regular-sequence degrees are equal restrict at least three of these four parameters to values at most $5$.
An all-degree version of Theorem \ref{IntroEGHMain}, together with the removal of the large-initial-degree hypothesis, would
settle the EGH Conjecture. 

\subsection{Macaulay's Bound and the Eisenbud--Green--Harris/Clements--Lindström  Bound}
Macaulay's Theorem is extremely useful as a purely computational tool.
It gives a bound on the Hilbert function in the next degree,
if one knows the Hilbert function in the current degree.
We make this more precise.
Let $I$ be an ideal and $L$ be the corresponding lex ideal given by Macaulay's Theorem.
Then for all $d\in \N$ we have $\dim I_{d+1} = \dim L_{d+1} \geq \dim R_1L_{d}$.
We refer to this inequality as {\it Macaulay's Bound in degree $d$}.
If one knows the value of $\dim L_d$ and uses that to get a nice expression for $\dim R_1L_d$,
then we get a computational lower bound for $\dim I_{d+1}$.
Such bounds are standard in the literature and are given as a sum of binomial coefficients.
Some authors consider an upper bound for $\dim (R/I)_{d+1}$,
which is an equivalent view.
We say that {\it Macaulay's Bound is sharp in degree $d$} if $\dim I_{d+1} = \dim R_1L_{d}$.

One can consider the same ideas for the EGH Conjecture.
Let $I$ be an ideal of type $v=v_{\ell,t}=(a_1,\dots, a_{\ell},b_0,\dots, b_t)\in \N^{\ell+t+1}$ and consider the LPP ideal $J$ given by Theorem \ref{IntroEGHMain} in degrees $< \alpha(I)+t$.
Then we have $\dim I_{d+1} = \dim J_{d+1} \geq \dim R_1J_{d}$.
However, this lower bound might not be optimal if we have a minimal generator in degree $d+1$ that belongs to the complete intersection given by $v$.
Another problem is that there might be another type that produces an even better bound,
due to longer regular sequences in smaller degrees.
Hence we should modify our definition to account for long regular sequences.

For example, let us take $I\subseteq K[x,y,z]$ to be generated by $\{x^d,x^{d-1}y,x^{d-1}z,x^{d-2}y^2,y^d, z^{d+1}\}$ with $d\geq 5$.
Then $\dim I_d = 5$, $\dim I_{d+1} = 12$, and $\dim I_{d+2}=22$.
Also, $I$ is of types $w=w_{3,2}=(0,0,1,5,1,0)$ and $u = u_{2,2} = (0,1,1,5,1,0)$.
Then if $d$ is sufficiently large, 
Theorem \ref{IntroEGHMain} produces LPP ideals $J^w$ and $J^u$ corresponding to the types $w$ and $u$.

Let us examine $w$ first which is the ``optimal'' type to use.
Well, $\dim J^w_d = 5$ and hence we are guaranteed $\dim R_1J^w_d \geq 11$.
However, $\dim I_{d+1} = 12$.
We can adjust for this because we know that there will be a longer regular sequence in $I_{d+1}$,
and the type $w$ records this.
So, we have $\dim I_{d+1} = 12 = 3-2 + 11 = \gamma_{w,I}(d+1) - \gamma_{w,I}(d) + \dim R_1J^w_{d}$.
Similarly, we get $\dim I_{d+2} = 22 = 3-3 +22 = \gamma_{w,I}(d+2) - \gamma_{w,I}(d+1) + \dim R_1J^w_{d+1}$.

The type $u$ carries worse information.
Since $u$ does not know that $I_d$ contains a regular sequence of length $2$,
we can only conclude that $\dim R_1 J^u_d \geq 9$,
which is what Macaulay's Bound provides. 
The type $u$ tells us that $I_{d+1}$ has a regular sequence of length $3$.
So, we could attempt a similar correction analogous to the one used for the type $w$,
more specifically consider $\gamma_{u,I}(d+1) - \gamma_{u,I}(d) + \dim R_1J^u_{d}$.
However, this gives $\gamma_{u,I}(d+1) - \gamma_{u,I}(d) + \dim R_1J^u_{d} = 3-1+9 = 11 < 12 = \dim I_{d+1}$.

One might now be tempted,
for an ideal $I$ of type $v$ with LPP ideal $J$ produced from Theorem \ref{IntroEGHMain},
to consider the inequality $\dim I_{d+1}\geq \gamma_I(d+1)-\gamma_I(d)+\dim R_1J_d$.
However, this inequality might not hold for some types.
For example,
consider $I\subseteq K[x,y,z]$ to be the ideal generated by $x^d$.
Well,
$I$ is of type $h = h_{1,0} = (1,1)$,
and let $J$ be the LPP ideal given by Theorem \ref{IntroEGHMain}.
Then $\gamma_I(d+1)-\gamma_I(d)+\dim R_1J_d = 1-0+3 = 4 > 3 = \dim I_{d+1}$.

We define minimal types to fix all of the issues presented above.
Let $I\subseteq R$ be an ideal of type $v = v_{\ell, t}\in \N^{\ell+t+1}$. 
For $i\in \{1,\dots, \ell\}$ define $\rho_I(i)$ to be the smallest degree $d$ such that $I$ contains a regular sequence of length $i$ in degree $d$.
For $d\in \N$ define $\lambda_I(d)$ to be the longest length of a regular sequence inside $I$ in degree $d$.
Let $f_1,\dots, f_\ell$ be a regular sequence given by the type $v$ with $\deg(f_1)\leq \cdots \leq \deg(f_\ell)$.
We say that {\it $v$ is minimal for $I$} if $f_1,\dots, f_\ell$ are minimal generators, 
for every $i\in \{1,\dots, \ell\}$ we have $\rho_I(i) = \deg(f_i)$,
and for every $d\in \N$ we have $\lambda_I(d) = \gamma_I(d)$. 

When $v$ is minimal, we have $\dim I_{d+1}= \dim J_{d+1} \geq \gamma_I(d+1)-\gamma_I(d)+\dim R_1J_d$.
We refer to this inequality as the {\it EGH or CL Bound in degree $d$}.
The EGH Bound provides an improvement of Macaulay's Bound which is more complicated to express computationally.
In Theorem \ref{IntroEGHFirst}, 
the best and ``most common'' improvement is $(n-1)^2$ more than Macaulay's Bound.
We say that the {\it EGH Bound is sharp in degree $d$} if $\dim I_{d+1}=\gamma_I(d+1)-\gamma_I(d)+\dim R_1J_d$,
whenever $v$ is minimal.

When $\gamma_I(d+1)=\gamma_I(d)$,
the EGH Bound being sharp means $\dim I_{d+1} = \dim R_1J_d$,
which is exactly the statement that Macaulay's Bound is sharp when we replace the LPP ideal with the lex ideal.
When $\ell=1$, sharpness of the EGH Bound is equivalent to sharpness of Macaulay’s Bound in every degree except $\alpha(I)-1$.

\subsection{A Generalization of Gotzmann's Persistence Theorem}
Gotzmann's Persistence Theorem \cite{Gotzmann} states that if Macaulay's Bound is sharp in degree $d$ and there are no minimal generators in degrees $> d$,
then Macaulay's Bound is sharp in degrees $\geq d$.
Generalizations of Gotzmann's Theorem have been proved for ideals containing monomial complete intersections \cite{AramovaHerzogHibi, BezrukovPersistence, FurediGriggs, Gasharov, GasharovMuraiPeeva}. 
The results \cite{BezrukovPersistence, FurediGriggs} were discovered in a purely combinatorial setting,
and the papers analyzed what happens when the bound from the Kruskal--Katona Theorem \cite{Katona, Kruskal} is sharp.
We generalize Gotzmann's Persistence Theorem for ideals containing arbitrary complete intersections.
If $I$ is of type $v$ and $f_1,f_2,\dots, f_\ell$ is a regular sequence in $I$ given by $v$,
then we say that $I$ is generated in degrees at most $q$ with respect to $v$,
if all minimal generators of $I$ in degrees $q+1,q+2,\dots, \alpha(I)+t$  belong to the ideal generated by $f_1,\dots, f_\ell$.

\begin{thm}\label{IntroEGHGotzmann}
	Let $v\in \N^{\ell+t+1}$.
	There exists $D_{v,n}\in \N$ such that for any ideal $I$ of minimal type $v$ with $\alpha(I) \geq D_{v,n}$,
	there exists an LPP ideal with the same Hilbert function as $I$ in degrees $\leq \alpha(I)+t$.
	Furthermore, if $I$ is generated in degrees $\leq q$ with respect to $v$ and the EGH Bound is sharp in degree $q$,
	then the EGH Bound is sharp in degrees $q+1,q+2,\dots, \alpha(I)+t-1$.
\end{thm}

The minimal type requirement in Theorem \ref{IntroEGHGotzmann} can't be dropped.
Gasharov gives an example in \cite{GasharovGotz} for an ideal containing a monomial complete intersection.
The example is for $n=3$ and an ideal $I$ with $\alpha(I) = 2$ and of type $v_{3,1} = (1,1,1,3, 1)$.
Theorem \ref{IntroEGHGotzmann} is new even for the case of monomial complete intersections.

\subsection{A Generalization of the Bezrukov--Füredi--Griggs Theorem}\label{CombinatoricsMini}
The main result in this section is Theorem \ref{BFGG} and it is a corollary of Theorem \ref{IntroEGHGotzmann}.
It's a purely combinatorial statement.
We go over some definitions that are standard in the combinatorics literature \cite{BezrukovLeck, Engel}.

For $\ell\leq n$ and $2\leq e_1\leq \cdots \leq e_\ell$ define $G_e = \{0,1,\dots ,e_1-1\}\times \cdots \times \{0,1,\dots, e_\ell -1\} \times \N^{n-\ell}$,
where $e= (e_1,\dots, e_\ell)$.
One can think of $\N^n$ as an infinite grid,
and hence think of $G_e$ as a grid whose coordinates are bounded in some directions.
If $\ell=n$ then $G_e$ is a finite grid.
The case $\ell =n$ and $2=e_1=\cdots =e_n$ makes $G_e$ into the $n$-dimensional hypercube,
and we denote it by $H_n$.
When $\ell = n$ and $d=e_1=\cdots e_n$ we write $G_{n,d}$ for $G_e$.

There is a partial order on $G_e$ induced from the one on $\N^n$,
where for two tuples $u,v\in \N^n$ with $u=(u_1,\dots, u_n)$ and $v=(v_1,\dots, v_n)$ we define $u\leq v$ if for all $i\in \{1,\dots, n\}$ we have $u_i\leq v_i$.
The {\it rank} of an element $u\in \N^n$ is defined to be the sum of its coordinates.
The {\it $d$-th level} of $G_e$ is defined to be all elements of rank $d$ in $G_e$.
A set $S\subseteq G_e$ is called a {\it $d$-set} if all the elements in $S$ have rank $d$.
For tuples $u,v\in G_e$ of the same rank,
we define $v>u$ in the {\it lexicographic order} if the first nonzero entry of $v-u$ is positive.
The last $k$ elements in the lexicographic order of rank $d$ are called the {\it lex segment} of size $k$ in level $d$.

The shadow of $u\in G_e$ is all the elements in $G_e$ obtained by adding $1$ to any coordinate of $u$.
The shadow of a set $S\subseteq G_e$ is the union of all the shadows of the elements in $S$.
The Clements--Lindström Theorem states that a lex segment of size $k$ in level $d$ of $G_e$ has the smallest shadow among all $d$-sets of size $k$,
and that the shadow of a lex segment is a lex segment.
The Kruskal--Katona Theorem is the special case of the Clements--Lindström Theorem when $G_e=H_n$.

A $d$-set $S\subseteq G_e$ in level $d$ is said to be {\it optimal} if its shadow has the same size as the shadow of the lex segment of size $|S|$ in level $d$.
Füredi and Griggs \cite{FurediGriggs},
and independently Bezrukov \cite{BezrukovPersistence},
proved that if a $d$-set $S\subseteq H_n$ is optimal,
then its shadow is optimal.
Due to Gasharov's example mentioned in the previous section,
this result can't be generalized by just replacing $H_n$ with $G_e$.
It turns out that we can make this replacement if the grid and level are large enough.

\begin{thm}\label{BFGG}
	Let $k,\delta \in \N$.
	There exists $D = D(k,\delta,n)\in \N$ such that for all $d\geq D$,
	if $S$ is an optimal $(d+k)$-set in the grid $G_{n,d}$ with $|S|=\delta$,
	then the shadow of $S$ is optimal.
\end{thm}

Theorem \ref{BFGG} is proved by using several algebraic tools.
It would be nice to have a purely combinatorial proof.
In fact, it would be nice to be able to completely solve the problem of when an optimal set in a grid has an optimal shadow.
As Theorem \ref{IntroWLPCIPersistence} shows, 
sharpness of the EGH Bound, or equivalently optimality of sets, forces the weak Lefschetz property.

\subsection{A Generalization of Green's Hyperplane Restriction Theorem}
Another direction of research is to consider the Hilbert function of a graded ideal after modding out by a general linear form.
A classical paper of Green \cite{GreenHyper} establishes that lex ideals provide the optimal bound yet again.
Green's Hyperplane Restriction Theorem states that if $I$ is a homogeneous ideal,
$L$ is the corresponding lex ideal,
and $f$ is a general linear form,
then $\dim (R/(I,f))_d\leq \dim (R/(L,f))_d$.
The value $\dim (R/(L,f))_d$ can be expressed as a sum of binomial coefficients, see \cite{BrunsHerzog}.
This again gives a completely numerical bound just like the case for Macaulay's Theorem.
We improve the  bound in Green's Hyperplane Restriction Theorem by replacing the lex ideal with the lex plus powers ideal.

\begin{thm}\label{IntroEGHGreenHyper}
	Let $v\in \N^{\ell+t+1}$.
	There exists $D_{v,n}\in \N$ such that for any ideal $I$ of type $v$ with $\alpha(I) \geq D_{v,n}$,
	there exists an LPP ideal $J$ with the same Hilbert function as $I$ in degrees $\leq \alpha(I)+t$.
	Furthermore, for $q\in \{n,n-1,\dots, 3\}$ there exist linear forms $f_n,f_{n-1},\dots, f_q$ such that for all $k\in \{0,1,\dots, t-1\}$ we have $\dim \left(R/(J,f_n,\dots, f_q)\right)_{\alpha(I)+k} \geq \dim \left(R/(I,f_n,\dots, f_q)\right)_{\alpha(I)+k}$.
\end{thm}

Note that we stated Green's Hyperplane Restriction Theorem by just using a single general linear form.
A simple inductive argument can generalize the statement to more linear forms,
since if given a lex ideal $L$,
evaluating $x_n=0$ gives a lex ideal in $K[x_1,\dots, x_{n-1}]$.
The case for more than one linear form in Theorem \ref{IntroEGHGreenHyper} is not as simple to handle.

\subsection{A Generalization of the Bigatti--Hulett--Pardue Theorem}
Bigatti \cite{Bigatti1993} and Hulett \cite{Hulett1993} in characteristic $0$,
and Pardue \cite{Pardue1996} in characteristic $p$,
proved that lex ideals have the largest graded Betti numbers among all ideals with the same Hilbert function.
Another natural direction is to replace a lex ideal with a lex plus powers ideal,
and prove that the graded-Betti-numbers inequality still holds.
This would provide a numerical improvement to the Bigatti--Hulett--Pardue Theorem for ideals containing regular sequences.
The conjecture first appeared is in a paper of Evans and Richert \cite{EvansRichert},
and later in a paper of Francisco and Richert \cite{FranciscoRichert}.
It is attributed to Evans.
This conjecture is today known as the Lex Plus Powers (LPP) Conjecture.

In order to make progress on the LPP Conjecture,
one first needs to establish the existence of an LPP ideal,
that is,
prove the EGH Conjecture or some case of it.
This means that there has been less progress on the LPP Conjecture than to the EGH Conjecture.
As the Clements--Lindström Theorem is the case of the EGH Conjecture for powers of the variables,
a natural step is to start with this case.
In 2008, Mermin, Peeva, and Stillman \cite{MerminPeevaStillman} proved the LPP Conjecture for ideals containing the squares of the variables.
In 2011, Mermin and Murai \cite{MerminMurai} settled the LPP Conjecture for ideals containing powers of the variables.
In 2008, Caviglia and Maclagan \cite{CavigliaMaclagan} proved a case of the EGH Conjecture for ideals containing regular sequences that have degrees that are far apart.
In 2018, Caviglia and Sammartano \cite{CavigliaSammartano} settled the LPP Conjecture for the case of the EGH Conjecture that Caviglia and Maclagan proved.
We generalize the Bigatti--Hulett--Pardue Theorem result by replacing a lex ideal with the generic initial ideal (with respect to revlex) of an LPP ideal.

\begin{thm}\label{IntroBettiEGH}
	Let $v\in \N^{\ell+t+1}$.
	There exists $D_{v,n}\in \N$ such that for any ideal $I$ of type $v$ with $\alpha(I) \geq D_{v,n}$,
	there exists an LPP ideal with the same Hilbert function as $I$ in degrees $\leq \alpha(I)+t$.
	Furthermore, we have the inequality of Betti numbers $b_{p,p+q}(I) \leq b_{p,p+q}(\gin(J))$ for $q\leq \alpha(I)+t-1$.
\end{thm}

Note that for a lex ideal $L$ we have $\gin(L)=L$,
hence when $\ell=1$ Theorem \ref{IntroBettiEGH} reduces to the Bigatti--Hulett--Pardue Theorem.

\subsection{Results on the Weak Lefschetz Property}
The weak and strong Lefschetz properties are abstractions of the Hard Lefschetz Theorem in complex geometry.
It is known that the WLP fails for complete intersections in positive characteristic.
A foundational result in the area is due to Stanley \cite{Stanley},
which proves that monomial complete intersections satisfy the weak and strong Lefschetz properties.
It has been conjectured that this should hold for any complete intersection.
Harima, Migliore, Nagel, and Watanabe \cite{JuanWLPHeight3} proved that the WLP holds for any complete intersection when $n=3$.
Boij, Migliore, Miró-Roig, and Nagel \cite{JuanWLPHeight4} proved that the WLP holds for any equigenerated complete intersection,
but only in a specified range of degrees,
when $n=4$. 
They also give a smaller range if one drops the equigenerated assumption.
For $n\geq 4$,
Beorchia and Miró-Roig \cite{WLPHeigthN} proved a similar statement,
but with a more restricted range.
We generalize these results by quotienting by more than one linear form.

\begin{thm}\label{IntroWLPCITheorem}
	Let $v\in \N^{\ell+t+1}$ and $i\in \{n,n-1,\dots, 3\}$.
	There exists $D_{v,n}\in \N$ such that for any complete intersection $I$ of minimal type $v$ with $\alpha(I) \geq D_{v,n}$,
	there exist linear forms $f_n,f_{n-1},\dots, f_{i+1}$ for which the WLP holds for $I+(f_n,\dots, f_{i+1})$ in degrees $\leq \alpha(I)+t-1$.
\end{thm}

When $i=n$ by $(f_n,\dots, f_{n+1})$ we mean the $0$ ideal.
This case just says that $I$ satisfies the WLP.
The next theorem generalizes Theorem \ref{IntroWLPCITheorem}.
First, it extends the current WLP results from complete intersections to a much larger class of ideals.
Second, it also establishes the stronger case of WLP when we use more general linear forms.
Finally, it links the EGH Conjecture to the WLP.

\begin{thm}\label{IntroWLPCIPersistence}
	Let $v\in \N^{\ell+t+1}$ and $i\in \{n,n-1,\dots, 3\}$.
	There exists $D_{v,n}\in \N$ such that for any ideal $I$ of minimal type $v$ with $\alpha(I) \geq D_{v,n}$,
	there exists an LPP ideal with the same Hilbert function as $I$ in degrees $\leq \alpha(I)+t$.
	Furthermore, if $I$ is generated in degrees $\leq q$ with respect to $v$ and the EGH Bound is sharp in degree $q$,
	then there exist linear forms $f_n,f_{n-1},\dots, f_{i+1}$ such that the WLP holds for $I+(f_n,\dots, f_{i+1})$ in degrees $q,q+1,\dots, \alpha(I)+t-1$.
\end{thm}

Theorem \ref{IntroWLPCITheorem} is a corollary of Theorem \ref{IntroWLPCIPersistence},
since the EGH Bound is sharp in all degrees for complete intersections.

\subsection{Additional Results on Initial Ideals}
If one specializes Theorem \ref{IntroMonster2} to the case of complete intersections,
then we can list the monomials in the generic initial ideal for the specified range of degrees.
The generic initial ideal given by Theorem \ref{IntroCIGinStructure} has been called an almost revlex ideal in the literature \cite{ParkHyun, Moreno, PardueGeneric}.
A monomial ideal is called an {\it almost revlex ideal},
if whenever a minimal monomial generator is in the ideal, 
then all monomials before that minimal generator are in the ideal.
See Figure \ref{fig:ci7-theorem112} for an example of Theorem \ref{IntroCIGinStructure}.

\begin{figure}[htbp]
	\centering
	
	\begin{subfigure}[b]{0.32\textwidth}
		\centering
		\includegraphics[width=\textwidth]{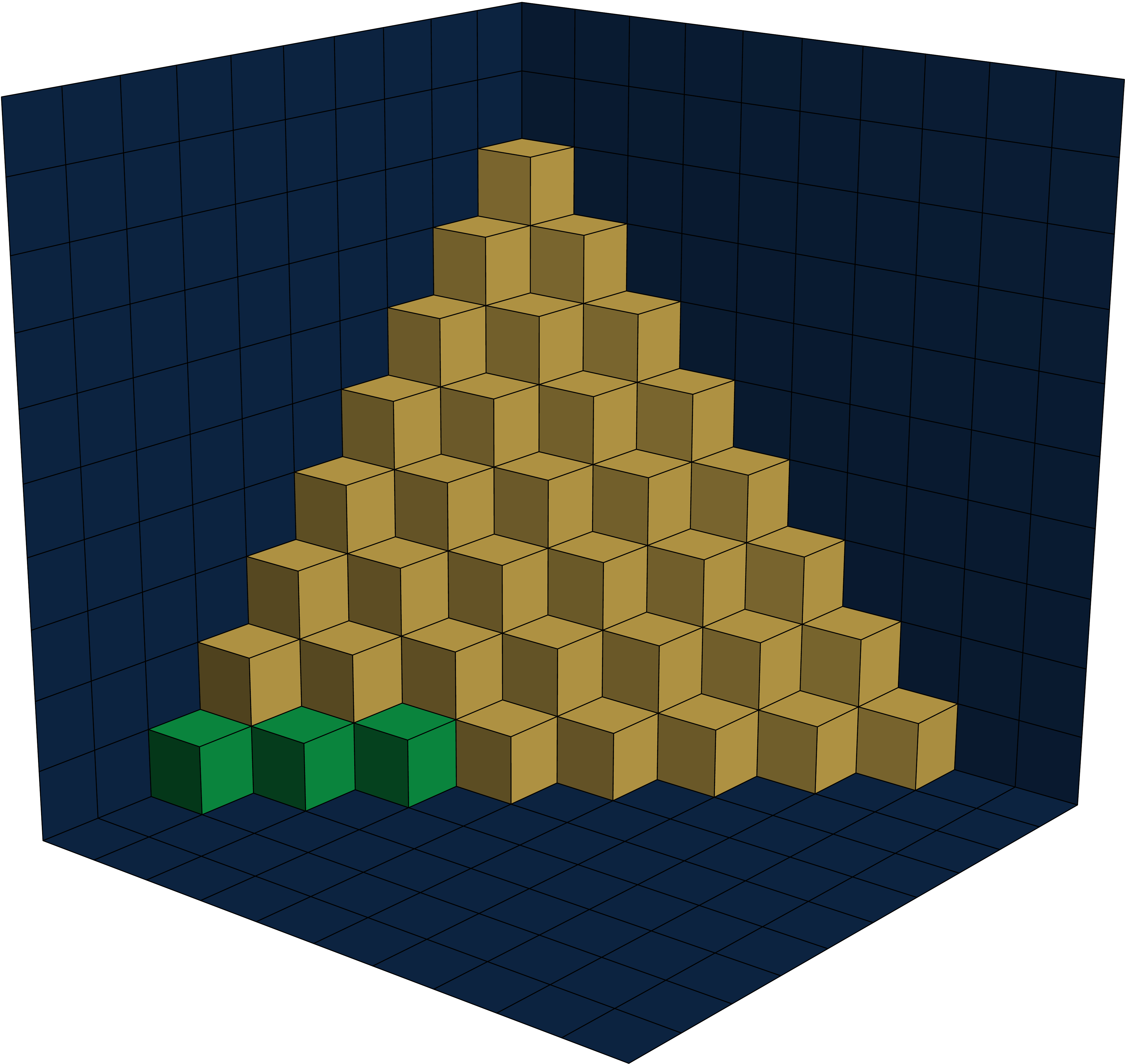}
		\caption{$\gin(I)_7$.}
		\label{fig:ci7-deg7}
	\end{subfigure}
	\hfill
	\begin{subfigure}[b]{0.32\textwidth}
		\centering
		\includegraphics[width=\textwidth]{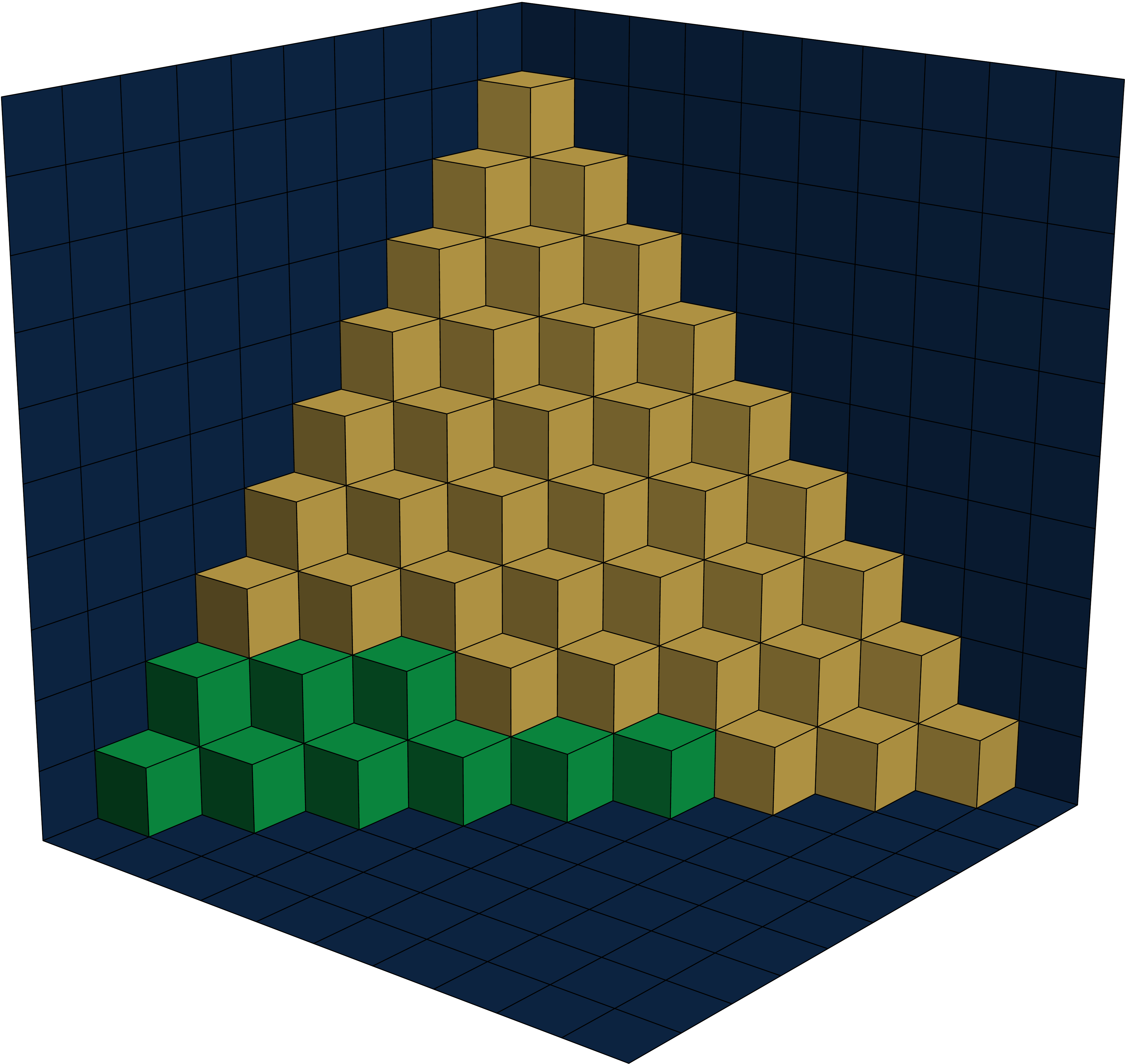}
		\caption{$\gin(I)_8$.}
		\label{fig:ci7-deg8}
	\end{subfigure}
	\hfill
	\begin{subfigure}[b]{0.32\textwidth}
		\centering
		\includegraphics[width=\textwidth]{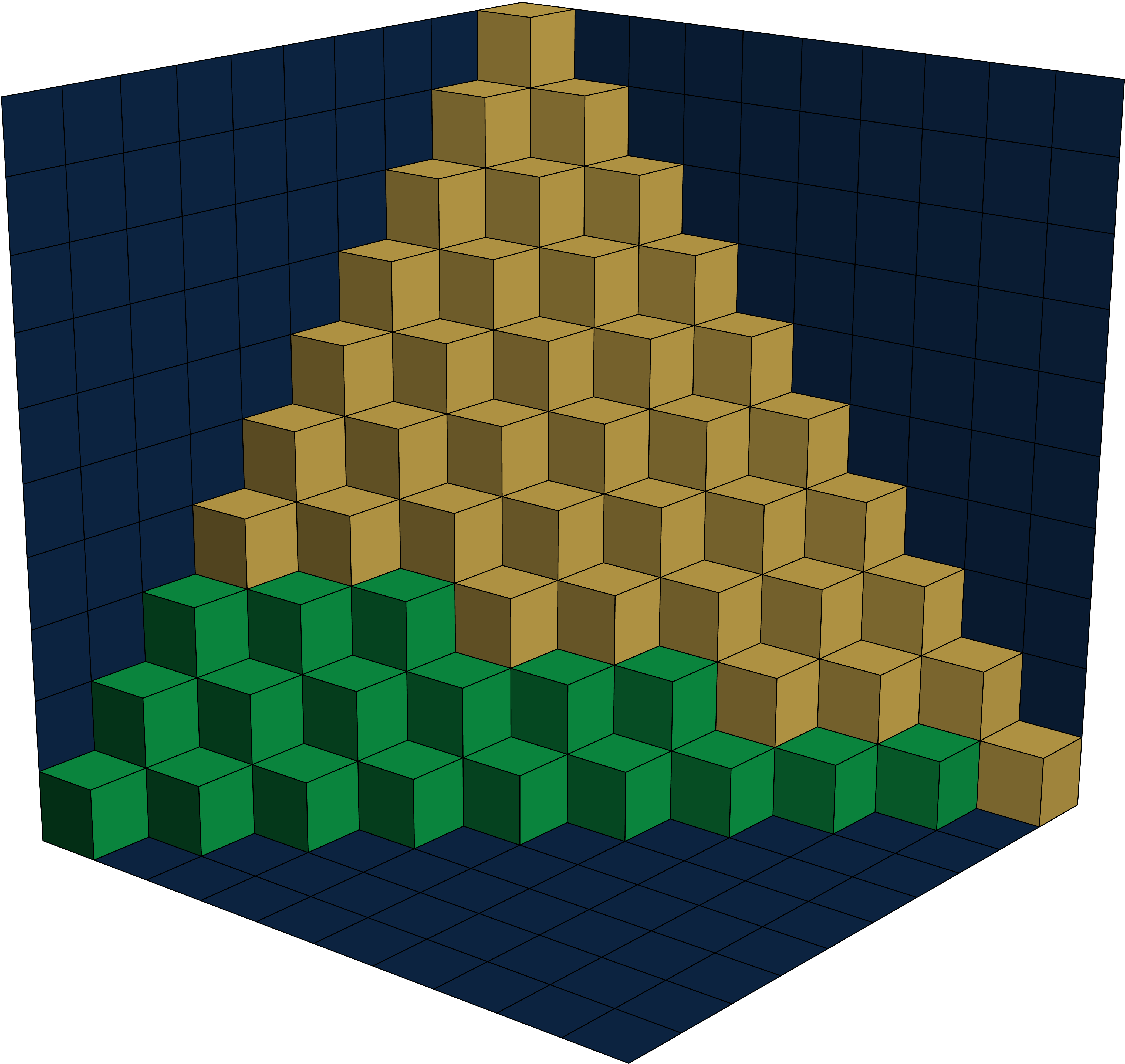}
		\caption{$\gin(I)_9$.}
		\label{fig:ci7-deg9}
	\end{subfigure}
	
	\caption{An example of Theorem 1.12 for a complete intersection $I$ equigenerated in degree $7$.}
	\label{fig:ci7-theorem112}
\end{figure}

\begin{thm}\label{IntroCIGinStructure}
	Let $I\subseteq R$ be a complete intersection of minimal type $v\in \N^{\ell+t+1}$.
	There exists an integer $D_{v,n} \in \N$ such that if $\alpha(I)\geq D_{v,n}$ then:
	\begin{enumerate}
		\item The set of minimal monomial generators of $\gin(I)$ in degree $\alpha(I)$ is
		\begin{align*}
			\left\lbrace x_1^{\alpha(I)}, x_1^{\alpha(I)-1}x_2 ,\dots,  x_1^{\alpha(I)-\gamma_I (\alpha(I))+1}x_2^{\gamma_I (\alpha(I))-1} \right\rbrace ,
		\end{align*}
		\item For $k\in \{1,2,\dots, t\}$,  the set of minimal monomial generators of $\gin(I)$ in degree $\alpha(I)+k$ is
		\begin{align*}
			\left\lbrace x_1^{\alpha(I)+k-q_{k-1}-1}x_2^{q_{k-1}+1}, x_1^{\alpha(I)+k-q_{k-1}-2}x_2^{q_{k-1}+2} ,\dots,  x_1^{\alpha(I)+k-q_k+1}x_2^{q_k-1} \right\rbrace ,
		\end{align*}
		where for all $j\in \{0,1,\dots , t\}$ we define $q_j =\sum_{i=0}^j \gamma_I(\alpha(I)+i)$.
	\end{enumerate}
\end{thm}

The following result is a generalization of Green's Crystallization Principle \cite{Green1998}.

\begin{thm}\label{IntroGeneralGreenNumeric}
	Suppose $\operatorname{char}(K)=0$.
	Let $I\subseteq R$ be an ideal generated in degrees $\leq d$ such that $\im(I)$ is strongly stable and has $q$ minimal generators in degree $d+1$.
	Then for all $k\geq 1$ we have
	\begin{align*}
		\dim I_{d+k} \leq q\binom{n-1+k-1}{n-1} + \dim R_k\im(I)_{d}.
	\end{align*}
\end{thm}

The case $q=0$ in Theorem \ref{IntroGeneralGreenNumeric} is Green's Crystallization Principle.

\subsection{Some History}
The EGH Conjecture was first stated for $\ell = n$;
Caviglia and Maclagan in \cite{CavigliaMaclagan} proved that this is equivalent to the $\ell\leq n$ version given above.
The case of the EGH Conjecture for almost complete intersections implies the Generalized Cayley--Bacharach Conjecture \cite{EGHCB},
also conjectured by Eisenbud, Green, and Harris.
Many authors have worked on the EGH Conjecture over the past $3$ decades \cite{Abedelfatah1, Abedelfatah2, Abedelfatah3, CavigliaConstantinescuVarbaro, CavigliaDeStefani2, CavigliaDeStefani, CavigliaMaclagan, Chen, Chong, Cooper, Francisco, Gasharov, GeramitaKreuzer, HoschsterSema, HerzogPopescu, Richert}.
The recent surveys \cite{GDSSurvey, SemaSurvey} go over the currently known cases of the EGH Conjecture,
and some of the  techniques that have been used to approach it.
The case $n=2$ was first settled by Richert \cite{Richert} using a result of Davis \cite{Davis1985}.

Theorems \ref{IntroWLPCITheorem} and \ref{IntroWLPCIPersistence} are new.
We note, however, that such properties have been defined before by Harima and Wachi \cite{HarimaWachi},
and are know as $k$-Lefschetz properties.
Palezzato and Torielli further studied k-Lefschetz properties in \cite{PalezzatoTorielli}.
For more details on the weak and strong Lefschetz properties see the survey by Migliore and Nagel \cite{JuanUweTour}.

The use of initial ideals and Gröbner bases goes back to $1900$ in the work of Gordan \cite{Gordan1900}.
Macaulay \cite{Macaulay} introduced total orderings on the monomials.
This work was extended by Gröbner \cite{Grobner1939, Grobner1950} and Buchberger \cite{Buchberger1965, Buchberger1970, Buchberger1976}.
Hironaka \cite{Hironaka1964} introduced ``standard bases'' which are analogous to Gröbner bases.
In the second half of the \nth{20} century there was a shift toward the use of generic initial ideals.
Hartshorne \cite{Hartshorne} used properties of generic initial ideals to prove connectedness of Hilbert schemes.
Grauert \cite{Grauert} considered generic initial ideals in the case of power series rings.
Galligo \cite{Galligo} developed the theory in characteristic $0$.
Bayer and Stillman extended the theory in arbitrary characteristic \cite{BayerStillmana}.


\subsection{Organization of the Paper}
Section \ref{DefinitionsAndNotation} covers some definitions and notation.
Section \ref{FSSection} covers a theorem of Fløystad and Stillman on generic initial ideals.
This section is crucial for the rest of the paper and it is used in several of the major results from the introduction.
In Section \ref{LefschetsIntro} we cover some known results about Lefschetz properties,
and prove some new ones.
We can prove Corollary \ref{WLPAsymptotic} with the methods of this paper instead of using theorems in the literature.
However, this would make the paper much longer and distract from the general method used.
In Section \ref{BezrukovSection} we adapt ideas of Bezrukov in order to express Hilbert functions of monomial ideals as weighted sums over the minimal generators.
This section is important for the rest of the paper.
We use the results and ideas in almost every major proof in the paper.
In Section \ref{GreenSection} we prove Theorem \ref{IntroGeneralGreenNumeric}.
In fact, we prove an even more general version of Green's Crystallization Theorem, see Theorem \ref{GeneralGreen}.

In Section \ref{ginStructure}, 
we use the results in Sections \ref{FSSection}, \ref{LefschetsIntro}, \ref{BezrukovSection}, and \ref{GreenSection} to prove Theorems \ref{IntroadvancedLefschetzBase}, \ref{IntroMonster2}, and \ref{IntroCIGinStructure}.
Some preliminary lemmas are also proved in order to do this.
After the structure of the generic initial ideal is established,
in Section \ref{TransformSection} we prove several results about how the monomials in the generic initial ideal can be partitioned and arranged. 
This section might be of interest by itself in the context of compression theory.
Compression is a standard technique in commutative algebra and combinatorics, see \cite{Engel, Harper, PeevaBook}.
The typical approach is to start with an ideal,
and slowly transform it to another nicer ideal (like a lex ideal) by a sequence of steps while doing small changes at each step.
In Section \ref{TransformSection},
we remove monomials from the generic initial ideal.
Although the remaining monomials do not form an ideal, 
we transform them by a sequence of compression operations into an ideal with several desirable properties.
See Theorem \ref{LexSegment} for more details.
On the other hand,
the removed monomials are used to form a monomial complete intersection,
see Lemma \ref{PowersSegment}.

In Section \ref{EGHResults} we prove Theorems \ref{IntroEGHFirst}, \ref{IntroEGHMain}, \ref{IntroEGHGotzmann}, \ref{IntroEGHGreenHyper}, \ref{IntroBettiEGH} by using the results in Sections \ref{BezrukovSection}, \ref{ginStructure}, and \ref{TransformSection}.
In Section \ref{TranslationSection} we show how to prove Theorem \ref{BFGG} from Theorem \ref{IntroEGHGotzmann}.
In Section \ref{LefschetzApplicationsSection} we prove Theorems \ref{IntroWLPCITheorem} and \ref{IntroWLPCIPersistence} by using results from Sections \ref{FSSection}, \ref{LefschetsIntro}, \ref{GreenSection}, and \ref{TransformSection}.

In Section \ref{ConjecturesSection} we state a few conjectures.
Finally, Section \ref{Thanks} is dedicated to acknowledgments.


\section{Background definitions and notation}\label{DefinitionsAndNotation}
{\it Monomials} are products of variables.
A polynomial in $R$ is called {\it homogeneous} if all the monomials in it have the same degree,
and the degree of a homogeneous polynomial is the degree of any monomial in it.
A {\it form} is a homogeneous polynomial,
and a {\it linear form} is a form of degree $1$. 
An ideal (vector space closed under multiplication by the variables) $I\subseteq R$ is {\it homogeneous/graded} if it is generated by homogeneous polynomials.
The Hilbert function of a graded ideal $I$ takes a natural numbers $d\in \N$ as input and outputs the dimension of the vector space of homogeneous polynomials of degree $d$ in $I$.
Homogeneous polynomials $f_1,\dots, f_\ell$ with $\ell\leq n$ form a {\it regular sequence/complete intersection},
if the ideal generated by $f_1,\dots, f_\ell$ has the same Hilbert function as the ideal generated by $x_1^{\deg f_1},\dots, x_\ell^{\deg{f_\ell}}$.

For two monomials of the same degree $m_1=x_1^{p_1}\cdots x_n^{p_n}$ and $m_2=x_1^{q_1}\cdots x_n^{q_n}$,
we say that $m_1>m_2$ in the {\it lexicographic (lex) order} if the first nonzero entry of $(p_1-q_1,\dots, p_n-q_n)$ is positive.
Similarly, $m_1>m_2$ in the {\it reverse lexicographic (revlex) order} if the last nonzero entry of $(p_1-q_1,\dots, p_n-q_n)$ is negative.
A monomial ideal (generated by monomials) is called a {\it lex ideal} if whenever a monomial is inside our ideal,
then any monomial larger than it is also inside of the ideal.
A set of monomials of the same degree is called a {\it lex segment}  if whenever a monomial is inside the set,
then any monomial larger than it is also inside of the set.

Every homogeneous polynomial $f\in R$ of degree $d$ is expressed as a unique linear combination of the monomials of degree $d$.
The largest monomial in the revlex order that has a nonzero coefficient in this linear combination is called the {\it initial term/monomial} of $f$,
and is denoted by $\im(f)$.
The initial ideal of an ideal $I$ is the ideal generated by all the initial monomials of elements in $I$,
and is denoted by $\im(I)$.
Due to Galligo \cite{Galligo},
for every ideal $I$ there exists a set of ring automorphisms $U$,
such that for $\sigma,\tau \in U$ we have that $\im(\sigma(I)) = \im(\tau (I))$.
We define the generic initial ideal of $I$ to be $\im(\sigma(I))$ and denote it by $\gin(I)$.

For a homogeneous ideal $I$, 
a linear form $f$, 
and $d\in \N$,
we say that $f$ is a \textit{weak Lefschetz element (WLE) on $A=R/I$ or $I$ in degree $d$},
if the multiplication map by $f$, 
$\times f_d: A_d \rightarrow A_{d+1}$, is injective or surjective.
We say that {\it $f$ is a WLE on $A$ or $I$} if $f$ is a WLE on $A$ in degree $d$ for all $d\in \N$.
We say that $A$ or $I$ satisfies the \textit{weak Lefschetz property (WLP)} if there exists a linear form $f$ that is a WLE on $A$.
The strong Lefschetz property is defined similarly by requiring injectivity or surjectivity of $f^k$.

For $i< j\leq n$ we define the {\it $(i,j)$-Borel move} to be the function that takes as input a monomial $m$ divisible by $x_j$ and outputs $x_im/x_j$.
A set of monomials is called {\it Borel or strongly stable} if it is closed under Borel moves.
A Borel set is sometimes called a {\it $2$-compressed set} in the literature.


\section{Fløystad--Stillman Spaces}\label{FSSection}
In this section we introduce results obtained by Fløystad and Stillman in \cite{FloystadStillman}.
Let $V$ be a vector space of homogeneous polynomials in $K[x_1,\dots, x_n]$. 
Following \cite{FloystadStillman} we define $V|_{x_n,\dots, x_{n-k+1}}$ to be the image of $V$ in $K[x_1,\dots, x_{n-k}]$ under the natural map
\begin{align*}
	K[x_1,\dots, x_n] \rightarrow K[x_1,\dots, x_{n-k}].
\end{align*}

Two standard facts about the revlex order are:
\begin{enumerate}
	\item $\im(V|_{x_n}) = \im(V)|_{x_n}$,
	\item $\im(V:x_n) = \im(V) : x_n$.
\end{enumerate}

We denote by $I_V$ the ideal generated by $V$ in $K[x_1,\dots, x_n]$.
Suppose $V\subseteq K[x_1,\dots, x_n]_d$ for some $d\in \N$.
There exists a Zariski open set $U$ of ring automorphisms,
such that for $f,g\in U$ we have $\im(f(V)) = \im(g(V))$.
Take $f\in U$ and define $W=f(V)$. 
If $k\in \{0,1,\dots, n-2\}$ then we define the \textit{$k$-th Fløystad--Stillman space} of $V$ to be
\begin{align*}
	\squash_k(V) = \left( \bigoplus_{r=1}^{d}((W|_{x_n,\dots, x_{n-k+1}}):x_{n-k}^r)|_{x_{n-k}}\right) \oplus (I_{W|_{x_n,\dots, x_{n-k+1}}})|_{x_{n-k}} \subseteq K[x_1,\dots, x_{n-k-1}].
\end{align*}
In particular,
\begin{align*}
	\squash_0(V) = \left( \bigoplus_{r=1}^{d}(W:x_{n}^r)|_{x_{n}}\right) \oplus (I_{W})|_{x_{n}} \subseteq K[x_1,\dots, x_{n-1}].
\end{align*}

See Figure \ref{fig:fs-example} for an example of $\im(V)$ and $\im(\squash_k(V))$.

\begin{figure}[htbp]
	\centering
	
	\begin{subfigure}[t]{0.48\textwidth}
		\centering
		\includegraphics[
		width=\linewidth
		]{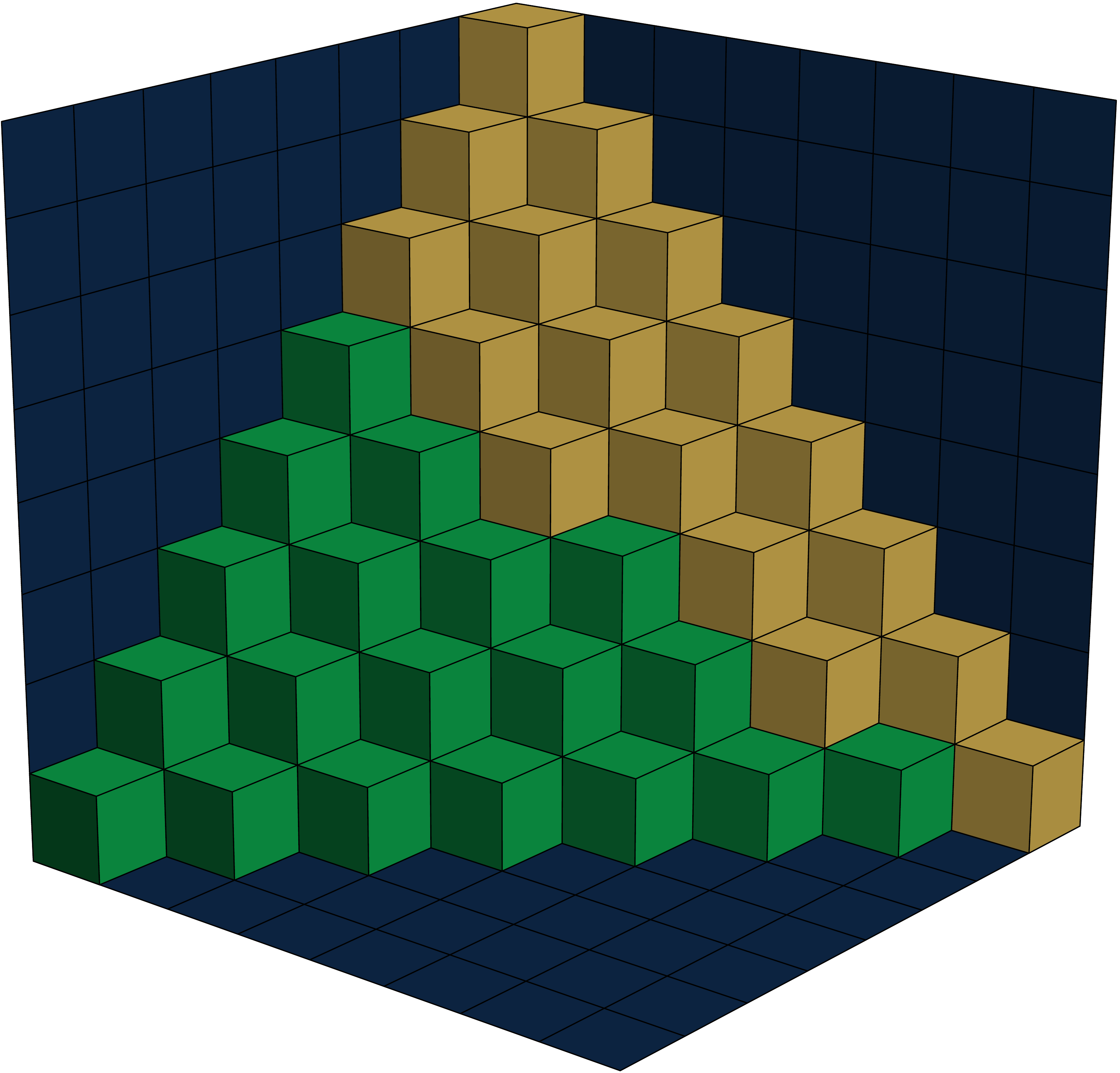}
		\caption{The monomials in $\operatorname{in}(V)$.}
		\label{fig:fs-in-v}
	\end{subfigure}
	\hfill
	\begin{subfigure}[t]{0.48\textwidth}
		\centering
		\includegraphics[
		width=\linewidth
		]{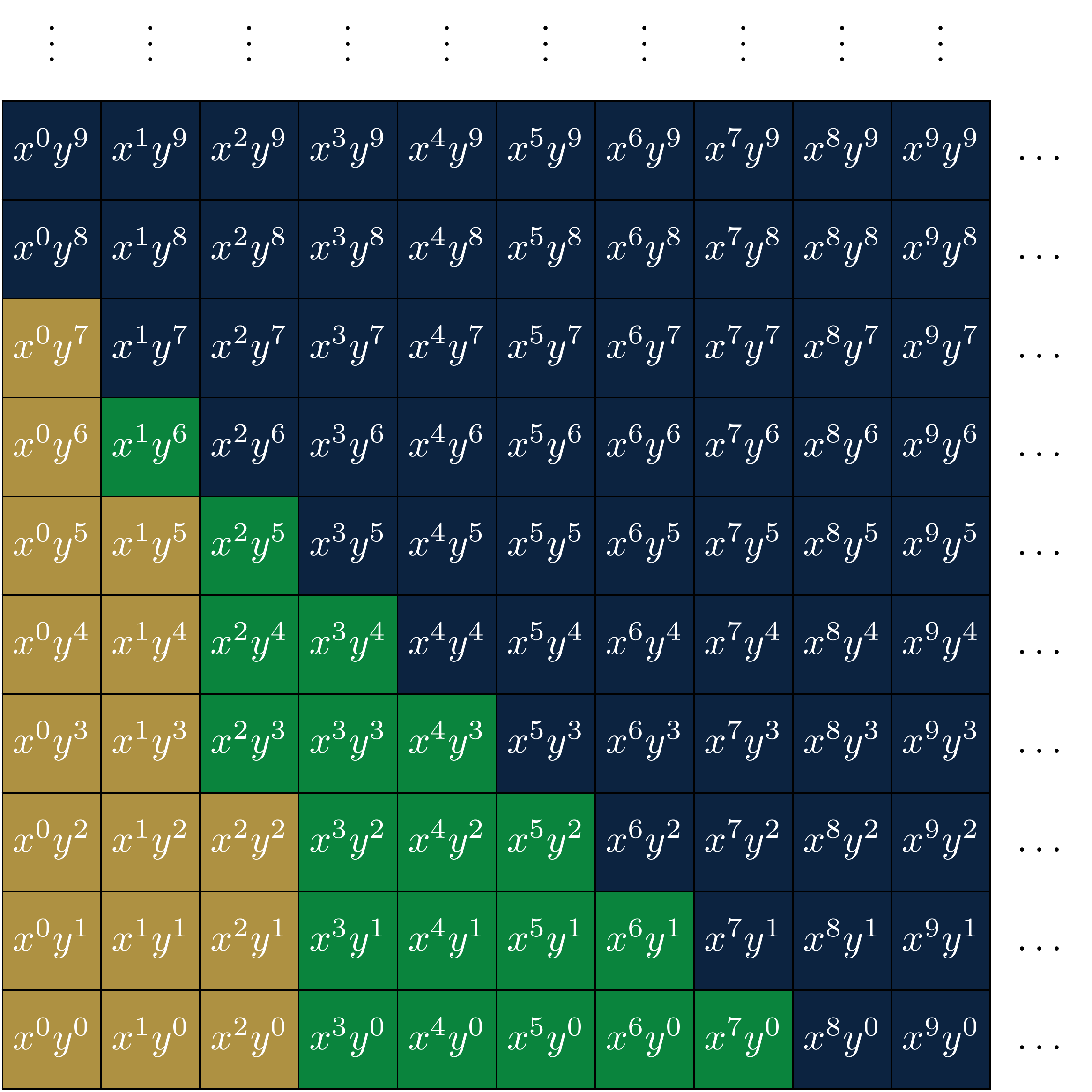}
		\caption{The monomials in $\operatorname{in}(FS_0(V))$ in degrees $\leq 7$.}
		\label{fig:fs-zero-v}
	\end{subfigure}
	
	\caption{Example of $\operatorname{in}(V)\subseteq K[x_1,x_2,x_3]$ and $\operatorname{in}(FS_0(V))\subseteq K[x,y]$.}
	\label{fig:fs-example}
\end{figure}

\begin{thm}[Fløystad--Stillman \cite{FloystadStillman}]\label{Fløystad--Stillman}
	If $V\subseteq R_d$ and $k\in \{0,1,\dots, n-2\}$ then $\squash_k(V)$ is an ideal and $\im(\squash_k(V))$ is strongly stable.
\end{thm}

Theorem \ref{Fløystad--Stillman} is proved while proving Main Theorem 2.4 in \cite{FloystadStillman}.
The authors in \cite{FloystadStillman} (during the proof of Main Theorem 2.4) prove the conclusions of Theorem \ref{Fløystad--Stillman} first,
and afterwards combine these results with Green's Crystallization Principle \cite{Green1998} to complete the proof of Main Theorem 2.4.

\section{Lefschetz Properties and Generic Initial Ideals}\label{LefschetsIntro}
The results in this section tell us that we can transfer information about the WLP from an ideal to its generic initial ideal.
This allows us to prove theorems about the WLP later,
but also gives us bounds on the monomials in the generic initial ideal after evaluating $x_n=0$.
These bounds are used in some of the proofs in Section \ref{ginStructure}.

\begin{prp}[Migliore--Nagel \cite{JuanUweTour}]\label{WLPequiv}
	Consider an ideal $I\subseteq R $, 
	a general linear form $\ell$,
	and $d\in \N$.
	Then $\ell$ is a WLE on $R/I$ in degree $d$ iff $\dim (R/(I,\ell))_{d+1} \leq \max \{0, \dim (R/I)_{d+1} - \dim(R/I)_{d}\}$.
\end{prp}

Lemma \ref{Conca} is a special case of a more general result of Conca \cite{Conca} concerning initial ideals.

\begin{lem}[Conca \cite{Conca}]\label{Conca}
	Suppose that we have a graded ideal $I\subseteq R$ and an ideal $J$ generated by $p$ general linear forms,
	where $0\leq p \leq n$.
	Then for all $d\in \N$ we have $\dim (R/(I+J))_d  = \dim (R/(\gin(I) + H))_d$,
	where $H$ is the ideal generated by $x_{n-p+1},x_{n-p+2},\dots, x_n$.
\end{lem}

Combining Proposition \ref{WLPequiv} and Lemma \ref{Conca} we obtain Corollary \ref{Wiebe}.
Corollary \ref{Wiebe} was discovered by Wiebe \cite{Wiebe}.
In fact, Wiebe uses more results from Conca \cite{Conca} to prove similar claims about the weak and strong Lefschetz properties and initial ideals under arbitrary monomial orders.

\begin{cor}[Wiebe \cite{Wiebe}]\label{Wiebe}
	Consider a graded ideal $I\subseteq R = K[x_1,\dots, x_n]$, 
	a general linear form $\ell$,
	and $d\in \N$.
	Then $\ell$ is a WLE on $R/I$ in degree $d$ iff $x_n$ is a WLE on $R/\gin(I)$ in degree $d$. 
\end{cor}

Lemma \ref{WLPNewGenerators} tells us that if $x_n$ is a weak Lefschetz element on $K[x_1,\dots, x_n]/I$ for a monomial ideal $I$,
then there are restrictions on the minimal generators that $I$ can have.

\begin{lem}\label{WLPNewGenerators}
	Let $I\subseteq R$ be a monomial ideal, and let $x_n$ be a WLE on $R/I$ in degree $d$.
	Suppose that there is some monomial of degree $d+1$ that is not divisible by $x_n$ and is not in $I_{d+1}$.
	Then all minimal generators in degree $d+1$ must not be divisible by $x_n$.
\end{lem}
\begin{proof}
	Assume to the contrary that there is some monomial $m\in I_{d+1}$ that is a minimal generator in degree $d+1$ and is divisible by $x_n$.
	Then multiplication by $x_n$ is not injective because $m/x_n\not\in I_d$ and $m\in I_{d+1}$.
	However, multiplication by $x_n$ is not surjective because there is some monomial not in $I_{d+1}$ that is not divisible by $x_n$.
	Thus, we have a contradiction with the assumption that $x_n$ is a weak Lefschetz element on $R/I$.
\end{proof}

If we know that some ideal satisfies the weak Lefschetz property then we can use Corollary \ref{Wiebe} and Lemma \ref{WLPNewGenerators} to gain information about its generic initial ideal.

\begin{thm}[Harima--Migliore--Nagel--Watanabe \cite{JuanWLPHeight3}]\label{JuanWLPHeight3}
	Every complete intersection in $K[x,y,z]$ has the weak Lefschetz property.
\end{thm}

\begin{thm}[Boij--Migliore--Miró-Roig--Nagel \cite{JuanWLPHeight4}]\label{JuanWLPHeight4}
	Suppose that $I\subseteq R=K[x_1,x_2,x_3,x_4]$ is generated by a regular sequence $f_1,f_2,f_3,f_4$ such that $\deg f_i=d$ for all $i$.
	Then the WLP holds for $R/I$ in degrees $< \floor{\frac{3d-3}{2}}$. 
\end{thm}

\begin{thm}[Beorchia--Miró-Roig \cite{WLPHeigthN}]\label{WLPHeigthN}
	Suppose that $I\subseteq R$ is generated by a complete intersection $f_1,\dots, f_n$ such that $\deg f_i=d$ for all $i$.
	Then the WLP holds for $R/I$ in degrees $< d + \ceil{\frac{d}{n}}$. 
\end{thm}

If the regular sequence does not have full length then we know that the WLP holds.

\begin{prp}\label{JuanWLP}
	If $I\subseteq R$ is generated by a complete intersection $f_1,\dots, f_\ell$ with $\ell<n$,
	then $R/I$ satisfies the WLP.
\end{prp}
\begin{proof}
	The quotient $R/I$ is Cohen-Macaulay of positive dimension, 
	so the depth is at least $1$. 
	Hence, $R/I$ has a nonzero divisor F. This means that multiplication by $F$ is injective. 
	Then Proposition 1.5.12 in \cite{BrunsHerzog} says that we can take $F$ to have degree $1$.
\end{proof}

The preceding results on the WLP for complete intersections are partial,
but if we assume that the smallest degree of the regular sequence is large then we know that the WLP will hold in a large number of degrees.
This is captured in Corollary \ref{WLPAsymptotic}.

\begin{cor}\label{WLPAsymptotic}
	Suppose that $I\subseteq R$ is equigenerated by a regular sequence $f_1,\dots, f_r$.
	Let $\ell$ be a general linear form.
	For every $t\in \N$ there exists $D_{t,n}\in \N$ such that if $\alpha(I)\geq D_{t,n}$,
	then $\ell$ is a WLE on $R/I$ in degrees $\leq \alpha(I)+t$.
	Furthermore, we can take $D_{t,n} = nt+1$.
\end{cor}
\begin{proof}
	Theorem \ref{JuanWLPHeight3}, Theorem \ref{JuanWLPHeight4}, Theorem \ref{WLPHeigthN}, and Proposition \ref{JuanWLP} each give a different bound.
	Using Theorem \ref{WLPHeigthN} we can get $D_{t,n} = nt+1$.
\end{proof}

 \begin{lem}\label{ginBound}
 	Let $t\in \N$.
 	There exists $D_{t,n}\in \N$ such that if $d\geq D_{t,n}$ then for any ideal $I$ generated by a regular sequence $f_1,\dots, f_r$ with $\deg(f_i) = d$ for all $i\in \{1,\dots, r\}$,
 	whenever $k\leq t$ we have
 	\begin{align*}
 		\dim (\gin(I)|_{x_n})_{d+k} = r\binom{n-1+k}{n-1} - r\binom{n-2+k}{n-1} = r \binom{n-2+k}{n-2}.
 	\end{align*}
 	Furthermore, we can take $D_{t,n} = nt+1$.
 \end{lem}
 \begin{proof}
 	Set $D_{t,n}'= nt+1$.
 	By Corollary \ref{WLPAsymptotic},
 	if $d\geq D_{t,n}'$ then $R/I$ satisfies the WLP in degrees $\leq d+t$.
 	Put $D_{t,n}'' = 2(t+1)$.
 	Let $D_{t,n} = \max\{D_{t,n}', D_{t,n}''\}$.
 	We can now use Lemma \ref{WLPNewGenerators} to obtain information about the minimal generators of $\gin(I)$.
 	
 	We prove the claim by induction on $k\leq t$.
 	If $k=0$ then $\dim (\gin(I)|_{x_n})_{d+k} = \dim (\gin(I)|_{x_n})_{d} = r = r\binom{n-1}{n-1} - r\binom{n-2}{n-1} =r\binom{n-2}{n-2}$ by Pascal's identity.
 	So, suppose that $k\geq 1$ and that the claim holds for all $k'<k$.
 	Let $M_i$ be the set of all monomials in $\gin(I)_{d+k}$ that are divisible by $x_n^i$ and are not divisible by $x_n^{i+1}$.
 	Since $D_{t,n}'' = 2(t+1)$ we have
 	\begin{align*}
 		\sum_{i=0}^k |M_i| = \dim \gin(I)_{d+k} = r\binom{n-1+k}{n-1}.
 	\end{align*}
 	Also, by the inductive hypothesis and Lemma \ref{WLPNewGenerators},
 	 for all $i\in \{1,\dots, k\}$ we have
 	\begin{align*}
 		|M_i| = r\binom{n-1+k-i}{n-1} - r\binom{n-1+k-i-1}{n-1}
 	\end{align*}
 	Therefore,
 	\begin{align*}
 		\dim (\gin(I)|_{x_n})_{d+k}  &= |M_0|\\
 		&= r\binom{n-1+k}{n-1} - \sum_{i=1}^k |M_i|\\
 		&= r\binom{n-1+k}{n-1} - \sum_{i=1}^k \left(r\binom{n-1+k-i}{n-1} -  r\binom{n-1+k-i-1}{n-1} \right)\\
 		&= r\binom{n-1+k}{n-1} - \left(r\binom{n-1+k-1}{n-1} -  r\binom{n-2}{n-1} \right)\\
 		&= r\binom{n-1+k}{n-1} - r\binom{n-2+k}{n-1}\\
 		&= r \binom{n-2+k}{n-2}.
 	\end{align*}
 	For the last part of the claim, clearly $D_{t,n} = D_{t,n}'$. 
 \end{proof}
 
 \section{The Bezrukov Function and Shadows}\label{BezrukovSection}
 The following definitions and results are based on Bezrukov's paper \cite{Bezrukov}.
 Bezrukov does not state or prove the results of this section,
 but the ideas are very similar and based on his work.
 
 \begin{lem}\label{RemovalLemma}
 	Let $S$ be a nonempty set of monomials of the same degree.
 	If $S$ is strongly stable, then there exists $m\in S$ such that $S\setminus \{m\}$ is strongly stable.
 \end{lem}
 \begin{proof}
 	Let $m\in S$ be the earliest monomial in the lexicographic order that is in $S$.
 	Suppose that $m'\in S'=S\setminus \{m\}$,
 	and we show that if $i<j$ such that $x_j|m'$ then $x_im'/x_j\in S'$.
 	If $x_im'/x_j \neq m$ then $x_im'/x_j\in S'$ because the strongly stable property of $S$ implies $x_im'/x_j\in S$.
 	If $x_im'/x_j = m$ then $m'$ is earlier in the lexicographic order than $m$,
 	but this is a contradiction with the definition of $m$ since $m'\in S'\subseteq S$.
 \end{proof}
 
 For a monomial $m\in K[x_1,\dots, x_n]$ of degree at least $1$ we define
\begin{align*}
	\max(m) &= \max \{i\in \{1,\dots, n\} \bigm | x_i | m\}\\
	\min(m) &= \min \{i\in \{1,\dots, n\} \bigm | x_i | m\}.
\end{align*}

For an integer $t\geq 0$ we define the \textit{Bezrukov function} $\Bez_n^t$ on the set of all monomials of a polynomial ring $K[x_1,\dots, x_n]$,
such that for a monomial $m$ we have $\Bez_n^t(m) = \binom{n-\max(m)+t}{n-\max(m)} $.
We will often omit the index $n$ when it is clear from context.
We also define $\Bez = \Bez^1$.
 
 Recall that for $S\subseteq K[x_1,\dots, x_n]$ we define the shadow of $S$ to be
 \begin{align*}
 	\uSdw(S) = \bigcup_{i=1}^n \{x_if \bigm | f\in S\}.
 \end{align*}
 
We also inductively define $\uSdw^t(S)$ to be $\uSdw(\uSdw^{t-1}(S))$,
and we define $\uSdw^0(S)=S$.
 
\begin{lem}\label{AddedShadowEquiv}
	Suppose that $S$ is a strongly stable set of monomials of degree $d$.
	If $m\not \in S$ is a monomial of degree $d$ such that $S\cup \{m\}$ is strongly stable, 
	then for any integer $t\geq 1$ and a monomial $m'$ of degree $t$ we have that $m'm\in \uSdw^t(m)\setminus \uSdw^t(S)$ iff $\min(m')\geq \max(m)$.
\end{lem}
\begin{proof}
We prove the forward direction first.
Suppose that $m'm\in \uSdw^t(m)\setminus \uSdw^t(S)$.
Assume to the contrary that $\min(m') < \max(m)$.
Let $j = \max(m)$. 
Since $\min(m') < j$, 
there exists some variable $x_i$ dividing $m'$ such that $i < j$. 
Since $S \cup \{m\}$ is strongly stable, 
$x_j$ divides $m$, and $i < j$, the monomial $u = (x_i/x_j)m$ must belong to $S \cup \{m\}$. 
Note that $u \neq m$ because they contain different variables; therefore, $u \in S$.
We can rewrite the product $m'm$ as:
\[
m'm = \left( \frac{m'}{x_i} \cdot x_i \right) \left( \frac{m}{x_j} \cdot x_j \right) 
= \left( \frac{m'}{x_i} \cdot x_j \right) \left( \frac{m}{x_j} \cdot x_i \right).
\]
Since $u$ divides $m'm$ and $u \in S$, 
it follows that $m'm \in \uSdw^t(S)$. 
This contradicts the hypothesis that $m'm \in \uSdw^t(m) \setminus \uSdw^t(S)$. 
Thus, we must have $\min(m') \geq \max(m)$.

We prove the backwards direction next.
Assume that $\min(m') \geq \max(m)$. It is clear that $m'm \in \uSdw^t(m)$. We must show $m'm \notin \uSdw^t(S)$.
Suppose, for the sake of contradiction, that $m'm \in \uSdw^t(S)$. 
Then there exists $u\in S$ that divides $m'm$.
Note that $u\neq m$ because $m\not\in S$.
Since $\max(m)\leq \min(m')$ we have that $m$ can be obtained from $u$ by a sequence of Borel moves.
However, this contradicts the assumption that $S$ is strongly stable.
\end{proof}
 
 \begin{lem}\label{WeightsLemma}
 	Let $S$ be a set of monomials of the same degree.
 	If $S$ is strongly stable and $t\geq 0$ then
 	\begin{align*}
 		|\uSdw^t(S)| = \sum_{m\in S} \Bez^t(m).
 	\end{align*}
 \end{lem}
 \begin{proof}
 	We assume that $t\geq 1$ as the case $t=0$ is trivial.
 	We prove the claim by induction on $k=|S|$.
 	The case $k=0$ is trivial.
 	So, suppose that $k\geq 1$ and that the claim holds for all $k'<k$.
 	By Lemma \ref{RemovalLemma} there is some $m^\ast\in S$ such that $S'=S\setminus \{m^\ast\}$ is strongly stable.
 	By the inductive hypothesis
 	\begin{align*}
 		|\uSdw^t(S')| = \sum_{m\in S'} \Bez^t(m).
 	\end{align*}
 	We have that $\uSdw^t(S) = \uSdw^t(S') \cup (\uSdw^t(m^\ast)\setminus \uSdw^t(S'))$,
 	and because the union is disjoint we get $|\uSdw^t(S)| = |\uSdw^t(S')| + |\uSdw^t(m^\ast)\setminus \uSdw^t(S')|$.
 	So, we need to show that $|\uSdw^t(m^\ast)\setminus \uSdw^t(S')| =  \Bez^t(m^\ast) $.
 	By Lemma \ref{AddedShadowEquiv} we have $|\uSdw^t(m^\ast)\setminus \uSdw^t(S')| = \binom{n-\max(m^\ast)+t}{n-\max(m^\ast)}$.
 	Therefore, $|\uSdw^t(m^\ast)\setminus \uSdw^t(S')| =\Bez^t(m^\ast)$.
 \end{proof}
 
 \begin{lem}\label{WeightsLemmaAdvanced}
 	Let $S_1,\dots S_k$ be sets of monomials in $R$, 
 	such that for each $i\in \{1,\dots, k\}$ all the monomials in $S_i$ have degree $d_i$,
 	and $d_1<d_2<\cdots <d_k$.
 	Also, let $d\in \N$ with $d\geq d_k$ and assume that for all $i,j\in \{1,\dots, k\}$ with $i<j$ we have $\uSdw^{d_j-d_i}(S_i)\cap S_j =\emptyset$.
 	Suppose that for all $i\in \{1,\dots, k\}$ we have that $\uSdw^{d_i-d_1}(S_1)\cup \cdots \cup \uSdw^{d_i-d_{i-1}}(S_{i-1})\cup S_i$ is strongly stable.
 	Then
 	\begin{align*}
 		|\uSdw^{d-d_1}(S_1)\cup \cdots \cup \uSdw^{d-d_{k}}(S_{k})| = \sum_{i=1}^k \sum_{m\in S_i} \Bez^{d-d_i}(m)
 	\end{align*}
 \end{lem}
 \begin{proof}
 	We prove the claim by induction on $k$.
 	If $k=1$ then the claim holds by Lemma \ref{WeightsLemma}.
 	So suppose that $k\geq 2$ and that the claim holds for all $k'<k$.
 	We now do a second induction on $p=|S_k|$.
 	If $p=0$ then we are done by the inductive hypothesis for $k$.
 	So suppose that $p\geq 1$ and that the claim holds for all $p'<p$.
 	
 	Let $S=\uSdw^{d_k-d_1}(S_1)\cup \cdots \cup \uSdw^{d_k-d_{k-1}}(S_{k-1})\cup S_k$.
 	First, we show that there exists $m^\ast\in S_k$ such that $S\setminus \{m^\ast\}$ is strongly stable.
 	Assume to the contrary that this is not the case.
 	Let $m^\ast$ be the earliest monomial in the lexicographic order that is inside of $S_k$.
 	Then $S\setminus \{m^\ast\}$ is not strongly stable.
 	Thus, there exists $m'\in S\setminus \{m^\ast\}$ and $i,j\in \{1,\dots, n\}$ with $i<j$ such that $x_im'/x_j = m^\ast$.
 	If $m'\in S\setminus S_k$ then we have a contradiction because $S\setminus S_k$ is strongly stable and $m^\ast\in S_k$.
 	So, we must have $m'\in S_k$.
 	However, now we have $m^\ast,m '\in S_k$ and $m'$ comes before $m^\ast$ in the lexicographic order,
 	a contradiction with the minimality of $m^\ast$.
 	Therefore, there exists $m^\ast\in S_k$ such that $S\setminus \{m^\ast\}$ is strongly stable.
 	
 	We have
 	\begin{align*}
 		\uSdw^{d-d_1}(S_1)\cup \cdots \cup \uSdw^{d-d_{k}}(S_{k})= \uSdw^{d-d_1}(S_1)\cup \cdots \cup \uSdw^{d-d_{k}}(S_{k}\setminus \{m^\ast\}) \cup \uSdw^{d-d_k}(m^\ast).
 	\end{align*}
 	Define $S_1'=S_1,\dots, S_{k-1}'=S_{k-1}, S_k' = S_k\setminus \{m^\ast\}$.
 	Then by the inductive hypothesis we have
 	\begin{align*}
 		|\uSdw^{d-d_1}(S_1')\cup \cdots \cup \uSdw^{d-d_{k}}(S_{k}')| = \sum_{i=1}^k \sum_{m\in S_i'} \Bez^{d-d_i}(m)
 	\end{align*}
 	Also, if $d>d_k$ then by Lemma \ref{AddedShadowEquiv} we have
 	\begin{align*}
 		|\uSdw^{d-d_k}(m^\ast) \setminus \uSdw^{d-d_k}(S\setminus \{m^\ast\})| = \binom{n-\max(m^\ast)+d-d_k}{n-\max(m^\ast)} = \Bez_n^{d-d_k}(m^\ast).
 	\end{align*}
 	The same equality holds for the case $d=d_k$ by definition of $\uSdw^0$ and $\Bez^0$.
 	Therefore,
 	\begin{align*}
 		|\uSdw^{d-d_1}(S_1)\cup \cdots \cup \uSdw^{d-d_{k}}(S_{k})| 
 		&=|\uSdw^{d-d_k}(S)|\\
 		&=|\uSdw^{d-d_k}(S\setminus \{m^\ast\}) \cup (\uSdw^{d-d_k}(m^\ast)\setminus \uSdw^{d-d_k}(S\setminus \{m^\ast\}))|\\
 		&=|\uSdw^{d-d_k}(S\setminus \{m^\ast\})|+|(\uSdw^{d-d_k}(m^\ast)\setminus \uSdw^{d-d_k}(S\setminus \{m^\ast\}))|\\
 		&= |\uSdw^{d-d_1}(S_1')\cup \cdots \cup \uSdw^{d-d_{k}}(S_{k}')| + |\uSdw^{d-d_k}(m^\ast) \setminus \uSdw^{d-d_k}(S\setminus \{m^\ast\}|\\
 		&= \left(\sum_{i=1}^k \sum_{m\in S_i'} \Bez^{d-d_i}(m)\right) + \Bez_n^{d-d_k}(m^\ast)\\
 		&= \sum_{i=1}^k \sum_{m\in S_i} \Bez^{d-d_i}(m).
 	\end{align*}
 \end{proof}
 
 \begin{cor}\label{WeightsLemmaAdvancedIdeal}
 	If $M\subseteq R$ is a strongly stable monomial ideal and $d\in \N$, then
 	\begin{align*}
 		\dim M_d =  \sum_{\substack{\text{monomial }m\in M\\ m \text{ minimal generator}\\ \deg(m) \leq d}} \Bez^{d-\deg(m)}(m)
 	\end{align*}
 \end{cor}
\begin{proof}
	Follows from Lemma \ref{WeightsLemmaAdvanced}.
\end{proof}
 
\section{A Generalization of Green's Crystallization Principle}\label{GreenSection}
We prove a generalization of Green's Crystallization Principle, Proposition 2.28 in \cite{Green1998}.

We can write any homogeneous polynomial $f\in R$ of degree $d$ as a unique linear combination of the monomials of degree $d$.
Define the \textit{coefficient of $m$ in $f$} to be the element $a\in K$ that multiplies $m$ in the unique linear combination of monomials that equal $f$. 
For a vector space $V\subseteq K[x_1,\dots, x_n]$ of homogeneous polynomials of degree $d$ and a set $B\subseteq V$,
we say that $B$ is a \textit{reduced basis} for $V$ if $B$ is a basis for $V$, 
and for each $f\in B$ there exists $m\in \im(V)$ such that $\im(f) = m$,
the coefficient of any monomial $m'$ in $f$ is $0$ whenever $m'$ is larger in the revlex order than $m$, 
and the coefficient of $m$ in any $f'\in B\setminus\{f\}$ is $0$.

Lemma \ref{reducedBasisVectorSpace} is mentioned in \cite{Green1998} without proof.
A generalization to a larger class of rings is proved in \cite{KuzmanovskiMacaulayPosetsAndRings}.
Many authors have most likely rediscovered it independently.

\begin{lem}\label{reducedBasisVectorSpace}
	Let $K$ be a field of arbitrary characteristic.
	If $V\subseteq R$ is a vector space of homogeneous polynomials of degree $d$,
	then there exists a reduced basis of $V$ in degree $d$.
\end{lem}

For a homogeneous polynomial $f\in K[x_1,\dots, x_n]$ we define
\begin{align*}
	\max(f) &= \max (\im(f))\\
	\min(f) &= \min ( \im(f)).
\end{align*}
For a set of homogeneous polynomials $S\subseteq R$ we define the \textit{Borel shadow} of $S$ by
\begin{align*}
	\uSdwB(S) = \bigcup_{i=1}^n \{x_ig\bigm | g\in S \text{ and } \max(g) \leq i\}.
\end{align*}

\begin{lem}\label{BorelShadowStructure}
	Let $K$ be a field of arbitrary characteristic.
	If $S\subseteq R$ is a set of homogeneous polynomials then for any $k\geq 1$ we have 
	\begin{align*}
		\uSdwB^k(S) = \bigcup_{\substack{q \text{ is a monomial}\\ \deg(q)=k}} \{qg\bigm | g\in S \text{ and } \max(g) \leq \min(q)\}.
	\end{align*}
\end{lem}
\begin{proof}
	We prove the claim by induction on $k$.
	The case $k=1$ follows from the definition of Borel shadow and because $\min(x_i) = i$ for all $i\in \{1,\dots, n\}$.
	So suppose that $k\geq 2$ and that the claim holds for all $k'<k$.
	By the inductive hypothesis we have
	\begin{align*}
		\uSdwB^{k-1}(S) = \bigcup_{\substack{q \text{ is a monomial}\\ \deg(q)=k-1}} \{qg\bigm | g\in S \text{ and } \max(g) \leq \min(q)\}.
	\end{align*}
	First, we show the $\subseteq$ relation.
	Let $f\in \uSdwB^k(S)$.
	Since $\uSdwB^k(S) = \uSdwB(\uSdwB^{k-1}(S))$ we have $f=x_imh$ for some $i\in \{1,\dots, n\}$, $h\in S$ and monomial $m$ of degree $k-1$,
	such that $\max(mh) \leq i$ and $\max(h) \leq \min(m)$.
	Thus, $\min(m) \leq \max(m) \leq \max(mh) \leq i$.
	Hence, $\min(m) = \min(x_im)$ because $\min(m) \leq i$.
	So, $\max(h) \leq \min(m) = \min(x_im)$.
	Therefore, $f\in \{x_img \bigm | g\in S \text{ and } \max(g) \leq \min(x_im)\}$,
	and the $\subseteq$ relation is proved.
	
	Next, we prove the $\supseteq$ relation.
	Let
	\begin{align*}
		f\in \bigcup_{\substack{q \text{ is a monomial}\\ \deg(q)=k}} \{qg\bigm | g\in S \text{ and } \max(g) \leq \min(q)\}.
	\end{align*}
	Then we have $f=mh$ for some monomial $m$ of degree $k$ and some $h\in S$ such that $\max(h) \leq \min(m)$.
	Let $i=\max(m)$ and set $m'=m/x_i$.
	Hence, $\max(h) \leq \min(m) = \min(m')$.
	Thus, $m'h\in \{m'g\bigm | g\in S \text{ and } \max(g) \leq \min(m')\}$.
	So, $m'h\in \uSdwB^{k-1}(S)$.
	Also, $\max(m'h) \leq i$.
	This implies that $mh=x_im'h\in \{x_ig\bigm | g\in \uSdwB^{k-1}(S) \text{ and } \max(g) \leq i\}$.
	Therefore, $mh\in \uSdwB(\uSdwB^{k-1}(S))= \uSdwB^k(S)$.
\end{proof}

\begin{lem}\label{BorelShadowUnique}
	Let $K$ be a field of arbitrary characteristic.
	Suppose that $S\subseteq R$ is a set of homogeneous polynomials of degree $d$,
	such that any two distinct polynomials in $S$ have a different initial monomial.
	If $m_1$ and $m_2$ are two monomials of degree $k$ such that $m_1<m_2$ in the revlex order,
	and $f_1,f_2\in S$ such that $\max(f_1)\leq \min(m_1)$ and $\max(f_2)\leq \min(m_2)$,
	then $\im(m_1f_1)< \im(m_2f_2)$.
	Furthermore:
	\begin{enumerate}
		\item One has $m_1f_1\neq m_2f_2$.
		\item One has
		\begin{align*}
			\{m_1g\bigm | g\in S \text{ and } \max(g) \leq \min(m_1)\}\cap \{m_2g\bigm | g\in S \text{ and } \max(g) \leq \min(m_2)\} = \emptyset.
		\end{align*}
		\item For any $k\geq 1$ we have
		\begin{align*}
			|\uSdwB^k(S)| = \sum_{\substack{q \text{ is a monomial}\\ \deg(q)=k}} |\{qg\bigm | g\in S \text{ and } \max(g) \leq \min(q)\}|.
		\end{align*}
	\end{enumerate}
\end{lem}
\begin{proof}
	We prove the first claim as the others follow immediately from it.
	We can write $m_1=x_1^{a_1}\cdots x_n^{a_n}$ and $m_2=x_1^{b_1}\cdots x_n^{b_n}$,
	and $\im(m_1f_1)=x_1^{a_1'}\cdots x_n^{a_n'}$ and $\im(m_2f_2)=x_1^{b_1'}\cdots x_n^{b_n'}$.
	Since $m_1<m_2$ in the revlex order,
	there exists $p\in \N$ such that $a_p>b_p$ and for all $i>p$ we have $a_i=b_i$.
	There exists $q<p$ such that $x_q|m_2$,
	because $m_1<m_2$ in the revlex order and $\deg(m_1)=k=\deg(m_2)$.
	Hence, $\max(f_2) \leq \min(m_2)\leq q<p$.
	Thus, for all $i\geq p$ we have $b_i = b_i'$.
	So, $a_p'\geq a_p>b_p=b_p'$ and for all $i>p$ we have $a_i' = a_i =b_i=b_i'$.
	Therefore, $\im(m_1f_1)< \im(m_2f_2)$.
	
	The next three claims follow because $\im(m_1f_1)< \im(m_2f_2)$ implies that $m_1f_1\neq m_2f_2$.
\end{proof}

\begin{cor}\label{BorelShadowLI}
	Let $K$ be a field of arbitrary characteristic.
	Suppose that $V\subseteq R$ is a vector space of homogeneous polynomials of degree $d$ with reduced basis $B$.
	Then $\uSdwB^k(B)$ is linearly independent. 
\end{cor}
\begin{proof}
	Any set of polynomials with pairwise distinct initial monomials is linearly independent.
	Therefore, the claim follows from Lemma \ref{BorelShadowUnique}.
\end{proof}

Consider a vector space $V\subseteq R$ of homogeneous polynomials of degree $d$ with reduced basis $B$.
Then by Corollary \ref{BorelShadowLI} we can find a set $E\subseteq \uSdw(B)\setminus \uSdwB(B)$ such that $E\cup \uSdwB(B)$ is a basis for the vector space spanned by $\uSdw(B)$.
We call $E\cup \uSdwB(B)$ a \textit{Borel shadow plus expanders basis} and call the elements of $E$ \textit{expanders}.

\begin{thm}\label{GeneralGreen}
	Let $K$ be a field of arbitrary characteristic.
	Suppose that $I\subseteq R$ is an ideal generated in degrees $\leq d$ and let $B$ be a reduced basis for $I_d$.
	Consider a Borel shadow plus expanders basis $E\cup \uSdwB(B)$ for $I_{d+1}$.
	If $k\geq 1$ then $\uSdw^{k-1}(E)\cup \uSdwB^{k}(B)$ is a spanning set for $I_{d+k}$. 
\end{thm}
\begin{proof}
	Well, $\uSdw^k(B)$ is a spanning set for $I_{d+k}$ because $I$ is generated in degrees $\leq d$.
	So, we are done if we show that $\uSdw^{k-1}(E)\cup \uSdwB^{k}(B)$ is a spanning set for $\Span(\uSdw^k(B))$.
	Let $h\in \uSdw^k(B)\setminus (\uSdw^{k-1}(E)\cup \uSdwB^{k}(B))$.
	Then $h\not\in \uSdw^{k-1}(E)$ and by Lemma \ref{BorelShadowStructure} there exist $f\in B$ and a monomial of degree $k$ such that $h=mf$ and $\max(f) > \min (m)$.
	Let $i\in \{1,\dots,n\}$ be such that $x_i|m$ and $x_if\not\in E\cup \uSdwB(B)$.
	Put $m' = m/x_i$.
	Then by the definition of Borel shadow plus expanders basis and Lemma \ref{BorelShadowStructure} we have
	\begin{align*}
		x_if = \left(\sum_{t\in B}\sum_{j=\max(t)}^n a_{t,j}x_jt\right) + \left(\sum_{\substack{\ell\in \{1,\dots, n\}\\ s\in B\\ \max(s)>\ell\\ x_{\ell} s\in E}} b_{\ell,s}x_{\ell} s\right),
	\end{align*}
	where $a_{t,j},b_{\ell,s}\in K$. 
	Hence
	\begin{align*}
		m'x_if = \left(\sum_{t\in B}\sum_{j=\max(t)}^n a_{t,j}m'x_jt\right) + \left(\sum_{\substack{\ell\in \{1,\dots, n\}\\ s\in B\\ \max(s)>\ell\\ x_{\ell}s\in E}} b_{\ell,s}m'x_{\ell}s\right).
	\end{align*}
	For every term $m'x_{\ell}s$ occurring in the second sum we have $m'x_{\ell}s\in \uSdw^{k-1}(E)$ by definition of $\uSdw$.
	Thus we just need to handle the double sum to finish the proof that $\uSdw^{k-1}(E)\cup \uSdwB^{k}(B)$ is a spanning set for $\uSdw^k(B)$.
	
	For $t\in B$ and $j \in \{ \max(t),\dots, n \}$, 
	if $\max(t) \leq \min(m'x_j)$ then $m'x_jt\in \uSdwB^k(B)$ by Lemma \ref{BorelShadowStructure}.
	If $m'=x_n^{k-1}$ then we are in the case of the previous sentence as for $t\in B$ and $j \in \{ \max(t),\dots, n \}$ we have $\max(t) \leq j=\min(m'x_j)$.
	This means that we can assume that $m'>x_n^{k-1}$ in the revlex order.
	
	So, we are concerned with the case $t\in B$ and $j \in \{ \max(t),\dots, n \}$ such that $\max(t) > \min(m'x_j)$.
	Then there exists $r\in \{1,\dots, n\}$ such that $x_r|m'$ and $r<j$.
	Hence, $m'x_jt = m''x_jx_rt$ with $m'' = m'/x_r$.
	Well, $m'> m''x_j$ in the revlex order because $j>r$.
	Thus, the problem of expressing $m'x_if$ as a linear combination of $\uSdw^{k-1}(E)\cup \uSdwB^{k}(B)$ is reduced to expressing $m''x_jx_rt$ as a linear combination of $\uSdw^{k-1}(E)\cup \uSdwB^{k}(B)$,
	and furthermore $m'> m''x_j$ in the revlex order.
	This means that we can continue recursively,
	and our recursion must stop because we already took care of the case $x_n^{k-1}$.
	
	Hence, $h=mf = m'x_if$ is in the span of $\uSdw^{k-1}(E)\cup \uSdwB^{k}(B)$,
	and thus $\uSdw^{k-1}(E)\cup \uSdwB^{k}(B)$ is a spanning set for $\uSdw^k(B)$.
	Therefore, $\uSdw^{k-1}(E)\cup \uSdwB^{k}(B)$ is a spanning set for $I_{d+k}$.
\end{proof}

\begin{lem}\label{BorelShadowEqualsShadow}
	Let $K$ be a field of arbitrary characteristic.
	If $M\subseteq R$ is a Borel set of monomials of degree $d$,
	then $\uSdwB(M) = \uSdw(M)$.
\end{lem}
\begin{proof}
	We have that $\uSdwB(M) \subseteq \uSdw(M)$ by the definitions.
	So suppose that $m\in \uSdw(M)$.
	Then $m = x_ip$ for $p=x_1^{a_1}x_2^{a_2}\cdots x_j^{a_j}$ with $a_j\geq 1$.
	If $\max(m)= i$ then $m\in \{x_ig\bigm | g\in M \text{ and } \max(g) \leq i\}$ and hence $m\in \uSdwB(M)$.
	Hence, suppose that $j>i$.
	Consider $p' = x_ip/x_j$.
	Then $p'\in M$ because $M$ is a Borel set.
	Hence, $m=x_ip=x_jp'\in \{x_jg\bigm | g\in M \text{ and } \max(g) \leq j\}$.
	Therefore, $m\in \uSdwB(M)$.
\end{proof}

\begin{cor}\label{GeneralGreenNumeric}
	Let $K$ be a field of arbitrary characteristic.
	Let $I\subseteq R$ be an ideal generated in degrees $\leq d$ and suppose that $\im(I)$ is strongly stable and has $q$ minimal generators in degree $d+1$.
	Then for all $k\geq 1$ we have $\dim I_{d+k} \leq q\binom{n-1+k-1}{n-1} + \dim R_k\im(I)_{d}$.
\end{cor}
\begin{proof}
	By Lemma \ref{reducedBasisVectorSpace} there exists a reduced basis $B$ for $I_d$.
	By Corollary \ref{BorelShadowLI} there exists a Borel shadow plus expanders basis $E\cup \uSdwB(B)$ for $I_{d+1}$.
	By Lemma \ref{BorelShadowEqualsShadow} we must have that $|E| = q$.
	By Theorem	\ref{GeneralGreen} we have
	\begin{align*}
		\dim I_{d+k} &\leq |\uSdw^{k-1}(E)\cup \uSdwB^{k}(B)| \leq |\uSdw^{k-1}(E)| + |\uSdwB^{k}(B)|.
	\end{align*}
	Since we are using $n$ variables we have $|\uSdw^{k-1}(E)| \leq q\binom{n-1+k-1}{n-1}$.
	By Lemma \ref{BorelShadowEqualsShadow} we have 
	\begin{align*}
		|\uSdwB^{k}(B)| = |\uSdw^k(\{m\in \im(I)_d \bigm | m \text{ is a monomial}\})|=\dim R_k\im(I)_{d}.
	\end{align*}
	Therefore,
	\begin{align*}
		\dim I_{d+k} \leq q\binom{n-1+k-1}{n-1} + \dim R_k\im(I)_{d}.
	\end{align*}
\end{proof}

\begin{cor}\label{GeneralGreengGin}
	Let $I\subseteq R$ be an ideal generated in degrees $\leq d$ and suppose that $\gin(I)$ has $q$ minimal generators in degree $d+1$.
	Then for all $k\geq 1$ we have
	\begin{align*}
		\dim I_{d+k} \leq q\binom{n-1+k-1}{n-1} + \dim R_k\gin(I)_{d}.
	\end{align*}
\end{cor}
\begin{proof}
	Follows from Corollary \ref{GeneralGreenNumeric} because $\gin(I)$ is strongly stable in characteristic $0$.
\end{proof}

\begin{cor}[Green \cite{Green1998}]
	Let $I\subseteq R$ be an ideal generated in degrees $\leq d$ and suppose that $\gin(I)$ has $0$ minimal generators in degree $d+1$.
	Then $\gin(I)$ has $0$ minimal generators in degree $d+2$.
\end{cor}
\begin{proof}
	Follows from Corollary \ref{GeneralGreengGin} with $q=0$.
\end{proof}

\begin{cor}\label{GeneralGreen2Vars}
	Let $I\subseteq K[x,y]$ be an ideal generated in degrees $\leq d+1$ with $I_d\neq 0$.
	If $\im(I_d)$ is strongly stable and $\im(I)$ has $q$ minimal generators in degree $d+1$ then for all $k\geq 1$ we have 
	\begin{align*}
		\dim I_{d+k} \leq kq + \dim I_d + k.
	\end{align*}
	In particular, if $\im(I)$ is strongly stable and has $q$ minimal monomial generators in degree $d+1$ then $\im(I)$ has $\leq q$ minimal monomial generators in degree $d+k$ for any $k\geq 1$.
\end{cor}
\begin{proof}
	Let $J$ be the ideal generated by $I_d$ and let $B$ be a reduced basis for $J_d$.
	Consider a Borel shadow plus expanders basis $E\cup \uSdwB(B)$ for $J_{d+1}$.
	Let $D$ be a basis for $I_{d+1}$ that contains $E\cup \uSdwB(B)$ and set $C = D\setminus (E\cup \uSdwB(B))$.
	Let $L$ be the ideal generated by $C$.
	Then $I_t=(J+L)_t$ for any $t\geq d$.
	
	Since $\im(I_d)$ is strongly stable we have that $\im(J_d)$ is strongly stable.
	Hence, $\im(\uSdwB(B))$ is strongly stable.
	Thus, for each $f\in E$ there exists a minimal generator in degree $d+1$ of $\im(I)$.
	This gives $|E|+|C| = q$.
	By Theorem \ref{GeneralGreen},
	we have that $\uSdw^{k-1}(E)\cup \uSdwB^{k}(B)$ is a spanning set for $J_{d+k}$.
	So, we have that $\uSdw^{k-1}(C) \cup \uSdw^{k-1}(E)\cup \uSdwB^{k}(B)$ is a spanning set for $I_{d+k}$.
	Hence,
	\begin{align*}
		\dim I_{d+k} &\leq |\uSdw^{k-1}(C) \cup \uSdw^{k-1}(E)\cup \uSdwB^{k}(B)|\\ 
		&\leq |\uSdw^{k-1}(C)| + |\uSdw^{k-1}(E)| + |\uSdwB^{k}(B)|\\
		&\leq k|C| + k|E| + |B|+k\\
		&= k(|C|+|E|) + |B|+k\\
		&= kq + \dim I_d + k.
	\end{align*}
	Next, we prove the claim under the assumption that $\im(I)$ is strongly stable.
	Since $\im(I)$ is strongly stable we have $ |B|+2 + |C|+|E|=|\uSdw(\im(I_{d+1}))| = |\uSdwB^{2}(B)| + q$.
	So, there can be at most $|C|+|E| = q$ minimal monomial generators of $\im(I)$ in degree $d+2$.
	We can inductively continue like this for any $k\geq 2$.
	Therefore, $\im(I)$ has $\leq q$ minimal monomial generators in degree $d+k$ for any $k\geq 1$.
\end{proof}

\section{Generic Initial Ideals and Regular Sequences}\label{ginStructure}
In this section we prove that the lengths of regular sequences force a lower bound on the uncovered bottom length $\fbl$.
The key result is Theorem \ref{advancedLefschetzBase},
which is then applied repeatedly to deduce stronger statements.
The proof of Theorem \ref{advancedLefschetzBase} combines the results from previous sections.
An overview of it goes as follows.

We assume to the contrary that $\fbl_I(d)$ is small for our ideal $I$.
Then we use the Fløystad--Stillman Theorem \ref{Fløystad--Stillman} together with the generalized Green's Crystallization Principle to force an upper bound on $\dim \gin(I_{d+k})|_{x_n}$ for $k\geq 0$.
When the minimal monomial generators of $\gin(I)_{d+k}$ behave nicely,
we get a contradiction with the lower bound for $\dim \gin(I_{d+k})|_{x_n}$ that is forced by the WLP for complete intersections.

It could be that the minimal monomial generators of $\gin(I)_{d+k}$ do not behave as we want.
Then we use Lemma \ref{MonsterLemma} to obtain a minimal monomial generator divisible by $x_3$.
We then use this minimal monomial generator to prove that the uncovered bottom length in degree $d+k$ (for some $k$) must be even smaller than the uncovered bottom length in degree $d$.

Now we can repeat the process again on an even smaller uncovered bottom length.
It could be that this process stops at some nonzero uncovered bottom length.
It could be that we get to an uncovered bottom length of $0$.
However, this forces a regular sequence of length at most one,
which contradicts assumptions about longer regular sequences inside our ideal.

Informally, 
Lemma \ref{MonsterLemma} states that under certain assumptions,
if a strongly stable ideal has a minimal monomial generator that is divisible by $x_i$ for some $i\in \{4,5,\dots, n\}$,
then it has a minimal monomial generator divisible by $x_3$.

For $N\subseteq \{1,\dots, n\}$ and $S\subseteq R$ we define the shadow of $S$ with respect to $N$ to be
\begin{align*}
	\uSdw_N(S) = \bigcup_{i\in N} \{x_if \bigm | f\in S\}.
\end{align*}

\begin{lem}\label{MonsterLemma}
	Suppose $n\geq 4$. 
	Let $M\subseteq R$ be a strongly stable monomial ideal, 
	and let $d,r,t\in \N$.
	For $k\in \N$ and $u\in \{0,1,\dots, n-3\}$ define $M_{d+k, u,<r}$ to be the set of monomials of degree $d+k$ in $M$ that are at most $x_1^{d-r}x_2^{r+k}$ in the lexicographic order and are not divisible by $x_{n},x_{n-1},\dots, x_{n-u+1}$.
	Assume that all monomial minimal generators of $M$ that are at most $x_1^{d-r}x_2^{r+k}$ in the lexicographic order in degree $d+k$ for any $k\in \{1,2,\dots, t\}$ are not divisible by $x_3,x_4,\dots, x_n$,
	and that all monomials in $M$ of degree $d$ that are at most $x_1^{d-r}x_2^{r}$ in the lexicographic order are not divisible by $x_3,x_4,\dots, x_n$.
	If $u\in \{0,1,\dots, n-3\}$ and $M_{d+t+1,u,<r}$ has a minimal monomial generator of $M$ that is divisible by $x_{n-u}$,
	then $M_{d+t+1, n-3,<r}$ has a minimal monomial generator that is divisible by $x_3$.
\end{lem}
\begin{proof}
	For $k\in \N$, $u\in \{0,1,\dots, n-3\}$, and $a\in \N$, let $P_{k,u,a}$ be all the monomials in $M_{d+k,u,<r}$ that are divisible by $x_{n-u}^a$ and not divisible by $x_{n-u}^{a+1}$.
	We first prove the following two claims.
	\begin{enumerate}
		\item If $k\in \{0,1,\dots, t-1\}$, $u\in \{0,1,\dots, n-3\}$, and $a\in \{1,2,\dots, k+1\}$ then we have that $\uSdw_{\{2,3,\dots, n-u-1\}}(P_{k,u,a})\subseteq \uSdw_{\{n-u\}}(P_{k,u,a-1}) = P_{k+1,u,a}$ and $P_{k+1,u,a}\setminus \uSdw_{\{2,3,\dots, n-u-1\}}(P_{k,u,a})$ does not contain any monomials divisible by $x_3,x_4,\dots ,x_{n-u-1},x_{n-u+1},\dots , x_n$.
		\item If $u\in \{0,1,\dots, n-4\}$ and there exists $a\geq 1$ such that $P_{t+1,u,a}$ contains a minimal monomial generator of $M$, 
		then $P_{t+1,u,a-1}$ contains a minimal monomial generator of $M$ that is divisible by $x_{n-u-1}$.
	\end{enumerate}
		
	We prove the first claim by induction on $k$.
	Suppose that $k=0$.
	Then $a=1$ and $P_{k,u,a} = \emptyset$, 
	since all monomials in $M$ of degree $d$ that are at most $x_1^{d-r}x_2^{r}$ in the lexicographic order are not divisible by $x_3,x_4,\dots, x_n$.
	Hence, $\uSdw_{\{2,3,\dots, n-u-1\}}(P_{k,u,a}) = \emptyset \subseteq  \uSdw_{\{n-u\}}(P_{k,u,a-1})$. 
	We must have $\uSdw_{\{n-u\}}(P_{k,u,a-1}) = P_{k+1,u,a}$,
	since all monomial minimal generators of $M$ that are at most $x_1^{d-r}x_2^{r+k+1}$ in the lexicographic order in degree $d+k+1$ are not divisible by $x_3,x_4,\dots, x_n$.
	Finally, $P_{k+1,u,a}\setminus \uSdw_{\{2,3,\dots, n-u-1\}}(P_{k,u,a})$ does not contain any monomials that are divisible by $x_3,x_4,\dots ,x_{n-u-1},x_{n-u+1},\dots , x_n$,
	because all monomials in $M$ of degree $d$ that are at most $x_1^{d-r}x_2^{r}$ in the lexicographic order are not divisible by $x_3,x_4,\dots, x_n$.
	
	So, suppose that $k\geq 1$ and that the claim holds for all $k'<k$.
	The claim clearly holds for $a=1$, so suppose that $a\geq 2$.
	Then
	\begin{align*}
		\uSdw_{\{2,3,\dots, n-u-1\}}(P_{k,u,a}) 
		&= \uSdw_{\{2,3,\dots, n-u-1\}}(\uSdw_{\{n-u\}}(P_{k-1,u,a-1}))\\
		&= \uSdw_{\{n-u\}}(\uSdw_{\{2,3,\dots, n-u-1\}}(P_{k-1,u,a-1}))\\
		&\subseteq \uSdw_{\{n-u\}}(P_{k,u,a-1}).
	\end{align*}
	We must have that $\uSdw_{\{n-u\}}(P_{k,u,a-1}) = P_{k+1,u,a}$ because all monomial minimal generators of $M$ that are at most $x_1^{d-r}x_2^{r+k+1}$ in the lexicographic order in degree $d+k+1$ (note that $k\leq t-1$) are not divisible by $x_3,x_4,\dots, x_n$.
	Next, we have
	\begin{align*}
		P_{k+1,u,a}\setminus \uSdw_{\{2,3,\dots, n-u-1\}}(P_{k,u,a})
		&=\uSdw_{\{n-u\}}(P_{k,u,a-1})\setminus \uSdw_{\{2,3,\dots, n-u-1\}}(P_{k,u,a})\\
		&= \uSdw_{\{n-u\}}(P_{k,u,a-1})\setminus \uSdw_{\{2,3,\dots, n-u-1\}}(\uSdw_{\{n-u\}}(P_{k-1,u,a-1}))\\
		&= \uSdw_{\{n-u\}}(P_{k,u,a-1})\setminus \uSdw_{\{n-u\}}(\uSdw_{\{2,3,\dots, n-u-1\}}(P_{k-1,u,a-1}))\\
		&= \uSdw_{\{n-u\}}(P_{k,u,a-1}\setminus \uSdw_{\{2,3,\dots, n-u-1\}}(P_{k-1,u,a-1})).
	\end{align*}
	By the inductive hypothesis $P_{k,u,a-1}\setminus \uSdw_{\{2,3,\dots, n-u-1\}}(P_{k-1,u,a-1})$ does not contain any monomials divisible by $x_3,x_4,\dots ,x_{n-u-1},x_{n-u+1},\dots , x_n$.
	Hence, $\uSdw_{\{n-u\}}(P_{k,u,a-1}\setminus \uSdw_{\{2,3,\dots, n-u-1\}}(P_{k-1,u,a-1}))$ does not contain any monomials divisible by $x_3,x_4,\dots ,x_{n-u-1},x_{n-u+1},\dots , x_n$,
	since we just multiply by $x_{n-u}$.
	Thus, the claim holds from the previous string of equalities.
	Therefore, the induction is complete and the first claim is proved.
	
	We now prove the second claim.
	Suppose that $a\geq 1$ and that $m\in M$ is a minimal monomial generator in $P_{t+1,u,a}$.
	Let $m'=x_{n-u-1}m/x_{n-u}$.
	Then $m'\in M$ because $M$ is strongly stable.
	Also, by definition of $m'$ and because $u\leq n-4$ we have $m'\in P_{t+1,u,a-1}$.
	We prove that $m'$ is a minimal generator of $M$.
	We have three cases based on $a$.
	
	{\it Case 1:} Suppose that $a=1$. 
	We need to show that $m'\not\in \uSdw_{\{2,3,\dots, n-u-1\}}(P_{t,u,a-1})$.
	Assume to the contrary that $m'\in \uSdw_{\{2,3,\dots, n-u-1\}}(P_{t,u,a-1})$.
	Then $m'=x_{n-u-1}m''$ for some $m''\in P_{t,u,a-1}$,
	since $M$ is strongly stable.
	Thus, $x_{n-u}m''\in \uSdw_{\{n-u\}}(P_{t,u,a-1}) \subseteq P_{t+1,u,a}$.
	However, $x_{n-u}m''=m$.
	Therefore, $m$ is not a minimal generator of $M$, a contradiction.
	
	{\it Case 2:} Suppose that $t+2>a>1$.
	We need to show $m'\not\in \uSdw_{\{2,3,\dots, n-u-1\}}(P_{t,u,a-1})\cup \uSdw_{\{n-u\}}(P_{t,u,a-2})$.
	Assume to the contrary that $m'\in \uSdw_{\{2,3,\dots, n-u-1\}}(P_{t,u,a-1})\cup \uSdw_{\{n-u\}}(P_{t,u,a-2})$.
	If we have $m'\in \uSdw_{\{2,3,\dots, n-u-1\}}(P_{t,u,a-1})$,
	then by a similar argument to case 1 we get that $m$ is not a minimal generator of $M$, a contradiction.
	So, we must have $m'\in \uSdw_{\{n-u\}}(P_{t,u,a-2})$ and $m'\not\in \uSdw_{\{2,3,\dots, n-u-1\}}(P_{t,u,a-1})$.
	However,
	\begin{align*}
		\uSdw_{\{n-u\}}(P_{t,u,a-2})\setminus \uSdw_{\{2,3,\dots, n-u-1\}}(P_{t,u,a-1}) 
		&= \uSdw_{\{n-u\}}(P_{t,u,a-2})\setminus \uSdw_{\{2,3,\dots, n-u-1\}}(\uSdw_{\{n-u\}}(P_{t-1,u,a-2}))\\
		&= \uSdw_{\{n-u\}}(P_{t,u,a-2})\setminus \uSdw_{\{n-u\}}(\uSdw_{\{2,3,\dots, n-u-1\}}(P_{t-1,u,a-2}))\\
		&=\uSdw_{\{n-u\}}(P_{t,u,a-2}\setminus \uSdw_{\{2,3,\dots, n-u-1\}}(P_{t-1,u,a-2})).
	\end{align*}
	We know that $P_{t,u,a-2}\setminus \uSdw_{\{2,3,\dots, n-u-1\}}(P_{t-1,u,a-2})$ does not contain any monomials divisible by $x_3,x_4,\dots ,x_{n-u-1},x_{n-u+1},\dots , x_n$.
	Hence, $\uSdw_{\{n-u\}}(P_{t,u,a-2}\setminus \uSdw_{\{2,3,\dots, n-u-1\}}(P_{t-1,u,a-2}))$ does not contain any monomials divisible by $x_3,x_4,\dots ,x_{n-u-1},x_{n-u+1},\dots , x_n$,
	since we just multiply by $x_{n-u}$.
	We have that $3\leq n-u-1 \leq n-1$ because $u\in \{0,1,\dots, n-4\}$,
	and hence $m'$ is divisible by one of $x_3,x_4,\dots ,x_{n-u-1},x_{n-u+1},\dots , x_n$,
	since $m'$ is divisible by $x_{n-u-1}$.
	So, we have a contradiction, 
	since $m'\in \uSdw_{\{n-u\}}(P_{t,u,a-2})\setminus \uSdw_{\{2,3,\dots, n-u-1\}}(P_{t,u,a-1})$,
	and $\uSdw_{\{n-u\}}(P_{t,u,a-2})\setminus \uSdw_{\{2,3,\dots, n-u-1\}}(P_{t,u,a-1})$ does not contain any monomials divisible by $x_3,x_4,\dots ,x_{n-u-1},x_{n-u+1},\dots , x_n$.
	
	{\it Case 3:} Suppose that $a>t+1$.
	Then $m'$ is a minimal generator because all monomial minimal generators of $M$ that are at most $x_1^{d-r}x_2^{r+k}$ in the lexicographic order in degree $d+k$ for any $k\in \{1,2,\dots, t\}$ are not divisible by $x_3,x_4,\dots, x_n$,
	and all monomials in $M$ of degree $d$ that are at most $x_1^{d-r}x_2^{r}$ in the lexicographic order are not divisible by $x_3,x_4,\dots, x_n$.
	
	Therefore, the second claim holds.
	We prove the main claim now.
	Suppose that $u\in \{0,1,\dots, n-3\}$ and that $M_{d+t+1,u,<r}$ has a minimal monomial generator of $M$ that is divisible by $x_{n-u}$.
	Assume to the contrary that $M_{d+t+1, n-3,<r}$ does not have a minimal monomial generator that is divisible by $x_3$.
	Then $u<n-3$.
	Without loss of generality we can take $u$ to be maximal possible.
	Then there exists $a\geq 0$ such that $P_{t+1,u,a}$ contains a minimal monomial generator of $M$.
	We must have $a\geq 1$,
	since otherwise $P_{t+1,u,a} = P_{t+1,u,0} = M_{d+t+1, u+1,<r}$,
	contradicting the maximality of $u$.
	However, we can now use the second claim to inductively deduce that $P_{t+1,u,0}$ must have a minimal monomial generator of $M$,
	contradicting the maximality of $u$ again.
	
	Therefore, the main claim holds and we are done.
\end{proof}


As mentioned earlier,
in the proof of Theorem \ref{advancedLefschetzBase},
we produce an upper bound and a lower bound for $\dim \gin(I_{d+k})|_{x_n}$ which contradict each other.
The only purpose of Lemma \ref{binomInequality} is to use it in this situation.

\begin{lem}\label{binomInequality}
	Let $C,x,p,v\in \N$ with $C,x\geq 1$ and $p\geq 2$.
	If $v > \frac{C}{p^{1/x} - 1} + x - 1$ then $p \binom{v}{x} > \binom{v+C}{x}$.
\end{lem}
\begin{proof}
	One has,
	\begin{align*}
		v &> \frac{C}{p^{1/x} - 1} + x - 1\\
		v - x + 1 &> \frac{C}{p^{1/x} - 1} \\
		\frac{1}{v - x + 1} &< \frac{p^{1/x} - 1}{C}\\
		\frac{C}{v - x + 1} &< p^{1/x} - 1 \\
		1 + \frac{C}{v - x + 1} &< p^{1/x}\\
		\left(1 + \frac{C}{v - x + 1}\right)^x &< p.
	\end{align*}
	Hence,
	\begin{align*}
		\frac{\binom{v+C}{x}}{\binom{v}{x}} &= \frac{(v+C)(v+C-1)\dots(v+C-x+1)}{v(v-1)\dots(v-x+1)}\\
		&= \left(1 + \frac{C}{v}\right)\left(1 + \frac{C}{v-1}\right)\dots\left(1 + \frac{C}{v-x+1}\right)\\
		&\leq  \left(1 + \frac{C}{v-x+1}\right)^x\\
		&<p.
	\end{align*}
	Therefore, $\binom{v+C}{x} < p \binom{v}{x}$.
\end{proof}

Take an ideal $I$ and a degree $d\geq 1$.
We define the {\it bottom length} of $I$ in degree $d$ to be $\bl_I(d)$,
the number of monomials in $\gin(I)$ of degree $d$ that are not divisible by $x_3,x_4,\dots, x_n$.
Hence, $\bl_I(d)$ counts the monomials of the form $x_1^{d-k}x_2^k$ in $\gin(I)_d$ for $k\in \{0,1,\dots, d\}$. 
Before reading the proof of Theorem \ref{advancedLefschetzBase},
the reader might find it helpful to review the definition of uncovered bottom length from Section \ref{MainResultsSection},
and the definition of Fløystad--Stillman spaces in Section \ref{FSSection}.
Figure \ref{fig:fs-example} might be useful to visually interpret some parts of the proof of Theorem \ref{advancedLefschetzBase}.
In the proof we also use results from Sections \ref{LefschetsIntro}, \ref{BezrukovSection}, and \ref{GreenSection}.

\begin{thm}\label{advancedLefschetzBase}
	Let $I\subseteq R$ be an equigenerated ideal with at most $w$ minimal generators, 
	and suppose that $I$ contains a regular sequence of length $\ell$ in degree $\alpha(I)$.
	Then there exists an integer $D=D(w,\ell, n)\in \N$ such that if $\alpha(I)\geq D$ then the following equivalent statements hold:
	\begin{enumerate}
		\item $\fbl_I(\alpha(I)) \geq \ell-1$.
		\item One has,
		\begin{align*}
			\dim (K[x_1,x_2]/\squash_{n-3}(I_{\alpha(I)}))_{\alpha(I)-1} - \dim (K[x_1,x_2]/\squash_{n-3}(I_{\alpha(I)}))_{\alpha(I)} \geq \ell - 1.
		\end{align*}	
	\end{enumerate}
	In particular, the following equivalent statements hold:
	\begin{enumerate}
		\item If $H$ is an ideal generated by general linear forms $g_1,g_2,\dots,g_{n-2}$ then
		\begin{align*}
			\dim (R/(I,H))_{\alpha(I)-1} - \dim (R/(I,H))_{\alpha(I)}\geq \ell -1.
		\end{align*}
		\item If $H$ is the ideal generated by $x_n,x_{n-1},\dots,x_{3}$ then
		\begin{align*}
			\dim (R/(\gin(I),H))_{\alpha(I)-1} - \dim (R/(\gin(I),H))_{\alpha(I)} \geq \ell-1.
		\end{align*}		
	\end{enumerate}
\end{thm}
\begin{proof}
	The first equivalence follows from the definitions.
	The second equivalence follows from Lemma \ref{Conca}.
	We prove $\fbl_I(\alpha(I)) \geq \ell-1$.
	
	We assume that $\ell \geq 2$ because the case $\ell=1$ follows from the definition of $\fbl_I$.
	Set $d=\alpha(I)$.
	Also, we assume that $x_1^{d-1}x_3\in \gin(I)$,
	since otherwise the claim follows from the strongly stable property of $\gin(I)$ if we set $D= \ell$.
	Without loss of generality we assume that $I$ is in general coordinates.
	Assume to the contrary that $\fbl_I(d) < \ell-1$.
	We will show that $d$ is bounded above by some integer.
	For an integer $q\geq 1$ we will inductively define integers $d_1,\dots, d_q$, $\ell_1,\dots, \ell_q$, $r_1,\dots ,r_q$, $C_1,\dots, C_{q}$, $D_1,\dots, D_q$, and $s_1,\dots, s_{q-1}$.
	
	Let $d_1=d$, $\ell_1 = \fbl_I(d)$ and $r_1=\bl_I(d) - \fbl_I(d)\geq 2$.
	Let $M=\im(I)$ and $M'$ be the ideal generated by $x_1^{d-r_1+1}$ and the minimal generators of $M$.
	By definition of $\ell_1$ and $r_1$ we have
	\begin{align*}
		\dim I_d = \dim M_d \leq \dim M_d' \leq \Bez_n^{r_1-1}(x_1^{d-r_1+1}) + \ell_1=\binom{n-1+r_1-1}{n-1} + \ell_1.
	\end{align*}
	Let
	\begin{align*}
		C_1 = \ceil{\frac{r_1-1}{2^{\frac{1}{n-2}}-1}} > \frac{r_1-1}{2^{\frac{1}{n-2}}-1}-1.
	\end{align*}
	By Lemma \ref{binomInequality} (with $v=n-2+c$, $C=r_1-1$, $p=2$ and $x=n-2$) for all $c\geq C_1$ we have $\binom{n-2+c+r_1-1}{n-2} < 2\binom{n-2+c}{n-2}$.
	By Corollary \ref{WLPAsymptotic} there exists an integer $D_{C_1,n}' = nC_1+1$ such that if $d\geq D_{C_1,n}'$ the WLP holds for the complete intersection contained in $I$ in degrees $\leq d+C_1$. 
	Define $D_1 = \max \{r_1+\ell_1+C_1\ell_1-1, D_{C_1,n}'\}$.
	We will now consider what happens when $d=d_1\geq D_1$.
	The Fløystad--Stillman Theorem \ref{Fløystad--Stillman} implies that $\squash_{n-3}(I_d)$ is an ideal and that $\im(\squash_{n-3}(I_d))$ is strongly stable.
	Thus, $\im(\squash_{n-3}(I_d))$ has at most $\ell_1=\fbl_I(d)\leq \ell-2$ minimal monomial generators in degree $d+1$ because $x_1^{d-1}x_3\in \im(I)$ and $\fbl_I(d) \leq \ell-2$.
	Hence, by Corollary \ref{GeneralGreen2Vars}, $\im(\squash_{n-3}(I_d))$ has at most $\ell_1$ new generators in all degrees after $d$.
	
	We will now consider two cases,
	in one case we will give an upper bound on $\dim I_{d+C_1}$ and obtain a contradiction,
	and in the other case we will continue to inductively define $\ell_2$, $r_2$, $C_2$ and $D_2$.
	
	{\it Case 1:} Assume that all monomial minimal generators of $M$ that are at most $x_1^{d-r_1}x_2^{r_1+k}$ in the lexicographic order in degree $d+k$ for any $k\in \{1,2,\dots, C_1\}$ are not divisible by $x_3,x_4,\dots, x_n$.
	Then for $k\in \{0,1,\dots, C_1\}$ we have
	\begin{align*}
		\dim I_{d+k}|_{x_n} 
		&= \dim M_{d+k}|_{x_n}\\ 
		&\leq \dim M_{d+k}'|_{x_n}\\
		&=  \sum_{\substack{\text{monomial }m\in M'\\ m \text{ minimal generator}\\ \deg(m) \leq d+k}} \Bez_{n-1}^{d+k-\deg(m)}(m) \tag{By Corollary \ref{WeightsLemmaAdvancedIdeal}}\\
		&= \Bez_{n-1}^{k+r_1-1}(x_1^{d-r_1+1})+\sum_{i=0}^{k} \sum_{\substack{\text{monomial }m\in M'\\ m \text{ minimal generator}\\ \deg(m) = d+k-i}} \Bez_{n-1}^i(m)\\
		&\leq \binom{n-2+k+r_1-1}{n-2} + \sum_{i=0}^k \ell_1\binom{n-3+i}{n-3}\\
		&= \binom{n-2+k+r_1-1}{n-2} + \ell_1 \binom{n-2+k}{n-2}.\\
		&\leq \binom{n-2+k+r_1-1}{n-2} + (\ell -2)\binom{n-2+k}{n-2}.
	\end{align*}
	Also, by Lemma \ref{ginBound} and the fact that generic initial ideals preserve inclusions,
	for every $k\in \{0,1,\dots, C_1\}$ we have
	\begin{align*}
		\dim I_{d+k}|_{x_n} 
		= \dim M_{d+k}|_{x_n}
		\geq \ell \binom{n-2+k}{n-2}. 
	\end{align*}
	However, now we get a contradiction because
	\begin{align*}
		0 &\geq \ell \binom{n-2+C_1}{n-2} - \left(\binom{n-2+C_1+r_1-1}{n-2} + (\ell -2)\binom{n-2+C_1}{n-2}\right)\\ 
		&= 2\binom{n-2+C_1}{n-2} - \binom{n-2+C_1+r_1-1}{n-2}\\
		&>0.
	\end{align*}
	
	{\it Case 2:} Assume that some minimal monomial generator of $M$ has degree $d+k$ for some $k\in \{1,2,\dots, C_1\}$,
	is at most $x_1^{d-r_1}x_2^{r_1+k}$ in the lexicographic order,
	and is divisible by one of $x_3,x_4,\dots, x_n$.
	We will show that there exists $s\in \{1,2,\dots, C_1\}$ such that $\fbl_I(d+s) \leq \ell_1-1 \leq \ell - 3$.
	Assume to the contrary that for all $s\in \{1,2,\dots, C_1\}$ we have $\fbl_I(d+s) = \ell_1$.
	Let $p\in \{1,2,\dots, C_1\}$ be the smallest integer such that some minimal monomial generator of $M$ is at most $x_1^{d-r_1}x_2^{r_1+p}$ in the lexicographic order,
	is in degree $d+p$ for some $p\in \{1,2,\dots, C_1\}$,
	and is divisible by one of $x_3,x_4,\dots, x_n$.
	If $n\geq 4$ then $M$ satisfies the conditions of Lemma \ref{MonsterLemma} with $r=r_1$ and $t=p-1$.
	Hence, if $n\geq 4$, $M_{d+p, n-3,<r_1}$ has a minimal monomial generator that is divisible by $x_3$.
	If $n=3$ then by the assumptions we have that $M_{d+p, n-3,<r_1}$ has a minimal monomial generator that is divisible by $x_3$.
	We will now construct a minimal monomial generator of $M$ in degree $d+p$ that implies $\fbl_I(d+s) < \ell_1$, and hence obtain a contradiction. 
	
	We have that $x_1^{d-r_1-\ell_1+1}x_2^{r_1+\ell_1-1}\in M_d$.
	Hence, $x_1^{d-r_1-\ell_1+1}x_2^{r_1+\ell_1-1+1}\in M_{d+1}$,
	and since $\fbl_I(d+1) = \ell_1$ we have $x_1^{d-r_1-2\ell_1+1}x_2^{r_1+2\ell_1-1+1}\in M_{d+1}$.
	Continuing like this up to degree $p$,
	we have $x_1^{d-r_1-(p+1)\ell_1+1}x_2^{r_1+(p+1)\ell_1-1+p}\in M_{d+p}$.
	Thus, 
	\begin{align*}
		x_1^{d-r_1-(p+1)\ell_1+1+\ell_1}x_2^{r_1+(p+1)\ell_1-1+p-\ell_1-1}x_3 = x_1^{d-r_1-(p-1+1)\ell_1+1}x_2^{r_1+(p-1+1)\ell_1-1+p-1}x_3 \in M_{d+p}
	\end{align*}
	is not a minimal generator of $M$,
	and if 
	\begin{align*}
		x_1^{d-r_1-(p-1+1)\ell_1+1-1}x_2^{r_1+(p-1+1)\ell_1-1+p-1+1}x_3 = x_1^{d-r_1-(p-1+1)\ell_1}x_2^{r_1+(p-1+1)\ell_1+p-1}x_3\in M_{d+p}
	\end{align*}
	then $x_1^{d-r_1-(p-1+1)\ell_1}x_2^{r_1+(p-1+1)\ell_1+p-1}x_3$ is a minimal generator of $M$.
	Continuing on like this, 
	for all $i\in \{1,\dots, p+1\}$,
	$x_1^{d-r_1-(p-i+1)\ell_1+1}x_2^{r_1+(p-i+1)\ell_1-1+p-i}x_3^i \in M_{d+p}$ is not a minimal generator of $M$,
	and if $x_1^{d-r_1-(p-i+1)\ell_1}x_2^{r_1+(p-i+1)\ell_1+p-i}x_3^i\in M_{d+p}$ then $x_1^{d-r_1-(p-i+1)\ell_1}x_2^{r_1+(p-i+1)\ell_1+p-i}x_3^i$ is a minimal generator of $M$. 
	Since $M_{d+p, n-3,<r_1}$ has a minimal monomial generator that is divisible by $x_3$ and $M$ is strongly stable,
	there exists $i\in \{1,2,\dots, p+1\}$ such that $x_1^{d-r_1-(p-i+1)\ell_1}x_2^{r_1+(p-i+1)\ell_1+p-i}x_3^{i}$ is a minimal monomial generator of $M$.
	Let $j\in \{1,2,\dots, p+1\}$ be the minimal such $i$.
	We show that $j=1$.
	Assume to the contrary that $j\geq 2$.
	Then the Fløystad--Stillman Theorem \ref{Fløystad--Stillman} implies that $\squash_{n-3}(I_{d+p})$ is an ideal and that $\im(\squash_{n-3}(I_{d+p}))$ is strongly stable.
	Thus, $\im(\squash_{n-3}(I_{d+p}))$ has at most $\ell_1-1$ minimal monomial generators in degree $d+p-(j-1)$.
	Hence, by Corollary \ref{GeneralGreen2Vars} we have that $\fbl_I(d+p-(j-2)) \leq \ell_1 - 1$.
	However, this produces a contradiction because we assumed that $\fbl_I(d+p-(j-2)) = \ell_1$.
	So, $j=1$.
	But now, this gives that $\fbl_I(d+p) < \ell_1$, a contradiction.
	Therefore, there exists $s\in \{1,2,\dots, C_1\}$ such that $\fbl_I(d+s) \leq \ell_1-1 \leq \ell - 3$.
	Let $s_1$ be the smallest such $s$ with this property.
	
	We now define
	\begin{align*}
		d_2 &= d_1+s_1\\
		\ell_2 &= \fbl_I(d_2)\\
		r_2 &= \bl_I(d_2) - \fbl_I(d_2)\\
		C_2 &= \ceil{\frac{r_2-1}{3^{\frac{1}{n-2}}-1}} > \frac{r_2-1}{3^{\frac{1}{n-2}}-1}-1\\
	\end{align*}
	By Corollary \ref{WLPAsymptotic} there exists an integer $D_{C_1+C_2,n}' = n(C_1+C_2)+1$ such that if $d\geq D_{C_1+C_2,n}'$ the WLP holds in degrees $\leq d+C_1+C_2$ for the complete intersection contained in $I$.
	Define $D_2 =\max\{r_2+\ell_2+C_2\ell_2, D_{C_1+C_2,n}'\}$.
	We can now repeat the same two cases as above,
	and either obtain a contradiction as in Case 1,
	or produce an $s_2$ as in Case 2 and continue to define $d_3$, $\ell_3$, $r_3$, $C_3$ and $D_3$.
	The key observation is that at each step $i\geq 2$ we have $\ell_{i-1} > \ell_{i}\geq 0$.
	So, this procedure must stop after some $q$ steps.
	If the procedure stops with $\ell_q >0$ then we have obtained a contradiction in case 1.
	Otherwise we get $\ell_q = 0$, 
	and by construction we have $D_q \geq r_q+\ell_q+C_q\ell_q = r_q = \bl_I(d_q)$.
	Thus, the Fløystad--Stillman Theorem \ref{Fløystad--Stillman} and Corollary \ref{GeneralGreen2Vars} imply that $\gin(I)$ never contains a power of $x_2$. 
	Therefore, $I$ does not have a regular sequence of length at least $2$,
	a contradiction.
	
	So, we must have $\fbl_I(d) \geq \ell-1$ if for all $i\in \{1,\dots, q\}$ we assume $d_i\geq D_i$.
	Note that for all $i\in \{1,\dots, q\}$ by definition of $d_i$ we have $d_i=d+ \sum_{j=1}^{i-1} s_{j}$.
	Thus, we must have
	\begin{align*}
		d < D^\star = \max \left\lbrace D_i - \sum_{j=1}^{i-1} s_{j} \bigm | i\in \{1,\dots, q\} \right\rbrace.
	\end{align*}
	So we have proved that $d$ is bounded above by $D^\star$,
	whenever $\fbl_I(d) < \ell -1$.
	The integer $D^\star$ was constructed by using some properties of the ideal $I$.
	Our claim requires an integer $D$ that only depends on $w$, $\ell$, and $n$.
	If we make small modifications we can produce an integer $D$ that does not depend on $I$.
	
	Let
	\begin{align*}
		\ell_1^\ast = \ell -2, r^\ast_1 = w, C^\ast_1 = \ceil{\frac{r^\ast_1-1}{2^{\frac{1}{n-2}}-1}}, D_{C_1^\ast,n}' = nC_1^\ast +1.
	\end{align*}
	Then set $D^\ast_1 = \max \{r^\ast_1+\ell^\ast_1+C^\ast_1\ell^\ast_1-1, D_{C_1^\ast,n}'\}$.
	By construction we have $\ell_1^\ast \geq \ell_1$, $r_1^\ast \geq r_1$, $C_1^\ast \geq C_1$, $D_{C_1^\ast}'\geq D_{C_1}'$, and $D_1^\ast \geq D_1$.
	Then for all $i\in \{2,3,\dots, \ell -1\}$ we inductively define
	\begin{align*}
		\ell_i^\ast &= \ell^\ast_{i-1} -1 \\
		r^\ast_i &= r^\ast_{i-1}+\ell^\ast_{i-1}+C^\ast_{i-1}\ell^\ast_{i-1}\\
		C^\ast_i &= \ceil{\frac{r^\ast_i-1}{(i+1)^{\frac{1}{n-2}}-1}}\\ 
		D_{C_1^\ast+\cdots +C^\ast_i,n}' &= n(C_1^\ast+\cdots +C^\ast_i) +1.
	\end{align*}
	We also set $D^\ast_i = \max \{r^\ast_i+\ell^\ast_i+C^\ast_i\ell^\ast_i, D_{C_1^\ast+\cdots +C^\ast_i,n}'\}$.
	Let $D^\ast = \max_{1\leq i \leq \ell-1} D_i^\ast$.
	By construction we have $D^\star \leq D^\ast$.
	So we have proved that $d$ is bounded above by $D^\ast$,
	whenever $\fbl_I(d) < \ell -1$,
	and $D^\ast$ only depends on $w$, $\ell$ and $n$.
	
	Therefore if $d\geq D = D^\ast$ we have $\fbl_I(d) \geq \ell-1$.
\end{proof}

If we repeatedly apply Theorem \ref{advancedLefschetzBase} then we can get a lower bound for $\fbl_I(\alpha(I)+k)$ for $k\geq 0$.

\begin{cor}\label{advancedLefschetzBaseMultipleDegreesForEquigen}
	Let $t\geq 0$, $\ell \geq 1$,
	and $I\subseteq R$ be an equigenerated ideal with at most $w$ minimal generators, 
	and suppose that $I_{\alpha(I)}$ contains a regular sequence of length $\ell$.
	There exists an integer $D = D(w,\ell, t,n)\in \N$,
	such that if $\alpha(I) \geq D$ then for all $k\in \{0,1,\dots, t\}$ we have $\fbl_I(\alpha(I)+k) \geq \ell-1$.
\end{cor}
\begin{proof}
	Note that for all $k\in \{0,1,\dots, t\}$,
	$I_{\alpha(I)+k}$ contains a regular sequence of length $\ell$ because $I_{\alpha(I)}$ does.
	Also, for all $k\in \{0,1,\dots, t\}$ we have $\dim I_{\alpha(I)+k} \leq w\binom{n-1+k}{n-1}$.
	For any $k\in \{0,1,\dots, t\}$ define $J(\alpha(I)+k)$ to be the ideal generated by $I_{\alpha(I)+k}$.
	Then by Theorem \ref{advancedLefschetzBase},
	for all $k\in \{0,1,\dots, t\}$ there exists an integer $D_k=D_k(w\binom{n-1+k}{n-1}, \ell, n)$,
	such that if $\alpha(I)+k\geq D_k$ we have $\fbl_{J(\alpha(I)+k)}(\alpha(I)+k) \geq \ell -1$.
	Let $D=\max \left\lbrace D_k - k\bigm | k\in \{0,\dots, t\} \right\rbrace$.
	Therefore, the claim holds.
\end{proof}

The reader might find it useful to review the definition of type in Section \ref{MainResultsSection}.

\begin{cor}\label{advancedLefschetzBaseMultipleDegrees}
	Consider integers $t\geq 0$ and $\ell \geq 1$,
	and suppose that $I\subseteq R$ is an ideal of type $v_{\ell,t}\in \N^{\ell+t+1}$.
	There exists an integer $D = D(v_{\ell,t})\in \N$,
	such that if $\alpha(I) \geq D$ then for all $k\in \{0,1,\dots, t\}$ we have $\fbl_I(\alpha(I)+k) \geq \gamma_I (\alpha(I)+k)-1$.
\end{cor}
\begin{proof}
	Similar to the proof of Corollary \ref{advancedLefschetzBaseMultipleDegreesForEquigen}.
\end{proof}

Restricting the results on $\fbl$ to complete intersections gives us uniqueness of the $\gin$ in a specified degree range.

\begin{thm}\label{CIGinStructure}
	Consider integers $t\geq 0$ and $\ell\geq 1$,
	and suppose $I\subseteq R$ is a complete intersection of minimal type $v_{\ell, t}\in \N^{\ell+t+1}$.
	There exists an integer $D = D(v_{\ell, t})\in \N$ such that if $\alpha(I)\geq D$ then:
	\begin{enumerate}
		\item The set of minimal monomial generators of $\gin(I)$ in degree $\alpha(I)$ is
		\begin{align*}
			\left\lbrace x_1^{\alpha(I)}, x_1^{\alpha(I)-1}x_2 ,\dots,  x_1^{\alpha(I)-\gamma_I (\alpha(I))+1}x_2^{\gamma_I (\alpha(I))-1} \right\rbrace ,
		\end{align*}
		\item For $k\in \{1,2,\dots, t\}$,  the set of minimal monomial generators of $\gin(I)$ in degree $\alpha(I)+k$ is
		\begin{align*}
			\left\lbrace x_1^{\alpha(I)+k-q_{k-1}-1}x_2^{q_{k-1}+1}, x_1^{\alpha(I)+k-q_{k-1}-2}x_2^{q_{k-1}+2} ,\dots,  x_1^{\alpha(I)+k-q_k+1}x_2^{q_k-1} \right\rbrace ,
		\end{align*}
		where for all $j\in \{0,1,\dots , t\}$ we define $q_j =\sum_{i=0}^j \gamma_I(\alpha(I)+i)$.
	\end{enumerate}
\end{thm}
\begin{proof}
	By Corollary \ref{advancedLefschetzBaseMultipleDegrees},
	there exists an integer $D$,
	such that if $\alpha(I) \geq D$ then for all $k\in \{0,1,\dots, t\}$ we have $\fbl_I(\alpha(I)+k) \geq \gamma_I (\alpha(I)+k)-1$.
	Suppose that $\alpha(I)\geq D$.
	
	We prove the claim by induction on $k$.
	Suppose that $k=0$.
	Since $\gin(I)$ is strongly stable and $\fbl_I(\alpha(I)) \geq \gamma_I (\alpha(I))-1$,
	we must have that the set of minimal monomial generators of $\gin(I)$ in degree $\alpha(I)$ is
	\begin{align*}
		\left\lbrace x_1^{\alpha(I)}, x_1^{\alpha(I)-1}x_2 ,\dots,  x_1^{\alpha(I)-\gamma_I (\alpha(I))+1}x_2^{\gamma_I (\alpha(I))-1} \right\rbrace .
	\end{align*}
	So, suppose that $k\geq 1$ and that the claim holds for all $k'<k$.
	We have $\fbl_I(\alpha(I)+k) \geq \gamma_I (\alpha(I)+k)-1$.
	So, $\gin(I)$ must have at least $\gamma_I (\alpha(I)+k)-1$ minimal monomial generators that are not divisible by $x_3,x_4,\dots, x_n$.
	We will show that it can't have more.
	
	Suppose that a regular sequence of length $\ell$ in $I$ given by $v_{\ell,t}$ has degrees $2\leq e_1\leq e_2\leq \cdots \leq e_\ell$.
	Define $\theta_1(e_1,\dots, e_\ell) = e_1$ and $\delta_1(e_1,\dots, e_\ell) = \gamma (\theta_1(e_1,\dots, e_\ell),e_1,\dots, e_\ell)-1$.
	Suppose that $w\geq 2$ and that $\theta_1(e_1,\dots, e_\ell),\dots, \theta_{w-1}(e_1,\dots, e_\ell)$ and $\delta_1(e_1,\dots, e_\ell),\dots, \delta_{w-1}(e_1,\dots, e_\ell)$ have been defined.
	If there exists an integer $d$ such that $\gamma (d,e_1,\dots, e_\ell) > \gamma (\theta_{w-1}(e_1,\dots, e_\ell),e_1,\dots, e_\ell)$,
	then define $\theta_w(e_1,\dots, e_\ell)$ to be the smallest integer that is larger than $\theta_{w-1}(e_1,\dots, e_\ell)$ such that 
	\begin{align*}
		\gamma (\theta_w(e_1,\dots, e_\ell),e_1,\dots, e_\ell) > \gamma (\theta_{w-1}(e_1,\dots, e_\ell),e_1,\dots, e_\ell),
	\end{align*}
	and we define
	\begin{align*}
		\delta_w (e_1,\dots, e_\ell)=\gamma (\theta_w(e_1,\dots, e_\ell),e_1,\dots, e_\ell) - \gamma (\theta_{w-1}(e_1,\dots, e_\ell),e_1,\dots, e_\ell).
	\end{align*}
	This defines integers $\theta_1(e_1,\dots, e_\ell),\dots, \theta_{\eta(e_1,\dots, e_\ell)}(e_1,\dots, e_\ell)$ for some integer $\eta(e_1,\dots, e_\ell)$ such that $\eta(e_1,\dots, e_\ell)\leq \ell$.
	
	Let
	\begin{align*}
		\eta &= \eta(e_1,\dots, e_\ell),\\
		\theta_i &= \theta_i(e_1,\dots, e_\ell) \text{ for } i\in \{1,\dots, \eta\}, \\
		\delta_j &= \delta_j(e_1,\dots, e_\ell) \text{ for } j\in \{1,\dots, \eta\}. \\
	\end{align*}
	By Corollary \ref{WeightsLemmaAdvancedIdeal} we have
	\begin{align*}
		\dim \gin(I)_{\alpha(I)+k} 
		&=\sum_{\substack{\text{monomial }m\in \gin(I)\\ m \text{ minimal generator}\\ \deg(m) \leq \alpha(I)+k}} \Bez^{\alpha(I)+k-\deg(m)}(m)\\
		&= \Bez^{\alpha(I)+k-\deg(x_1^{\alpha(I)})}(x_1^{\alpha(I)}) + \sum_{\substack{\text{monomial }m\in \gin(I)\setminus\{x^{\alpha(I)}\}\\ m \text{ minimal generator}\\ \deg(m) \leq \alpha(I)+k}} \Bez^{\alpha(I)+k-\deg(m)}(m)\\
		&\geq \Bez^{\alpha(I)+k-\deg(x_1^{\alpha(I)})}(x_1^{\alpha(I)}) +  \sum_{i=1}^\eta \sum_{j=\theta_i}^{\alpha(I)+k} \delta_i \binom{n-2+\alpha(I)+k-j}{n-2}\\
		&= \binom{n-1+k}{n-1} +  \sum_{i=1}^\eta \sum_{j=\theta_i}^{\alpha(I)+k} \delta_i \binom{n-2+\alpha(I)+k-j}{n-2}.
	\end{align*}
	Since $I$ is a complete intersection we must have
	\begin{align*}
		\dim \gin(I)_{\alpha(I)+k} \leq \binom{n-1+k}{n-1} +  \sum_{i=1}^\eta \sum_{j=\theta_i}^{\alpha(I)+k} \delta_i \binom{n-2+\alpha(I)+k-j}{n-2}.
	\end{align*}
	Thus, $\gin(I)$ has exactly $\gamma (\alpha(I)+k,e_1,\dots, e_\ell)-1$ minimal monomial generators that are not divisible by $x_3,x_4,\dots, x_n$.
	By the inductive hypothesis we know that $x_1^{\alpha(I)+k-1-q_{k-1}+1}x_2^{q_{k-1}-1} \in \gin(I)$.
	Hence, $x_1^{\alpha(I)+k-1-q_{k-1}+1}x_2^{q_{k-1}} \in \gin(I)$.
	Therefore, the set of minimal monomial generators of $\gin(I)$ in degree $\alpha(I)+k$ is
	\begin{align*}
		\left\lbrace x_1^{\alpha(I)+k-q_{k-1}-1}x_2^{q_{k-1}+1}, x_1^{\alpha(I)+k-q_{k-1}-2}x_2^{q_{k-1}+2} ,\dots,  x_1^{\alpha(I)+k-q_k+1}x_2^{q_k-1} \right\rbrace ,
	\end{align*}
	since $q_{k}-1 -(q_{k-1}+1)+1 = \gamma(\alpha(I)+k, e_1,\dots, e_\ell) -1 = \fbl_{I}(\alpha(I)+k)$.
\end{proof}

For a type $v = (a_1,\dots, a_\ell, b_0,\dots, b_t)\in \N^{\ell+t+1}$ we define the {\it max bound function of $v$} to be the function $\MB_{v}:\{0,1,\dots, t\} \rightarrow \N$,
such that for all $k\in \{0,1,\dots, t\}$ we have 
\begin{align*}
	\MB_{v}(k) = \sum_{i=0}^{k} b_i\binom{n-1+k-i}{n-1}.
\end{align*}
We will often just write $\MB$ when the type is clear from context. 
If $I$ is an ideal of type $v$ then $\dim I_{\alpha(I)+k} \leq \MB(k)$ for all $k\in \{0,1,\dots, t\}$.
We are now ready for the final result in this section.

\begin{thm}\label{Monster2}
	Let $t\geq 0$, $\ell \geq 1$, $v_{\ell,t}\in \N^{\ell+t+1}$, $w\in \N$,
	and suppose that $\MB(t) \leq w$.
	Let $I\subseteq R$ be an ideal of type $v_{\ell, t}\in \N^{\ell + t+1}$.
	There exists an integer $D = D(v_{\ell,t},w)\in \N$,
	such that if $\alpha(I) \geq D$ then for all $k\in \{0,1,\dots, t\}$, $r\in \{3,4,\dots, n\}$, and $e = (e_3,e_4,\dots, e_r)\in \N^{r-2}$ with $E=e_3+\cdots +e_r \leq k$,
	we have $\fbl_I^{(r,e)}(\alpha(I)+k) \geq \gamma_I (\alpha(I)+k-E)-1$.
	
	In particular,
	there exists an integer $D = D(v_{\ell,t})\in \N$,
	such that if $\alpha(I) \geq D$ then for all $k\in \{0,1,\dots, t\}$, $r\in \{3,4,\dots, n\}$, and $e = (e_3,e_4,\dots, e_r)\in \N^{r-2}$ with $E=e_3+\cdots +e_r \leq k$,
	we have $\fbl_I^{(r,e)}(\alpha(I)+k) \geq \gamma_I (\alpha(I)+k-E)-1$.
\end{thm}
\begin{proof}
	Without loss of generality,
	suppose that $I$ is in general coordinates and $\alpha(I) \geq 1$.
	Let $k\in \{0,1,\dots, t\}$, 
	$r\geq 3$, 
	and $e = (e_3,e_4,\dots, e_r)\in \N^{r-2}$ such that $E=e_3+\cdots +e_r \leq k+1$.
	Define $I_{k,r,e}$ to be the vector space spanned by all polynomials that belong to a reduced basis of $I_{\alpha(I)+k}$,
	and whose initial monomials are of the form $x_1^{\alpha(I)+k-j-E}x_2^jx_3^{e_3}x_4^{e_4}\cdots x_r^{e_r}$ with  $j\in \{0,1,2,\dots, \alpha(I)+k-E\}$.
	For each $f\in I_{k,r,e}$ we can write $f=f_1+f_2$,
	such that all the monomials that make up $f_1$ have the form $x_1^{\alpha(I)+k-j-E}x_2^jx_3^{e_3}x_4^{e_4}\cdots x_r^{e_r}$ with  $j\in \{0,1,2,\dots, \alpha(I)+k-E\}$.
	Form the set $S_{k,r,e}$ consisting of the polynomials $f_1/(x_3^{e_3}x_4^{e_4}\cdots x_r^{e_r}) \in K[x_1,x_2]$, 
	where $f\in I_{k,r,e}$ and we write $f=f_1+f_2$ as before.
	For $E\leq k$ set $e' = (e_3,e_4,\dots, e_r+1)$.
	Then for $E\leq k$ we consider $J(k,r,e) \subseteq K[x_1,x_2]$,
	which is the direct sum of $S_{k,r,e'}$ and the ideal generated by $S_{k,r,e}$.
	By the Fløystad--Stillman Theorem \ref{Fløystad--Stillman} we have that $\squash_{n-r}(I_{\alpha(I)+k})$ is an ideal and $\im (\squash_{n-r}(I_{\alpha(I)+k}))$ is strongly stable.
	Thus $J(k,r,e)$ is an ideal.
	Furthermore, $\im (J(k,r,e)_{\alpha(I)+k-E-1})$ and $\im (J(k,r,e)_{\alpha(I)+k-E})$ are strongly stable.
	Note that $k$, $r$ and $e$ were arbitrary.
	So we have defined a finite collection of ideals using these three parameters.
	For any such $k,r,e$ with $E\leq k$ we also define
	\begin{align*}
		b_{k,r,e} &= \dim (J(k,r,e) )_{\alpha(I)+k-E-1},\\
		c_{k,r,e} &= b_{k,r,e} - b_{k-E,r,0},\\
		t_{k,r,e} &= c_{k,r,e}+k.
	\end{align*}
	Also let
	\begin{align*}
		t' &= \max \{t_{k,r,e}\bigm | k\in \{0,1,\dots, t\}, r\geq 3, e = (e_3,e_4,\dots, e_r)\in \N^{r-2} \text{ with } E=e_3+\cdots +e_r \leq k\},\\
		t'' &= w+t.
	\end{align*}
	Then $t' \leq t''$ because for all $k\in \{0,1,\dots, t\}$ we have $\dim I_{\alpha(I)+k} \leq \MB(k) \leq \MB(t) \leq w$.
	Let $A(k)$ be the ideal generated by $I_{\alpha(I)+k}$ for all $k\in \{0,1,\dots, t\}$.
	For all $k\in \{0,1,\dots, t\}$ we have that $A(k)$ is equigenerated by at most $\MB(k)$ minimal generators.
	By Corollary \ref{advancedLefschetzBaseMultipleDegreesForEquigen},
	for all $k\in \{0,1,\dots, t\}$,
	there exists an integer $D_k$,
	such that if $\alpha(A(k)) \geq D_k$ then for all $k'\in \{0,1,\dots, t''\}$ we have $\fbl_{A(k)}^{(3,0)} (\alpha(A(k)) +k') = \fbl_{A(k)}(\alpha(A(k))+k') \geq \gamma_I(\alpha(I)+k)-1$.
	Let $D=\max \{D_k \bigm | k\in \{0,1,\dots, t\}$.
	Suppose that $\alpha(I)\geq D$.
	Assume to the contrary that there exist $k\in \{0,1,\dots, t\}$, $r\geq 3$, and $e = (e_3,e_4,\dots, e_r)\in \N^{r-2}$ such that $E=e_3+\cdots +e_r \leq k$,
	for which we have $\fbl_I^{(r,e)}(\alpha(I)+k) < \gamma_I (\alpha(I)+k-E)-1$.
	Then we must have $k\geq 1$.
	
	By Corollary \ref{GeneralGreen2Vars} we have
	\begin{align*}
		\dim (J(k,r,e))_{\alpha(I)+t''} &= \dim (J(k,r,e))_{(\alpha(I)+k-E-1)+(t''-k+E+1)}\\
		&\leq (t''-k+E+1)(\fbl^{(r,e)}_I(\alpha(I)+k)+1) + \dim (J(k,r,e))_{\alpha(I)+k-E-1}\\
		&\leq (t''-k+E+1)(\gamma_I (\alpha(I)+k-E)-1) + \dim (J(k,r,e))_{\alpha(I)+k-E-1}.
	\end{align*}
	Also,
	we know that for all $k'\in \{0,1,\dots, t''\}$ we have $\fbl_{A(k)}(\alpha(A(k))+k') \geq \gamma_I(\alpha(I)+k)-1$,
	which forces
	\begin{align*}
		\dim (J(k,r,e))_{\alpha(I)+t''} &\geq \dim (J(k-E,r,0))_{\alpha(I)+t''}\\
		&\geq (t''-k+E+1)(\fbl_{A(k-E)}(\alpha(I)+k-E)+1) + \dim (J(k-E,r,0))_{\alpha(I)+k-E-1}\\
		&\geq (t''-k+E+1)\gamma_I (\alpha(I)+k-E) + \dim (J(k-E,r,0))_{\alpha(I)+k-E-1}.\\
	\end{align*}
	Thus,
	\begin{align*}
		0&\leq \dim (J(k,r,e))_{\alpha(I)+k-E-1} - \dim (J(k-E,r,0))_{\alpha(I)+k-E-1} - (t''-k+E+1)\\
		&= b_{k,r,e} - b_{k-E,r,0} -t''+k-E-1\\
		&= c_{k,r,e} -t''+k-E-1\\
		&\leq c_{k,r,e} -t_{k,r,e}+k-E-1\\
		&=c_{k,r,e} -c_{k,r,e}-k+k-E-1\\
		&\leq-1.
	\end{align*}
	Therefore, we have a contradiction and the claim holds.
	In order to construct an integer that only depends on $v_{\ell, t}$ we can take $w = \MB(t)$.
\end{proof}

\section{Lex Plus Powers Transformable Ideals}\label{TransformSection}
The results of Section \ref{ginStructure} identify,
in each degree $\alpha(I)+k$,
distinguished forced monomials inside of $\gin(I)_{\alpha(I)+k}$.
We remove these forced monomials and their shadows from $\gin(I)_{\alpha(I)+k}$ to obtain sets $\mathfrak{L}_{I,k}\subseteq \gin(I)_{\alpha(I)+k}$.  
The goal of this section is to compress the sets $\mathfrak{L}_{I,k}$ simultaneously until they become strongly stable,
while preserving both their distribution with respect to $\max(m)$ and their growth from one degree to the next.
These two preserved quantities are exactly what will be needed in Section \ref{EGHResults}.

\begin{figure}[htbp]
	\centering
	\includegraphics[width=0.42\textwidth]{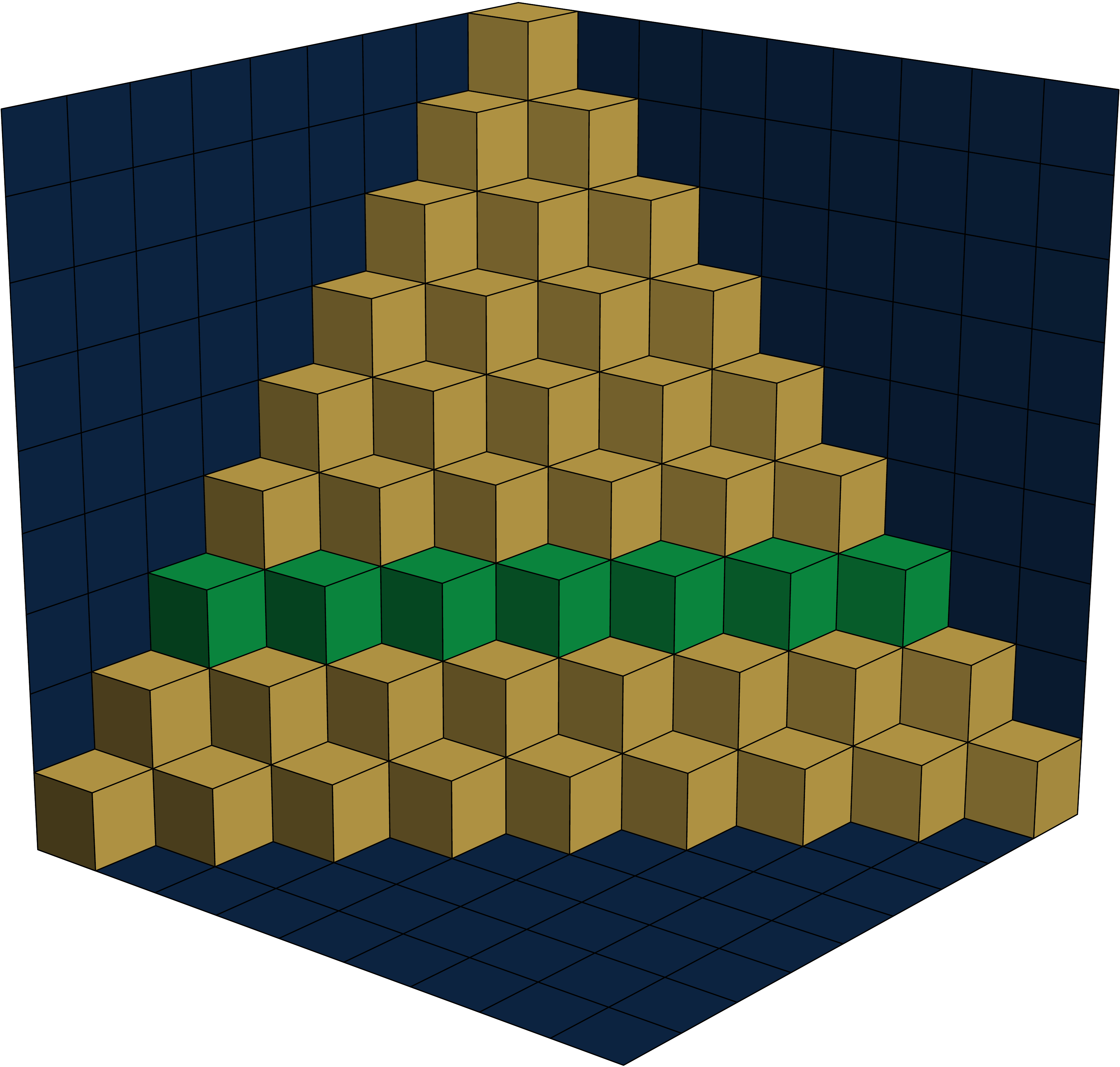}
	\caption{An example of $\mathfrak{M}_{3,8}(1,2, \mathfrak{f}_{1,2})$ inside of $\mathfrak{M}_{3,8}$.}
	\label{fig:line-example}
\end{figure}

For $n,k\in \N$ let $[n] = \{1,\dots, n\}$ and for $S\subseteq [n]$ define $S^C = [n]\setminus S$. 
By $\mathfrak{M}_{n,k}$ we denote the set of monomials of degree $k$ in $R$,
and by $\mathfrak{M}=\mathfrak{M}_n$ denote the set of all monomials in $R$.
For all $w\in [n]$ define $\mathfrak{M}_{n,k,w}$ be the set of monomials in $\mathfrak{M}_{n,k}$ that are not divisible by $x_{w+1}, x_{w+2},\dots, x_n$,
where $\mathfrak{M}_{n,k,n} = \mathfrak{M}_{n,k}$.

We will now cut up $\mathfrak{M}_{n,k}$ into different parts.
See Figure \ref{fig:line-example} for one of these parts.
Take $i,j\in \{1,\dots, n\}$ with $i<j$, 
and for each $u\in \{i,j\}^C$ consider $e_u\in \N$ such that $\sum_{u\in \{i,j\}^C} e_u \leq k$.
Define $\mathfrak{f}_{i,j}$ on $\{i,j\}^C$ by $\mathfrak{f}_{i,j}(u)=e_u$.
Then define $\mathfrak{M}_{n,k}(i,j, \mathfrak{f}_{i,j})$ to be the set of all monomials $m\in\mathfrak{M}_{n,k}$,
such that for $m'=\prod_{u\in \{i,j\}^C} x_u^{e_u}$ we have that $m/m'$ is not divisible by $x_u$ for all $u\in \{i,j\}^C$.
Note that if we fix $i$ and $j$ and go over all possible $\mathfrak{f}_{i,j}$, 
then the $\mathfrak{M}_{n,k}(i,j, \mathfrak{f}_{i,j})$ partition $\mathfrak{M}_{n,k}$.
We say that this is the {\it partitioning} of $\mathfrak{M}_{n,k}$ with respect to $i$ and $j$.

In the results to follow,
we will often be working with fixed $n,k,i,j, \mathfrak{f}_{i,j}$.
It is cumbersome and distracting to keep track of these parameters in our notation.
Also, we want to be able to say when the strongly stable property of a set is preserved under certain operations.
For these reasons we define some simplified and extra notation.
Let $m\in \mathfrak{M}_{n,k}(i,j, \mathfrak{f}_{i,j})$ and define the {\it line} of $m$ to be $\Line(m) = \mathfrak{M}_{n,k}(i,j, \mathfrak{f}_{i,j})$.
Note that two different monomials can have the same line,
but each monomial belongs to a unique line because the sets $\mathfrak{M}_{n,k}(i,j, \mathfrak{f}_{i,j})$ partition $\mathfrak{M}_{n,k}$.

Suppose that $m\in \mathfrak{M}_{n,k}$ and $p,q\in [n]$ with $p<q$ and $x_q|m$.
We define the {\it Borel move of $m$} by $\mathfrak{b}_{p,q}(m)=x_pm/x_q$.
For $A\subseteq \mathfrak{M}_{n,k}$ we define $\mathfrak{b}_{p,q}(A)$ to be the set of monomials of the form $\mathfrak{b}_{p,q}(m)$,
where $m\in A$ and $x_q|m$.
Lemma \ref{lineStructure} tells us how lines change after a Borel move.

\begin{lem}\label{lineStructure}
	Let $n,k\in \N$ and suppose that $i,j,p,q\in [n]$ with $i<j$, $p<q$, and $\mathfrak{b} = \mathfrak{b}_{p,q}$.
	Consider the partitioning of $\mathfrak{M}_{n,k}$ with respect to $i$ and $j$. 
	Pick $m^\ast\in \mathfrak{M}_{n,k}$ with $x_q|m^\ast$.
	\begin{enumerate}
		\item If $p\not\in \{i,j\}$ and $q\not\in \{i,j\}$ then the restriction of $\mathfrak{b}$ to $\Line(m^\ast)$ is a bijection onto $\Line(\mathfrak{b}(m^\ast))$.
		In particular, $|\Line(m^\ast)| = |\Line(\mathfrak{b}(m^\ast))|$.
		\item Suppose that $p\not\in \{i,j\}$ and $q\in \{i,j\}$.
		\begin{enumerate}
			\item Suppose $q=i$. 
			Let $m_s$ be the smallest monomial in the lex order in $\Line(m^\ast)$, hence $x_i\not | m_s$. 
			Then the restriction of $\mathfrak{b}$ to $\Line(m^\ast)\setminus\{m_s\}$ is a bijection onto $\Line(\mathfrak{b}(m^\ast))$.
			\item Suppose $q=j$. 
			Let $m_l$ be the largest monomial in the lex order in $\Line(m^\ast)$, hence $x_j\not| m_l$. 
			Then the restriction of $\mathfrak{b}$ to $\Line(m^\ast)\setminus\{m_l\}$ is a bijection onto $\Line(\mathfrak{b}(m^\ast))$.
		\end{enumerate}
		In particular, we have $|\Line(m^\ast)| = 1 + |\Line(\mathfrak{b}(m^\ast))|$.
		\item Suppose that $p\in \{i,j\}$ and $q\not\in \{i,j\}$.
		\begin{enumerate}
			\item Suppose $p=i$. 
			Let $m_s$ be the smallest monomial in the lex order in $\Line(\mathfrak{b}(m^\ast))$, hence $x_i\not | m_s$. 
			Then the restriction of $\mathfrak{b}$ to $\Line(m^\ast)$ is a bijection onto $\Line(\mathfrak{b}(m^\ast))\setminus \{m_s\}$.
			\item Suppose $p=j$. 
			Let $m_l$ be the largest monomial in the lex order in $\Line(\mathfrak{b}(m)^\ast)$, hence $x_j\not| m_l$. 
			Then the restriction of $\mathfrak{b}$ to $\Line(m^\ast)$ is a bijection onto $\Line(\mathfrak{b}(m^\ast))\setminus\{m_l\}$.
		\end{enumerate}
		In particular, we have $1 + |\Line(m^\ast)| = |\Line(\mathfrak{b}(m^\ast))|$.
		\item If $p\in \{i,j\}$ and $q\in \{i,j\}$ then $\Line(m^\ast) = \Line(\mathfrak{b}(m^\ast))$.
	\end{enumerate}
\end{lem}
\begin{proof}
	Suppose that $m^\ast\in \mathfrak{M}_{n,k}(i,j, \mathfrak{f}_{i,j})$ and $\mathfrak{b}(m^\ast) \in \mathfrak{M}_{n,k}(i,j, \mathfrak{g}_{i,j})$.
	Set $\mathfrak{f} = \mathfrak{f}_{i,j}$ and $\mathfrak{g} = \mathfrak{g}_{i,j}$.
	
	If $p\not\in \{i,j\}$ and $q\not\in \{i,j\}$ then $\sum_{\{i,j\}^C} \mathfrak{f}(u) = \sum_{\{i,j\}^C} \mathfrak{g}(u)$,
	hence the restriction of $\mathfrak{b}$ to $\Line(m^\ast)$ is a bijection onto $\Line(\mathfrak{b}(m^\ast))$.
	So $|\Line(m^\ast)| = |\Line(\mathfrak{b}(m^\ast))|$.
	
	If $p\not\in \{i,j\}$ and $q\in \{i,j\}$ then $\sum_{\{i,j\}^C} \mathfrak{f}(u) = 1 + \sum_{\{i,j\}^C} \mathfrak{g}(u)$,
	thus $|\Line(m^\ast)| = 1 + |\Line(\mathfrak{b}(m^\ast))|$.
	For the remaining details,
	note that division by $x_i$ or $x_j$ produces a bijection.
	If $p\in \{i,j\}$ and $q\not\in \{i,j\}$ then $1 + \sum_{\{i,j\}^C} \mathfrak{f}(u) = \sum_{\{i,j\}^C} \mathfrak{g}(u)$,
	and the other details follow for the same reason.
	
	Finally, we are left with the case $p\in \{i,j\}$ and $q\in \{i,j\}$.
	Then $p=i$ and $q=j$.
	Therefore, $\Line(m^\ast) = \Line(\mathfrak{b}(m^\ast))$.
\end{proof}

For a set $A\subseteq \mathfrak{M}_{n,k}(i,j, \mathfrak{f}_{i,j})$ define the {\it compression} of $A$, 
$\comp_{i,j, \mathfrak{f}_{i,j}}(A)$, 
to be the last $|A|$ monomials in the lexicographic order in $\mathfrak{M}_{n,k}(i,j, \mathfrak{f}_{i,j})$.
Then for $A\subseteq \mathfrak{M}_{n,k}$ and $i,j\in [n]$ define the {\it compression} of $A$ with respect to $i$ and $j$ to be the disjoint union
\begin{align*}
	\comp_{i,j}(A) = \bigcup_{\mathfrak{f}_{i,j}} \comp_{i,j, \mathfrak{f}_{i,j}}(A\cap \mathfrak{M}_{n,k}(i,j, \mathfrak{f}_{i,j})).
\end{align*}

\begin{figure}[htbp]
	\centering
	
	\begin{subfigure}[t]{0.48\textwidth}
		\centering
		\includegraphics[width=\linewidth]{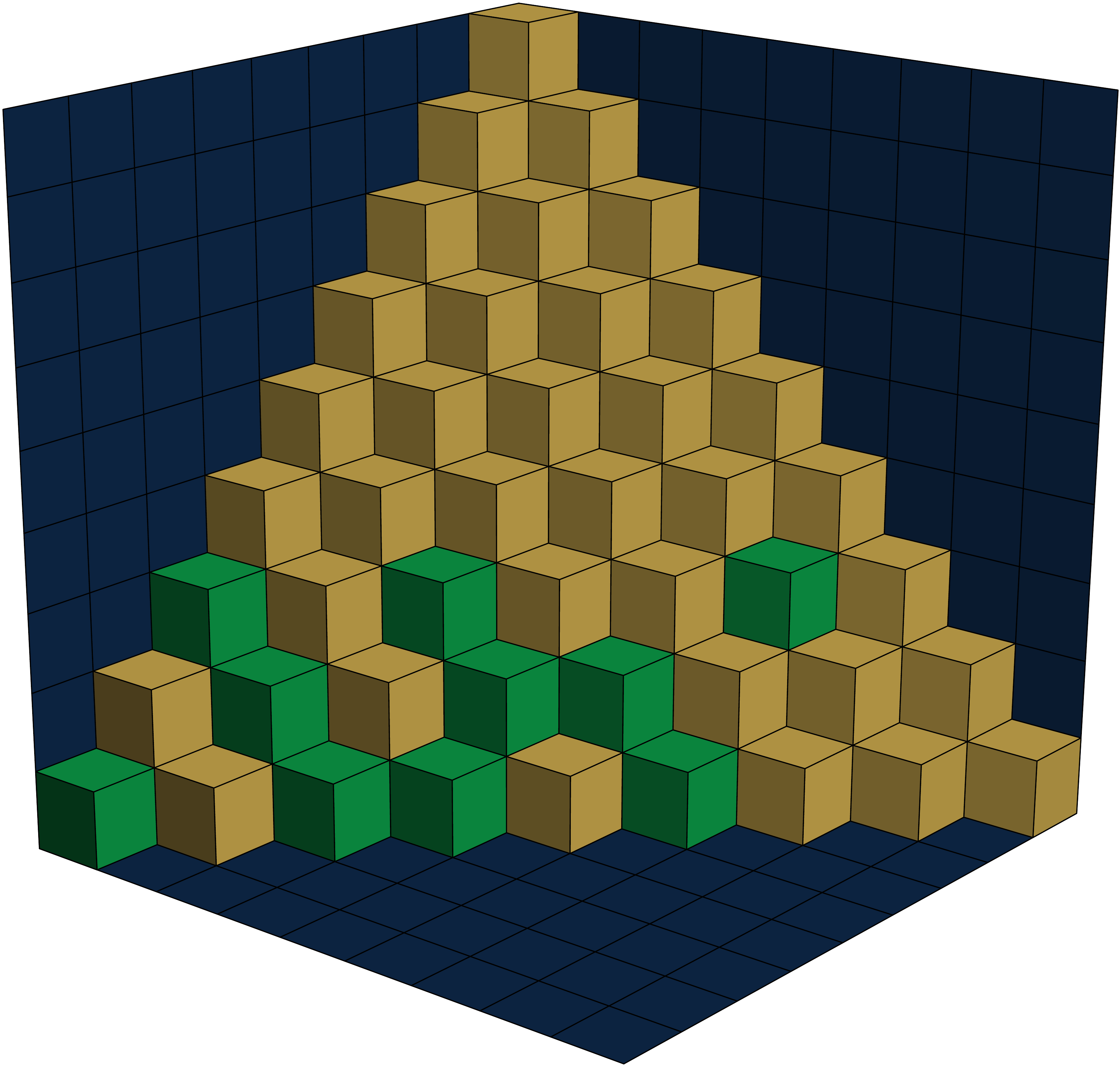}
		\caption{A set \(A\) in degree \(8\).}
		\label{fig:compression-set}
	\end{subfigure}
	\hfill
	\begin{subfigure}[t]{0.48\textwidth}
		\centering
		\includegraphics[width=\linewidth]{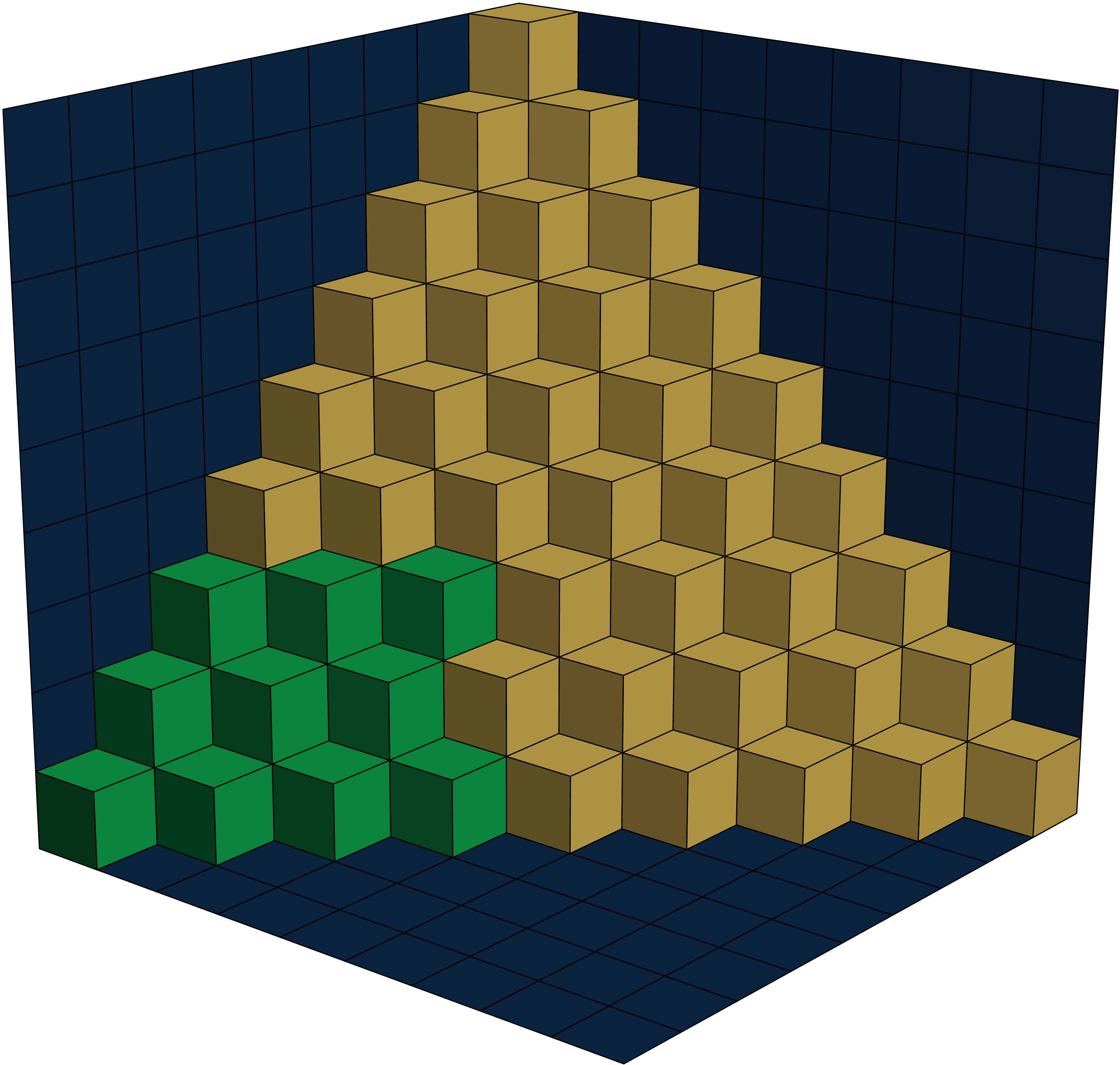}
		\caption{The compression \(\comp_{1,2}(A)\).}
		\label{fig:compression-comp}
	\end{subfigure}
	
	\caption{An illustration of the compression operation in degree \(8\) when $n=3$.}
	\label{fig:compression-example}
\end{figure}

See Figure \ref{fig:compression-example} for an example of compression.
From the definition we get that $|A| = |\comp_{i,j}(A)|$,
and if $A\subseteq B\subseteq \mathfrak{M}_{n,k}$ then $\comp_{i,j}(A)\subseteq \comp_{i,j}(B)$.
If $\comp_{i,j}(A) = A$ then we say that $A$ is $(i,j)$ compressed.
We will omit the indices when they are clear from context and just write $\comp$.

\begin{cor}\label{BorelComp}
	Let $A\subseteq \mathfrak{M}_{n,k}$, $m^\ast \in \mathscr{M}_{n,k}$ and $i,j,p,q\in [n]$ with $i<j$, $p<q$.
	Consider the partitioning of $\mathfrak{M}_{n,k}$ with respect to $i$ and $j$.
	\begin{enumerate}
		\item If $p\not\in \{i,j\}$ and $q\not\in \{i,j\}$, and $\mathfrak{b}_{p,q}(A\cap \Line(m^\ast)) \subseteq A$,
		then $\mathfrak{b}_{p,q}(\comp(A)\cap \Line(m^\ast)) \subseteq \comp(A)$.
		\item Suppose that $p\not\in \{i,j\}$ and $q\in \{i,j\}$.
		\begin{enumerate}
			\item Suppose $q=i$.
			If $\mathfrak{b}_{p,i}(A\cap \Line(m^\ast)) \subseteq A$ and $\mathfrak{b}_{p,j}(A\cap \Line(m^\ast)) \subseteq A$,
			then $\mathfrak{b}_{p,i}(\comp(A)\cap \Line(m^\ast)) \subseteq \comp(A)$ and $\mathfrak{b}_{p,j}(\comp(A)\cap \Line(m^\ast)) \subseteq \comp(A)$.
			\item Suppose $q=j$. 
			If $\mathfrak{b}_{p,q}(A\cap \Line(m^\ast)) \subseteq A$,
			then $\mathfrak{b}_{p,q}(\comp(A)\cap \Line(m^\ast)) \subseteq \comp(A)$.
		\end{enumerate}
		\item Suppose that $p\in \{i,j\}$ and $q\not\in \{i,j\}$.
		\begin{enumerate}
			\item Suppose $p=i$. 
			If $\mathfrak{b}_{p,q}(A\cap \Line(m^\ast)) \subseteq A$,
			then $\mathfrak{b}_{p,q}(\comp(A)\cap \Line(m^\ast)) \subseteq \comp(A)$.
			\item Suppose $p=j$. 
			If $\mathfrak{b}_{i,q}(A\cap \Line(m^\ast)) \subseteq A$ and $\mathfrak{b}_{j,q}(A\cap \Line(m^\ast)) \subseteq A$,
			then $\mathfrak{b}_{i,q}(\comp(A)\cap \Line(m^\ast)) \subseteq \comp(A)$ and $\mathfrak{b}_{j,q}(\comp(A)\cap \Line(m^\ast)) \subseteq \comp(A)$.
		\end{enumerate}
		\item If $p\in \{i,j\}$ and $q\in \{i,j\}$, and $\mathfrak{b}_{p,q}(A\cap \Line(m^\ast)) \subseteq A$,
		then $\mathfrak{b}_{p,q}(\comp(A)\cap \Line(m^\ast)) \subseteq \comp(A)$.
	\end{enumerate}
\end{cor}
\begin{proof}
	Follows from Lemma \ref{lineStructure} and the definitions of $\comp$ and $\mathfrak{b}$.
\end{proof}

Suppose $n\geq 3$ and that $I\subseteq K[x_1,\dots, x_n]$ is an ideal of type $v_{\ell,t} \in \N^{\ell +t+1}$.
We say that $I$ is \textit{lex plus powers transformable with respect to $v_{\ell,t}$} ($\LPPT(v_{\ell, t})$),
if for all $k\in \{0,1,\dots, t\}$, $r\geq 3$, and $e = (e_3,e_4,\dots, e_r)\in \N^{r-2}$ such that $E=e_3+\cdots +e_r \leq k$,
we have $\fbl_I^{(r,e)}(\alpha(I)+k) \geq \gamma_I (\alpha(I)+k-E)-1$.
Let $f_1,\dots, f_\ell$ be the regular sequence given by the type $v_{\ell, t}$ with degrees $\deg(f_1),\dots, \deg(f_\ell)$.
For all $k\in \{0,1,\dots, t\}$,
if $\alpha(I)+k \geq \deg(f_1)$ then we define $\Gamma_I(\alpha(I)+k)$ to be the set of the smallest $\gamma_I (\alpha(I)+k)-1$ monomials with respect to the lexicographic order that are inside $\gin(I)_{\alpha(I)+k}$,
otherwise $\Gamma_I(\alpha(I)+k) = \emptyset$.
Note that the definitions of lexicographic order, $\LPPT(v_{\ell, t})$, and $\Gamma_I(\alpha(I)+k)$ imply that $\Gamma_I(\alpha(I)+k)$ only contains monomials that are divisible by $x_2$,
and not divisible by $x_3,x_4,\dots, x_n$.
Furthermore, $\Gamma_I(\alpha(I)+k)$ does not contain $x_1^{\alpha(I)+k}$ by the definition of uncovered bottom length.

If $I$ is of type $v_{\ell,t}$ and is $\LPPT(v_{\ell, t})$,
then for all  $k\in \{0,1,\dots, t\}$ we define
\begin{align*}
	\mathfrak{P}_{I,k} &= \bigcup_{i=0}^k \uSdwB^{i}(\Gamma_I(\alpha(I)+k-i)),\\
	\mathfrak{L}_{I,k} &= \gin(I)_{\alpha(I)+k} \setminus \mathfrak{P}_{I,k}.
\end{align*}

\begin{figure}[htbp]
	\centering
	
	\begin{subfigure}[t]{0.315\textwidth}
		\centering
		\includegraphics[width=\linewidth]{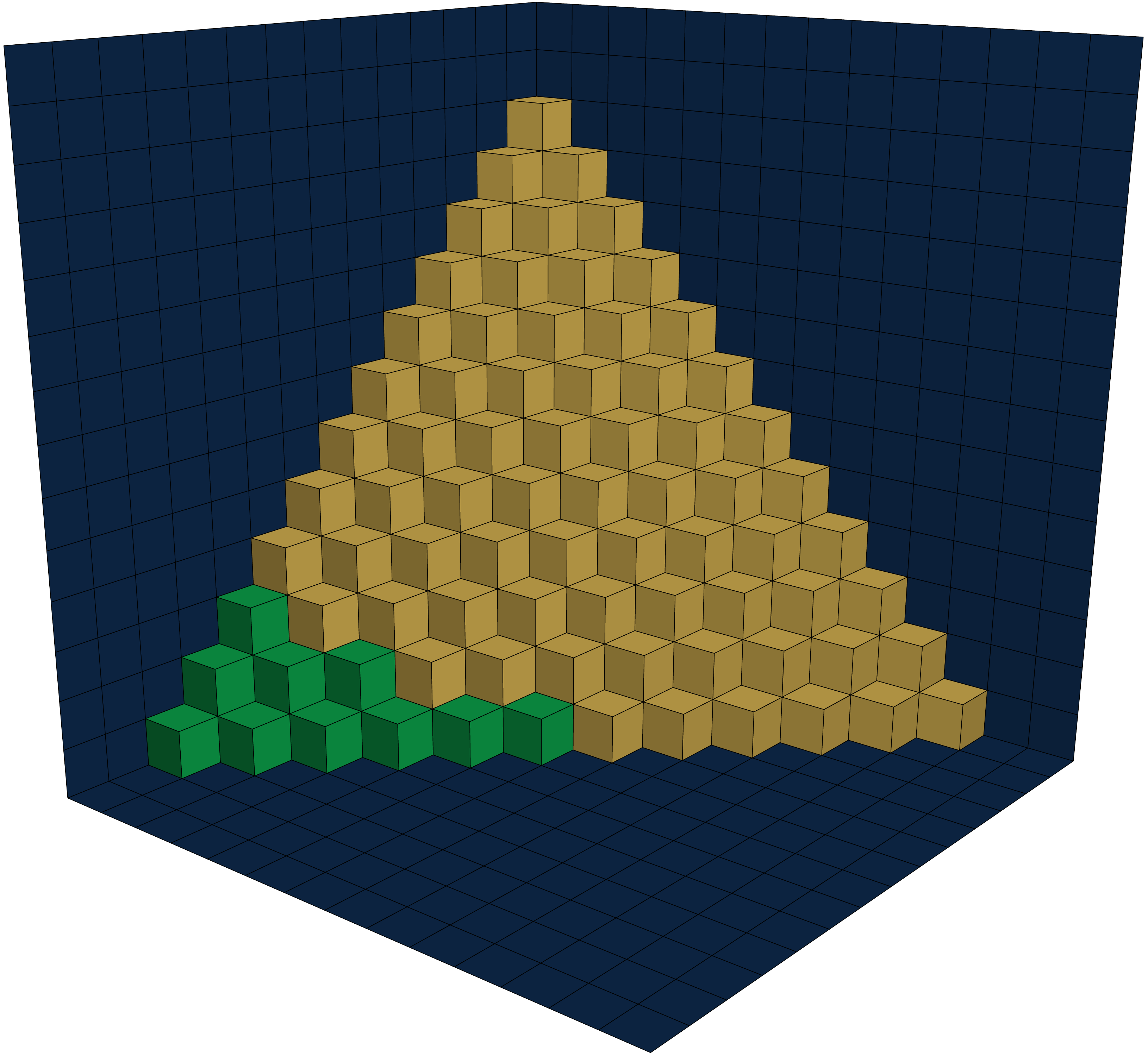}
		\caption{$\gin(I)_{\alpha(I)}$}
		\label{fig:b1}
	\end{subfigure}
	\hfill
	\begin{subfigure}[t]{0.315\textwidth}
		\centering
		\includegraphics[width=\linewidth]{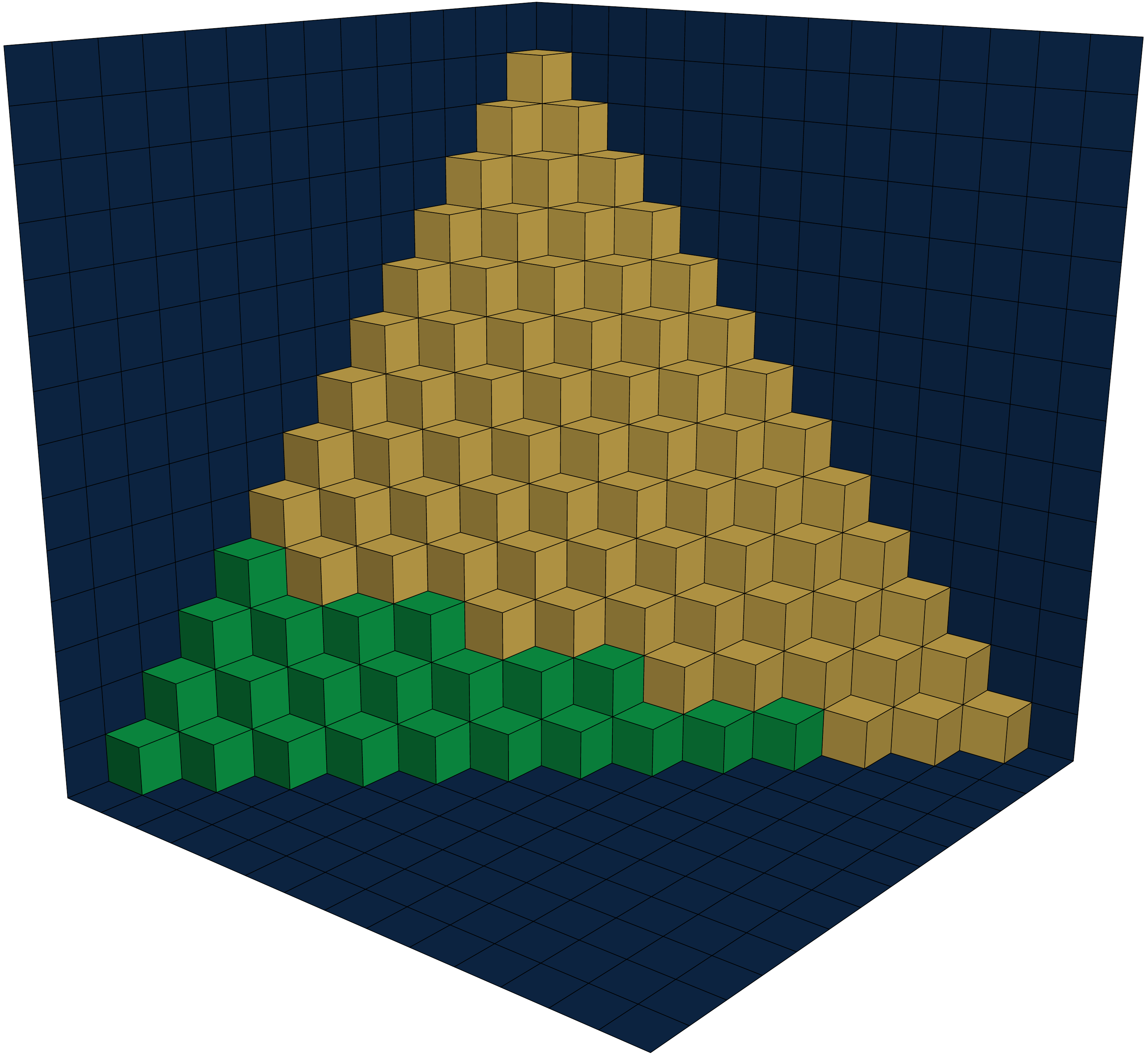}
		\caption{$\gin(I)_{\alpha(I)+1}$}
		\label{fig:b2}
	\end{subfigure}
	\hfill
	\begin{subfigure}[t]{0.315\textwidth}
		\centering
		\includegraphics[width=\linewidth]{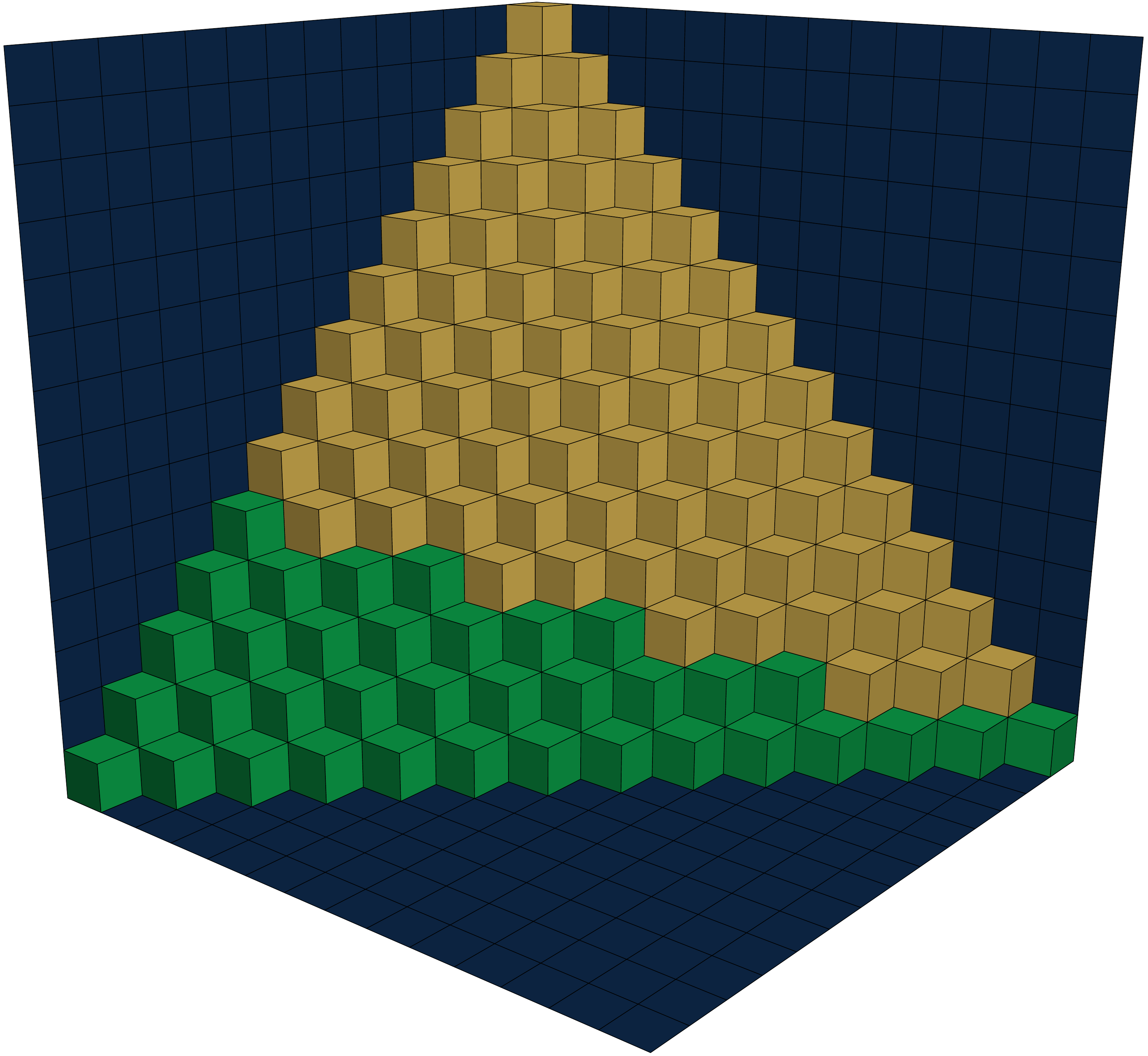}
		\caption{$\gin(I)_{\alpha(I)+2}$}
		\label{fig:b3}
	\end{subfigure}
	
	\medskip
	
	\begin{subfigure}[t]{0.315\textwidth}
		\centering
		\includegraphics[width=\linewidth]{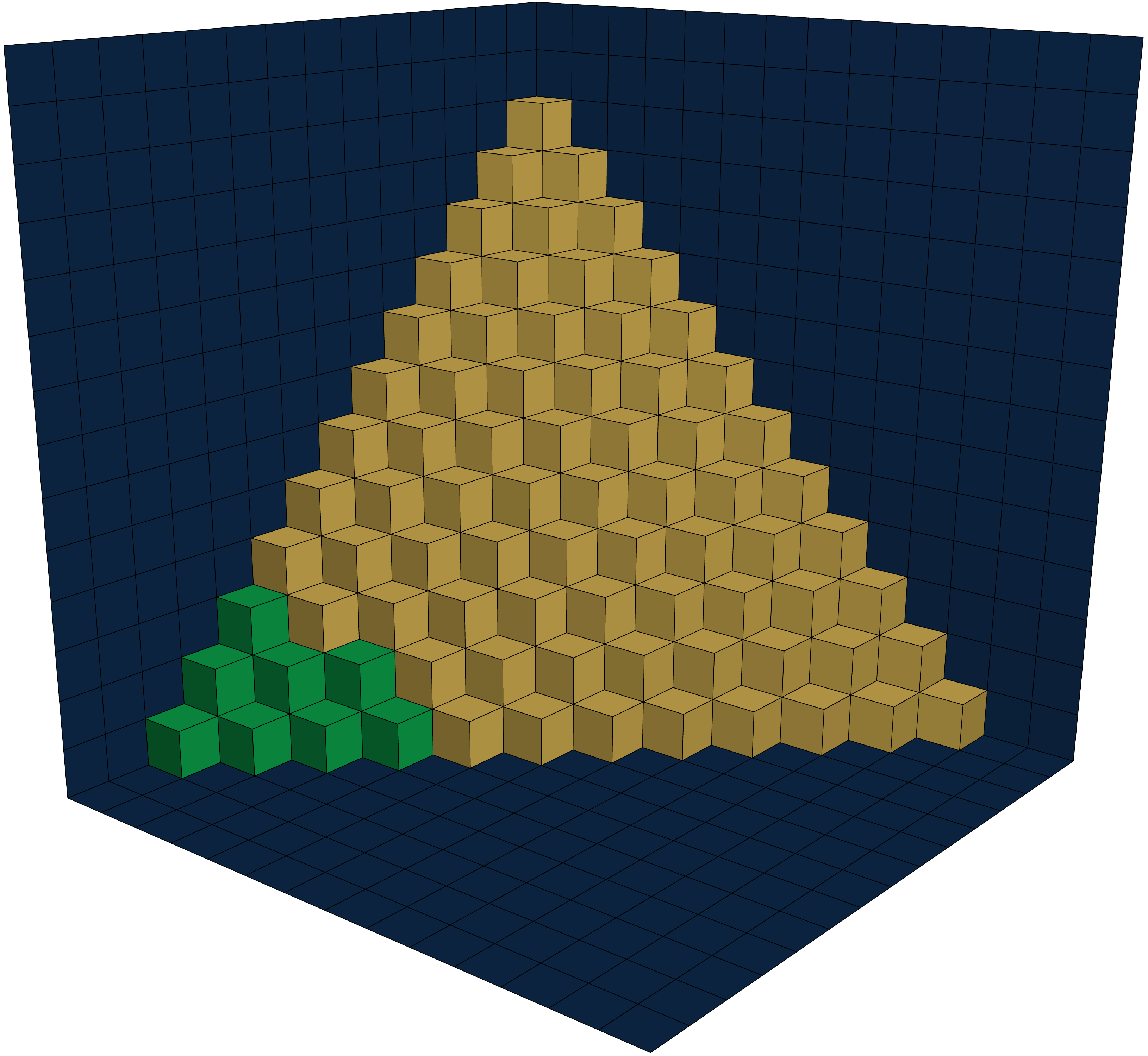}
		\caption{$\mathfrak{L}_{I,0}$}
		\label{fig:b1-rem}
	\end{subfigure}
	\hfill
	\begin{subfigure}[t]{0.315\textwidth}
		\centering
		\includegraphics[width=\linewidth]{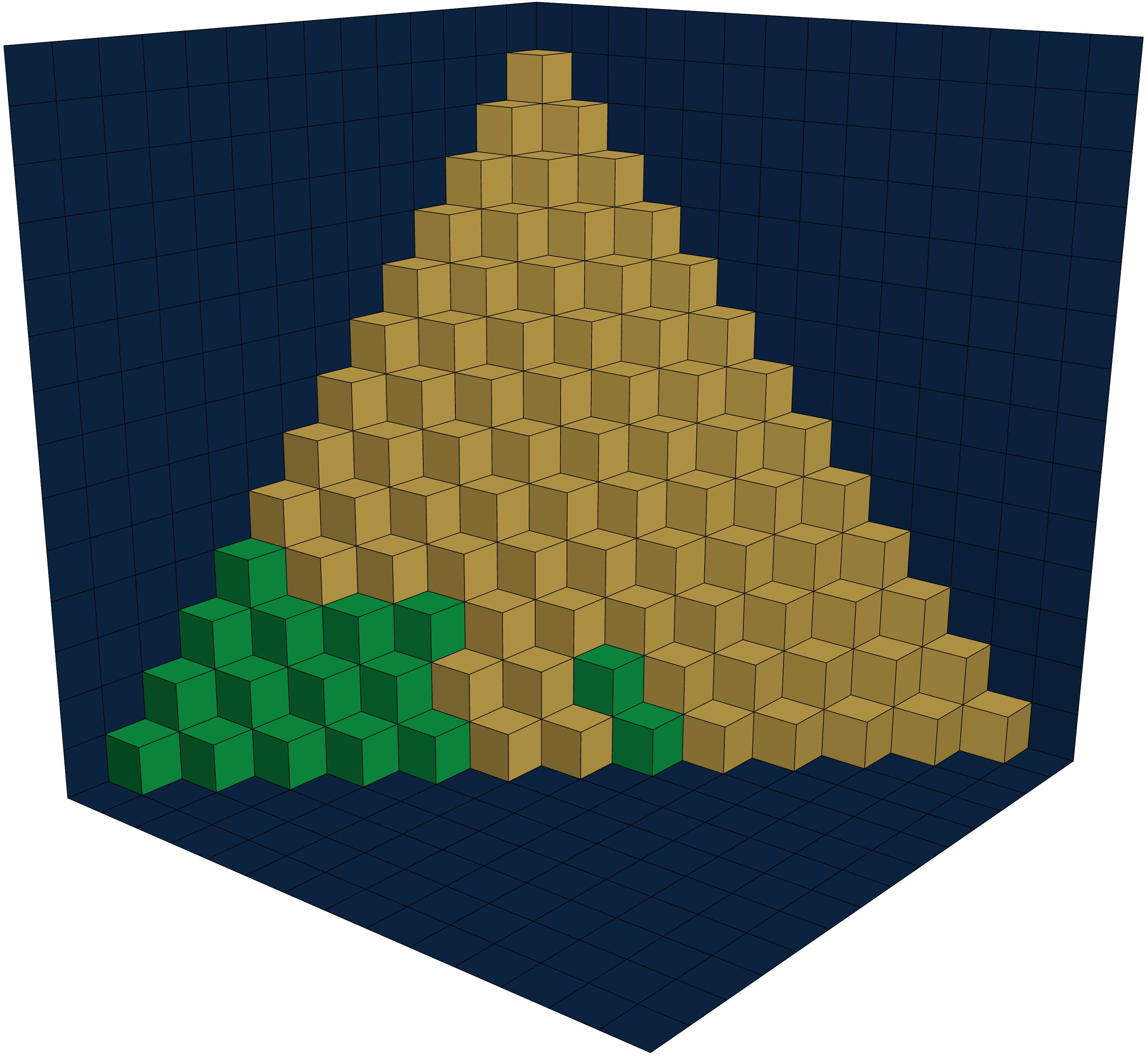}
		\caption{$\mathfrak{L}_{I,1}$}
		\label{fig:b2-rem}
	\end{subfigure}
	\hfill
	\begin{subfigure}[t]{0.315\textwidth}
		\centering
		\includegraphics[width=\linewidth]{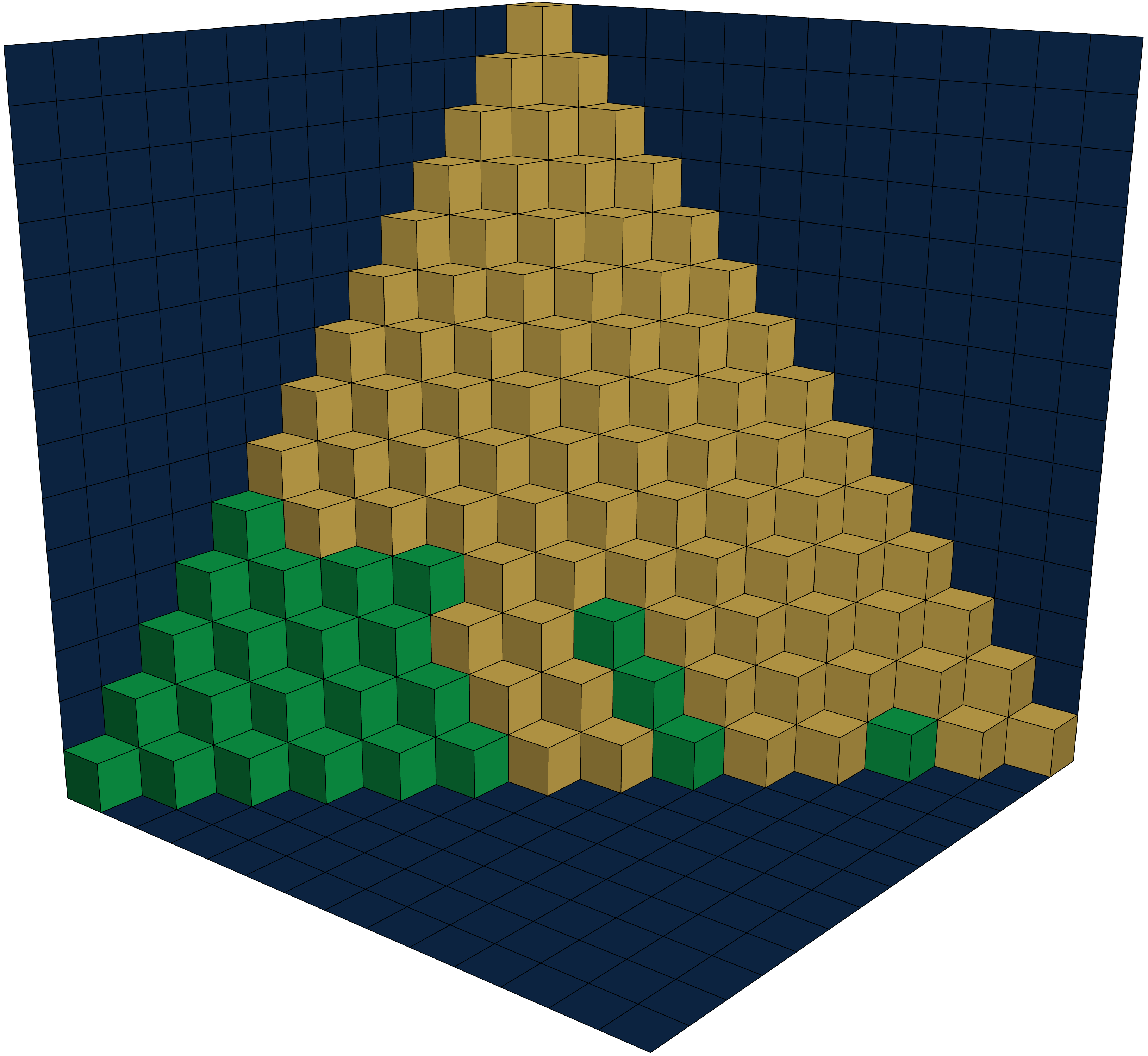}
		\caption{$\mathfrak{L}_{I,2}$}
		\label{fig:b3-rem}
	\end{subfigure}
	
	\medskip
	
	\begin{subfigure}[t]{0.315\textwidth}
		\centering
		\includegraphics[width=\linewidth]{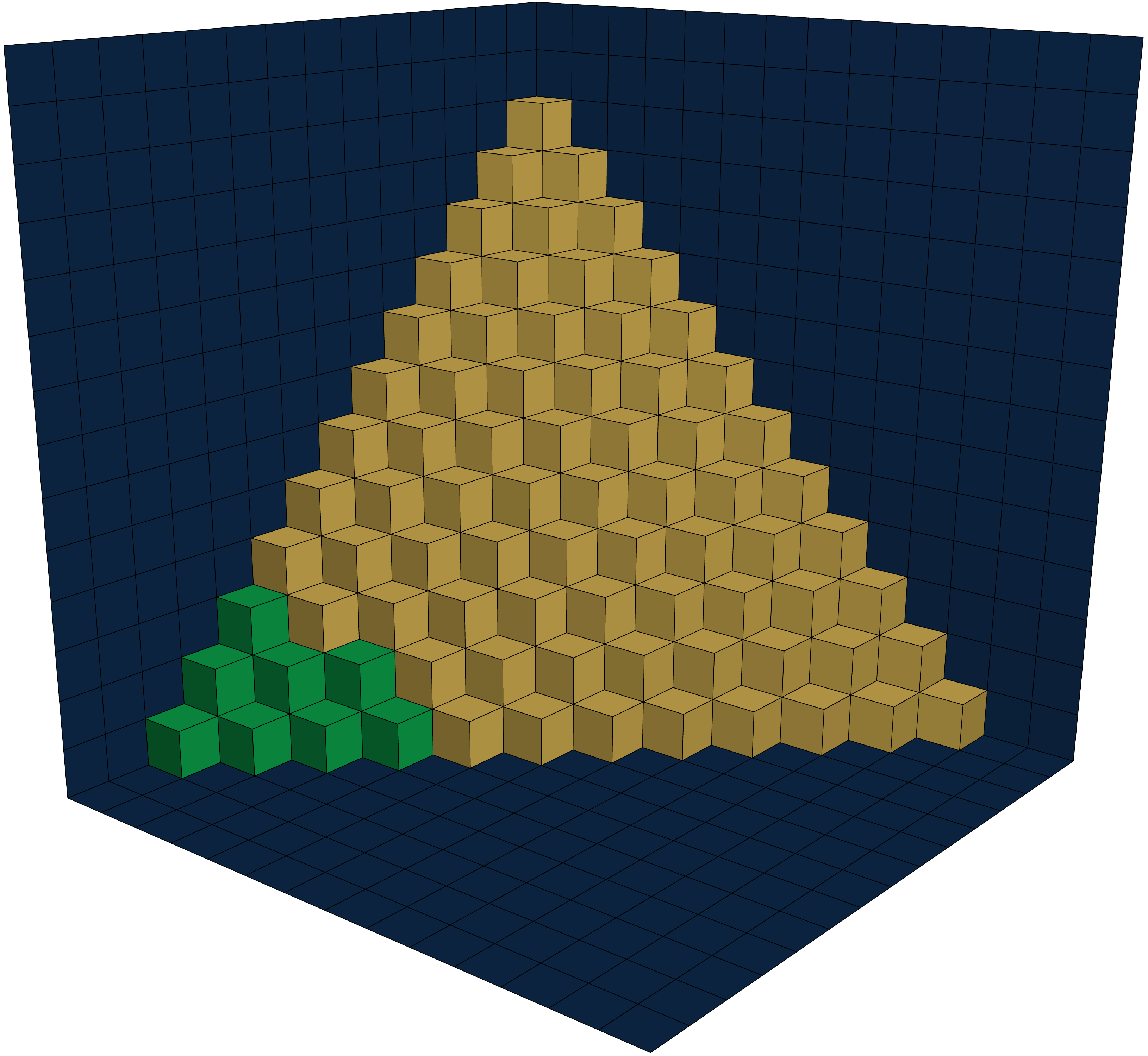}
		\caption{$\comp_{1,2}(\mathfrak{L}_{I,0})$}
		\label{fig:b1-rem-comp}
	\end{subfigure}
	\hfill
	\begin{subfigure}[t]{0.315\textwidth}
		\centering
		\includegraphics[width=\linewidth]{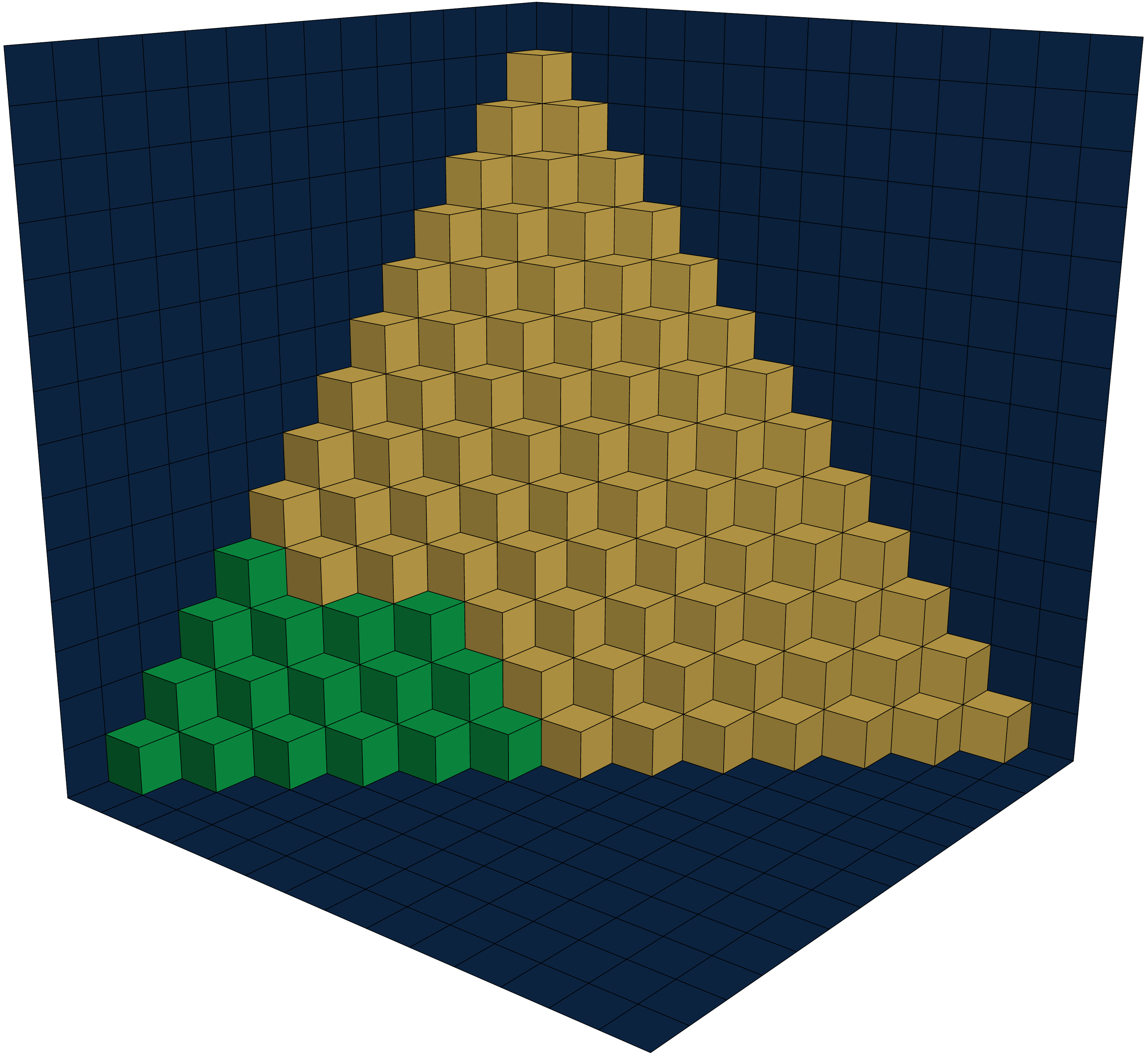}
		\caption{$\comp_{1,2}(\mathfrak{L}_{I,1})$}
		\label{fig:b2-rem-comp}
	\end{subfigure}
	\hfill
	\begin{subfigure}[t]{0.315\textwidth}
		\centering
		\includegraphics[width=\linewidth]{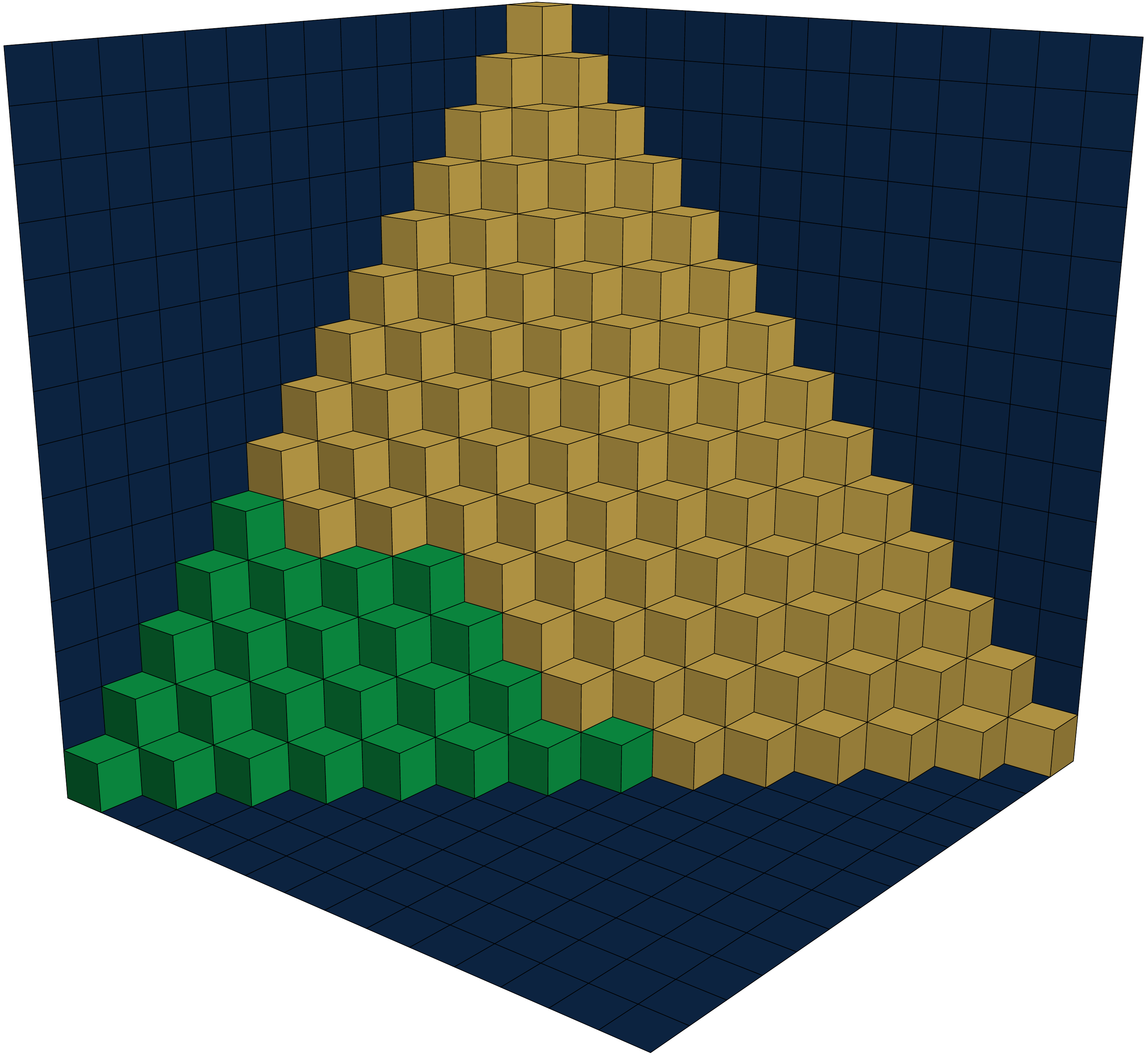}
		\caption{$\comp_{1,2}(\mathfrak{L}_{I,2})$}
		\label{fig:b3-rem-comp}
	\end{subfigure}
	
	\caption{An example of Theorem \ref{LexSegment} when $n=3$, with an ideal $I$ containing a regular sequence of length $3$ in degree $\alpha(I)$ that is recorded by the type.}
	\label{fig:removal-and-compression}
\end{figure}

Figure \ref{fig:removal-and-compression} gives an example of Theorem \ref{LexSegment}.

\begin{thm}\label{LexSegment}
	Suppose that $I\subseteq R$ is an ideal of type $v_{\ell,t} \in \N^{\ell +t+1}$ that is $\LPPT(v_{\ell, t})$ with $t\geq 0$.
	Then for all $k\in \{0,1,\dots, t\}$ there exists a set $\mathfrak{A}_k\subseteq \mathfrak{M}_{n,\alpha(I)+k}$ that is strongly stable,
	and there exists a bijection $\sigma: \mathfrak{L}_{I,k} \rightarrow \mathfrak{A}_k$ such that for all $m\in \mathfrak{L}_{I,k}$ we have $\max(m) = \max(\sigma(m))$.
	Furthermore, for all $k\in \{0,1,\dots, t-1\}$ we have $\uSdw(\mathfrak{A}_{k}) \subseteq \mathfrak{A}_{k+1}$,
	and for all $h\in \{1,2,3,\dots, n\}$  we have
	\begin{align*}
		\left\lbrace m\in \mathfrak{L}_{I,k+1}\setminus \uSdwB(\mathfrak{L}_{I,k}) \bigm | \max(m) = h\right\rbrace| 
		&= |\left\lbrace m\in \mathfrak{A}_{k+1}\setminus \uSdwB(\mathfrak{A}_{k}) \bigm | \max(m)=h \right\rbrace|\\
		&= |\left\lbrace m\in \mathfrak{A}_{k+1}\setminus \uSdw(\mathfrak{A}_{k}) \bigm | \max(m)=h \right\rbrace|.
	\end{align*}
\end{thm}
\begin{proof}
	Let $f_1,\dots, f_\ell$ be a regular sequence given by the type $v_{\ell, t}$ with $\deg(f_1)\leq \cdots \leq \deg(f_{\ell})$.
	If $k\in \{0,1,\dots, t\}$ and $\alpha(I)+k \leq \deg(f_1)$,
	then $\mathfrak{A}_k = \mathfrak{L}_{I,k}$ has all the properties required by the claim.
	Since we are going to use a sequence of compression operations below,
	and compression operations do not change strongly stable sets,
	without loss of generality we can assume $\deg(f_1) = \alpha(I)$.
	 
	For all $k\in \{0,1,\dots, t\}$ define $A_{k,0} = \mathfrak{L}_{I,k}$.
	For all $k\in \{0,1,\dots, t\}$ we will show that there exist sets $A_{k,1},A_{k,2},\dots, A_{k,n-2}$ such that for all $u\in \{1,2,\dots , n-2\}$ the following properties hold.
	\begin{enumerate}
		\item There exists a bijection $\sigma_{u}: A_{k,u-1}\rightarrow A_{k,u}$ such that $\max(m) = \max(\sigma_u(m))$.
		\item For $p,q\in [n]$ with $p<q$ and $u+2<q$,
		if $m\in A_{k,u}$ such that $x_q|m$ and $\max(m) = q$,
		then $\mathfrak{b}_{p,q}(m)\in A_{k,u}$.
		\item The set $A_{k,u} \cap \mathfrak{M}_{n,\alpha(I)+k,u+2}$ is strongly stable.
		\item For $k\in \{0,1,\dots, t-1\}$ we have $\uSdw_{\{1,2,\dots, u+2\}}(A_{k,u}\cap \mathfrak{M}_{n,\alpha(I)+k,u+2}) \subseteq A_{k+1,u}\cap \mathfrak{M}_{n,\alpha(I)+k+1,u+2}$.
		\item For $k\in \{0,1,\dots, t-1\}$ and all $h\in \{1,2,\dots, u+2\}$ we have
		\begin{align*}
			\left\lbrace m\in A_{k+1,u-1}\setminus \uSdwB(A_{k,u-1}) \bigm | \max(m) = h\right\rbrace| = |\left\lbrace m\in A_{k+1,u}\setminus \uSdw(A_{k,u}) \bigm | \max(m)=h \right\rbrace|.
		\end{align*}
		\item For $k\in \{0,1,\dots, t-1\}$ and all $s\in \{u+3,u+4,\dots, n\}$ we have $
		\uSdw_{\{s\}}(A_{k,u}\cap \mathfrak{M}_{n,\alpha(I)+k,s}) \subseteq A_{k+1,u}\cap \mathfrak{M}_{n,\alpha(I)+k+1,s}$.
		\item For $k\in \{0,1,\dots, t-1\}$ and $s\in \{u+3,u+4,\dots, n\}$ we have
		\begin{align*}
			|\{m\in A_{k+1,u-1}\setminus \uSdw_{\{s\}}(A_{k,u-1}) \bigm | \max(m)=s\}| = |\{m\in A_{k+1,u}\setminus \uSdw_{\{s\}}(A_{k,u}) \bigm | \max(m)=s\}|.
		\end{align*}
	\end{enumerate}
	We prove the claim by induction on $u$.
	So, we first handle the base case $u=1$.

	Suppose that $k\in \{0,1,\dots, t\}$.
	We will first consider the $(1,2)$ partitioning of $\mathfrak{M}_{n,\alpha(I)+k}$,
	because we are first going to use $\comp_{1,2}$.
	Let $A_{k,1} = \comp_{1,2}(A_{k,0})$.
	Note that $|A_{k,0}| = |A_{k,1}|$ by definition of $\comp_{1,2}$.
	Also there is a bijection $\sigma_1: A_{k,0}\rightarrow A_{k,1}$ such that $\max(m) = \max(\sigma_1(m))$,
	since $\mathfrak{P}_{I,k}$ does not contain powers of $x_1$ and we are using $\comp_{1,2}$.
	We will now show that $A_{k,1}$ partially satisfies the second property.
	
	Since $\gin(I)$ is strongly stable,
	if $p,q\in \{3,4,\dots, n\}$ with $p<q$ and $m\in \gin(I)_{\alpha(I)+k}$ such that $x_q|m$,
	then $\mathfrak{b}_{p,q}(m)\in \gin (I)$.
	For all $k'\in \{0,1,\dots, t\}$,
	$\Gamma_I(\alpha(I)+k')$ only contains monomials that are divisible by $x_2$,
	and not divisible by $x_3,x_4,\dots, x_n$.
	Hence, 
	if $p,q\in \{3,4,\dots, n\}$ with $p<q$ and $m\in A_{k,0}$ such that $x_q|m$,
	then we have $\mathfrak{b}_{p,q}(m)\in A_{k,0}$ by the definition of $\mathfrak{L}_{I,k}$.
	Thus, 
	by Corollary \ref{BorelComp}, 
	for all $p,q\in \{3,4,\dots, n\}$ with $p<q$ and $m\in A_{k,1}$ such that $x_q|m$,
	we have that $\mathfrak{b}_{p,q}(m)\in A_{k,1}$.
	Also by Corollary \ref{BorelComp},
	for all $m\in A_{k,1}$ such that $x_2|m$ we have $\mathfrak{b}_{1,2}(m)\in A_{k,1}$.
	
	Next,
	we suppose that $p\in \{1,2\}$ and $q\in \{3,4,\dots, n\}$,
	as this is the only case left to consider. 
	We can't get the same type of statement as the other cases,
	but we can prove the statement required by the second property.
	We show that if $m\in A_{k,1}$ such that $x_q|m$ and $\max(m)=q$,
	then we have $\mathfrak{b}_{p,q}(m)\in A_{k,1}$.
	Let $m^\ast \in \mathfrak{M}_{n,\alpha(I)+k}\cap A_{k,1}$ such that $x_q|m^\ast$ and $\max(m^\ast) = q$.
	We will look at $\Line(m^\ast)$ and $\Line(\mathfrak{b}_{p,q}(m^\ast))$.
	By Lemma \ref{lineStructure} and since $\gin(I)$ is strongly stable,
	$|\gin(I)\cap \Line(m^\ast)| + 1\leq |\gin(I)\cap \Line(\mathfrak{b}_{p,q}(m^\ast))|$.
	Thus,
	if $\gin(I)\cap \Line(m^\ast) = A_{k,0}\cap \Line(m^\ast)$ and $\gin(I)\cap \Line(\mathfrak{b}_{p,q}(m^\ast)) = A_{k,0}\cap \Line(\mathfrak{b_{p,q}}(m^\ast))$,
	then for all $m\in A_{k,1}\cap \Line(m^\ast)$ such that $x_q|m$,
	we have $\mathfrak{b}_{p,q}(m)\in A_{k,1}\cap \Line(\mathfrak{b}_{p,q}(m^\ast))$.
	
	Hence,
	we assume $\gin(I)\cap \Line(m^\ast) \neq A_{k,0}\cap \Line(m^\ast)$ or $\gin(I)\cap \Line(\mathfrak{b}_{p,q}(m^\ast)) \neq A_{k,0}\cap \Line(\mathfrak{b_{p,q}}(m^\ast))$.
	Then $\Line(\mathfrak{b}_{p,q}(m^\ast))\cap \mathfrak{P}_{I,k} \neq \emptyset$.
	So,
	by the definition of $\mathfrak{P}_{I,k}$,
	there exist $j\in \{0,1,\dots, k\}$ and monomials $m_j,m_{j+1}\dots, m_k$ with $\deg(m_i) = i$ and $x_1\not | m_i$ for all $i\in \{j,j+1,\dots, k\}$, 
	such that
	\begin{align*}
		\bigcup_{i=j}^k \{m_im\bigm | m\in \Gamma_I(\alpha(I)+k-i) \} = \Line(\mathfrak{b}_{p,q}(m^\ast))\cap \mathfrak{P}_{I,k}.
	\end{align*}
	Thus,
	\begin{align*}
		\bigcup_{i=j+1}^{k} \{m_imx_q/x_2\bigm | m\in \Gamma_I(\alpha(I)+k-i) \} = \Line(m^\ast)\cap \mathfrak{P}_{I,k}.
	\end{align*}
	Since the above unions are disjoint we have
	\begin{align*}
		|\Line(\mathfrak{b}_{p,q}(m^\ast))\cap \mathfrak{P}_{I,k}| &= \sum_{i=j}^k |\Gamma_I(\alpha(I)+k-i)| = \sum_{i=j}^k (\gamma_I(\alpha(I)+k-i)-1),\\
		|\Line(m^\ast)\cap \mathfrak{P}_{I,k}| &= \sum_{i=j+1}^k |\Gamma_I(\alpha(I)+k-i)| = \sum_{i=j+1}^k (\gamma_I(\alpha(I)+k-i) -1).
	\end{align*}
	We have $\Line(\mathfrak{b}_{p,q}(m^\ast)) = \mathfrak{M}_{n,\alpha(I)+k}(1,2, \mathfrak{f}_{1,2})$,
	for some $\mathfrak{f}_{1,2}$.
	Since $\max(m^\ast) = q$,
	for all $i\in \{q+1,q+2,\dots, n\}$ we have $\mathfrak{f}_{1,2}(i) = 0$.
	Set $m^\star = x_{3}^{\mathfrak{f}_{1,2}(3)}\cdots x_q^{\mathfrak{f}_{1,2}(q)}$ and $e=(\mathfrak{f}_{1,2}(3),\mathfrak{f}_{1,2}(4),\dots, \mathfrak{f}_{1,2}(q))$.
	We have that $x_2\not | m_j$,
	since otherwise we would contradict the minimality of $j$.
	Thus, $m_j= m^\star$.
	Since $\gin(I)$ is strongly stable and $I$ is $\LPPT(v_{\ell, t})$ we have 
	\begin{align*}
		|\gin(I)\cap \Line(\mathfrak{b}_{p,q}(m^\ast))| &\geq |\gin(I)\cap \Line(m^\ast)| + 1 + \fbl_I^{(q,e)}(\alpha(I)+k-j)\\
		&\geq |\gin(I)\cap \Line(m^\ast)| + \gamma_I(\alpha(I)+k-j).
	\end{align*}
	Hence,
	\begin{align*}
		&|A_{k,0}\cap \Line(\mathfrak{b}_{p,q}(m^\ast))| - |A_{k,0}\cap \Line(m^\ast)| \\
		&= |\gin(I)\cap \Line(\mathfrak{b}_{p,q}(m^\ast))|-|\mathfrak{P}_{I,k}\cap \Line(\mathfrak{b}_{p,q}(m^\ast))| -(|\gin(I)\cap \Line(m^\ast)|-|\mathfrak{P}_{I,k}\cap \Line(m^\ast)|)\\
		&= |\gin(I)\cap \Line(\mathfrak{b}_{p,q}(m^\ast))| - |\gin(I)\cap \Line(m^\ast)| -|\mathfrak{P}_{I,k}\cap \Line(\mathfrak{b}_{p,q}(m^\ast))| +|\mathfrak{P}_{I,k}\cap \Line(m^\ast)|\\
		&\geq \gamma_I(\alpha(I)+k-j) -\sum_{i=j}^k (\gamma_I(\alpha(I)+k-i)-1) + \sum_{i=j+1}^k (\gamma_I(\alpha(I)+k-i) -1)\\
		&= \gamma_I(\alpha(I)+k-j) - (\gamma_I(\alpha(I)+k-j) - 1)\\
		&=1.
	\end{align*}
	Hence, $|A_{k,0}\cap \Line(\mathfrak{b}_{p,q}(m^\ast))|  \geq |A_{k,0}\cap \Line(m^\ast)| +1$,
	and if $m\in A_{k,1}\cap \Line(m^\ast)$ such that $x_q|m$,
	then we have $\mathfrak{b}_{p,q}(m)\in A_{k,1}$.
	Since $m^\ast$ was arbitrary,
	if $m\in A_{k,1}$ such that $x_q|m$ and $\max(m)=q$,
	then we have $\mathfrak{b}_{p,q}(m)\in A_{k,1}$.
	Thus we can conclude that $A_{k,1}\cap \mathfrak{M}_{n,\alpha(I)+k,3}$ is strongly stable.
	
	Next, we prove the shadow properties.
	Suppose $k\in \{0,1,\dots, t-1\}$.
	We have $x_1^{\alpha(I)+k}\in A_{k,0}$ and $x_1^{\alpha(I)+k}x_1, x_1^{\alpha(I)+k}x_2\in A_{k+1,0}$ by definition of $\mathfrak{P}_{I,k}$ and $\mathfrak{P}_{I,k+1}$.
	For any $m\in A_{k,0}\cap \mathfrak{M}_{n,\alpha(I)+k,2}$ we have that $mx_2\in A_{k+1,0}\cap \mathfrak{M}_{n,\alpha(I)+k+1,2}$,
	since $\gin(I)$ is an ideal and by definition of $\mathfrak{L}_{I,k}$ and $\mathfrak{L}_{I,k+1}$.
	Thus, $\uSdw_{\{1,2\}}(A_{k,1}\cap \mathfrak{M}_{n,\alpha(I)+k,2}) \subseteq A_{k+1,1}\cap \mathfrak{M}_{n,\alpha(I)+k+1,2}$ and
	\begin{align*}
		\left\lbrace m\in A_{k+1,0}\setminus \uSdwB(A_{k,0}) \bigm | \max(m) = 2\right\rbrace| = |\left\lbrace m\in A_{k+1,1}\setminus \uSdwB(A_{k,1}) \bigm | \max(m)=2 \right\rbrace|.
	\end{align*}
	Suppose that $s\in \{3,4,\dots, n\}$ and we show that $\uSdw_{\{s\}}(A_{k,1}\cap \mathfrak{M}_{n,\alpha(I)+k,s}) \subseteq A_{k+1,1}\cap \mathfrak{M}_{n,\alpha(I)+k+1,s}$.
	Well, multiplication by $x_s$ produces an injection from $\mathfrak{M}_{n,\alpha(I)+k,s}$ to $\mathfrak{M}_{n,\alpha(I)+k+1,s}$.
	Since $\gin(I)$ is an ideal and by definition of $\mathfrak{L}_{I,k}$ and $\mathfrak{L}_{I,k+1}$, 
	multiplication by $x_s$ induces an injection from $A_{k,0}\cap \mathfrak{M}_{n,\alpha(I)+k,s}$ to $A_{k+1,0}\cap \mathfrak{M}_{n,\alpha(I)+k+1,s}$.
	Let the image of this injection be $A_{k,0}'\cap \mathfrak{M}_{n,\alpha(I)+k+1,s}$.
	Hence, $A_{k,0}'\cap \mathfrak{M}_{n,\alpha(I)+k+1,s}\subseteq A_{k+1,0}\cap \mathfrak{M}_{n,\alpha(I)+k+1,s}$.
	Thus,
	\begin{align*}
		\uSdw_{\{s\}}(A_{k,1}\cap \mathfrak{M}_{n,\alpha(I)+k,s}) &= \comp_{1,2}(A_{k,0}'\cap \mathfrak{M}_{n,\alpha(I)+k+1,s})\\
		&\subseteq \comp_{1,2}(A_{k+1,0}\cap \mathfrak{M}_{n,\alpha(I)+k+1,s})\\ 
		&= A_{k+1,1}\cap \mathfrak{M}_{n,\alpha(I)+k+1,s},
	\end{align*}
	because $s\geq 3>2$ and $\comp_{1,2}$ preserves subsets.
	In particular, $\uSdw_{\{1,2,3\}}(A_{k,1}\cap \mathfrak{M}_{n,\alpha(I)+k,3}) \subseteq A_{k+1,1}\cap \mathfrak{M}_{n,\alpha(I)+k+1,3}$,
	because we are dealing with strongly stable sets and shadows are the same as Borel shadows in this case.
	Also,
	\begin{align*}
		|\{m\in A_{k+1,0}\setminus \uSdw_{\{s\}}(A_{k,0}) \bigm | \max(m)=s\}| = |\{m\in A_{k+1,1}\setminus \uSdw_{\{s\}}(A_{k,1}) \bigm | \max(m)=s\}|.
	\end{align*}
	Therefore, we can conclude that for all $h\in \{1,2,3\}$ we have
	\begin{align*}
		\left\lbrace m\in A_{k+1,0}\setminus \uSdwB(A_{k,0}) \bigm | \max(m) = h\right\rbrace| = |\left\lbrace m\in A_{k+1,1}\setminus \uSdwB(A_{k,1}) \bigm | \max(m)=h \right\rbrace|.
	\end{align*}

	Thus, we have proved the inductive claim for $u=1$.
	So, suppose that $u\geq 1$ and that the claim holds for all $u'\leq u$.
	We prove that the claim holds for $u+1 \leq n-2$.
	For each $k\in \{0,1,\dots, t\}$ we will do the same sequence of compressions to construct $A_{k,u+1}$.
	For all $k\in \{0,1,\dots, t\}$ let $B_{k,u,0} = A_{k,u-1}$ and $B_{k,u,1} = A_{k,u}$.
	We will now show that for each $k\in \{0,1,\dots, t\}$,
	there exists a sequence of sets $B_{k,u,1},B_{k,u,2},\dots,$ that preserves the $7$ properties that $A_{k,u}$ satisfies.
	More formally, the following properties hold.
	\begin{enumerate}
		\item For $w\in \{1,2,\dots \}$ there exists a bijection $\tau_{k,w}: B_{k,u,w-1}\rightarrow B_{k,u,w}$ such that $\max(m) = \max(\tau_{k,w}(m))$.
		\item For $w\in \{1,2,\dots \}$, for $p,q\in [n]$ with $p<q$ and $u+2<q$,
		if $m\in B_{k,u,w}$ such that $x_q|m$ and $\max(m) = q$,
		then $\mathfrak{b}_{p,q}(m)\in B_{k,u,w}$.
		\item For $w\in \{1,2,\dots \}$,
		the set $B_{k,u,w}\cap \mathfrak{M}_{n,\alpha(I)+k,u+2}$ is strongly stable.
		\item For $k\in \{0,1,\dots, t-1\}$ and $w\in \{1,2,\dots \}$ we have $\uSdw_{\{1,2,\dots, u+2\}}(B_{k,u,w}\cap \mathfrak{M}_{n,\alpha(I)+k,u+2}) \subseteq B_{k+1,u,w}\cap \mathfrak{M}_{n,\alpha(I)+k+1,u+2}$.
		\item For $k\in \{0,1,\dots, t-1\}$, $w\in \{1,2,\dots \}$ and $h\in \{1,2,\dots, u+2\}$ we have
		\begin{align*}
			\left\lbrace m\in B_{k+1,u,w-1}\setminus \uSdwB(B_{k,u,w-1}) \bigm | \max(m) = h\right\rbrace| \\
			= |\left\lbrace m\in B_{k+1,u,w}\setminus \uSdw(B_{k,u,w}) \bigm | \max(m)=h \right\rbrace|.
		\end{align*}
		\item For $k\in \{0,1,\dots, t-1\}$, $w\in \{1,2,\dots \}$, and $s\in \{u+3,u+4,\dots, n\}$ we have $\uSdw_{\{s\}}(B_{k,u,w}\cap \mathfrak{M}_{n,\alpha(I)+k,s}) \subseteq B_{k+1,u,w}\cap \mathfrak{M}_{n,\alpha(I)+k+1,s}$.
		\item For $k\in \{0,1,\dots, t-1\}$, $w\in \{1,2,\dots \}$, and $s\in \{u+3,u+4,\dots, n\}$ we have
		\begin{align*}
			|\{m\in B_{k+1,u,w-1}\setminus \uSdw_{\{s\}}(B_{k,u, w-1}) \bigm | \max(m)=s\}| \\
			= |\{m\in B_{k+1,u,w}\setminus \uSdw_{\{s\}}(B_{k,u,w}) \bigm | \max(m)=s\}|.
		\end{align*}
	\end{enumerate}
	We prove the claim by induction on $w\in \{1,2,\dots \}$.
	The base case $w=1$ follows from the definition of $B_{k,u,0}$ and $B_{k,u,1}$.
	So, suppose that $w\geq 1$ and the claim holds for all $w'\leq w$.
	We prove the claim for $w+1$.
	If for all $k\in \{0,1,\dots, t\}$ we have that $B_{k,u,w}\cap \mathfrak{M}_{n,\alpha(I)+k,u+3}$ is strongly stable,
	then for all $k\in \{0,1,\dots, t\}$ we define $B_{k,u,w+1} = B_{k,u,w}$.
	So, suppose that there exists $f\in \{0,1,\dots, t\}$ such that $B_{f,u,w}\cap \mathfrak{M}_{n,\alpha(I)+f,u+3}$ is not strongly stable.
	Then there exist $i,j\in [u+2]$ with $i<j$ such that $\comp_{i,j}(B_{f,u,w}) \neq B_{f,u,w}$,
	since for $p,q\in [n]$ with $p<q$ and $u+2<q$,
	if $m\in B_{f,u,w}$ such that $x_q|m$ and $\max(m) = q$,
	then $\mathfrak{b}_{p,q}(m)\in B_{f,u,w}$.
	For all $k\in \{0,1,\dots, t\}$ we define $B_{k,u,w+1} = \comp_{i,j}(B_{k,u,w})$.
	
	We need to show that $B_{k,u,w+1}$ satisfies the $7$ desired properties.
	The first property is satisfied because $B_{k,u,w}\cap \mathfrak{M}_{n,\alpha(I)+k,u+2}$ is strongly stable and $i<j<u+3$.
	The second property is satisfied by Corollary \ref{BorelComp} because it is satisfied for $B_{k,u,w}$.
	The third property is satisfied because $B_{k,u,w}\cap \mathfrak{M}_{n,\alpha(I)+k,u+2}$ is strongly stable and $i<j<u+3$.
	The fourth property is satisfied because $B_{k,u,w}\cap \mathfrak{M}_{n,\alpha(I)+k,u+2}$ and $B_{k+1,u,w}\cap \mathfrak{M}_{n,\alpha(I)+k,u+2}$ are strongly stable and $i<j<u+3$.
	The fifth property is satisfied by reasoning similar to that used for the fourth.
	The sixth and seventh properties are satisfied by an argument similar to the one used in the base case $u=1$.
	Thus, we have completed the induction on $w$,
	and there exist $t+1$ infinite sequences such that each term of the sequences satisfies the $7$ properties described above.
	
	We will now show that each of these sequences is eventually constant.
	If there exists $c\in \N$ such that for all $k\in \{0,1,\dots, t\}$ we have $B_{k,u,c} = B_{k,u,c+1}$, 
	then $B_{k,u,c} = B_{k, u,c+i}$ for all $i\in \N$,
	by the definition of the sequences.
	So, we just need to show that such a situation occurs.
	To see this,
	note that every time when two consecutive entries in the sequence are not the same,
	we perform a compression.
	Each time we compress,
	loosely speaking,
	we move monomials later in the lexicographic order.
	Since the set of monomials under consideration is finite,
	we must eventually reach a point where the compression function is the identity function. 
	Let $c\in \N$ be the index where this occurs.
	Then for all $k\in \{0,1,\dots, t\}$ we have that $B_{k,u,c}\cap \mathfrak{M}_{n,\alpha(I)+k,u+3}$ is strongly stable,
	because otherwise some sequence would not be constant after $c$.
	
	We define $A_{k,u+1} = B_{k,u,c}$.
	Of course, we need to check that $A_{k,u+1}$ satisfies the $7$ desired properties.
	Properties $1,2,3,6,7$ are satisfied because $A_{k,u+1} = B_{k,u,c}$.
	Properties $4$ and $5$ follow from an argument similar to the one used in the base case.
	Hence, we have completed the induction on $u$.
	
	Finally, for all $k\in \{0,1,\dots, t\}$ we define $\mathfrak{A}_{k} = A_{k,n-2}$.
	Our desired function $\sigma$ is given by the composition of the $\sigma_u$.
	The strongly stable property for $\mathfrak{A}_{k}$ follows from the third property of $A_{k,n-2}$ since $u=n-2$.
	The last two shadow properties follow from properties $4$ and $5$ with $u=n-2$. 
\end{proof}

\begin{lem}\label{PowersSegment}
	Let $I\subseteq R$ be an ideal of type $v_{\ell,t} \in \N^{\ell +t+1}$ that is $\LPPT(v_{\ell, t})$,
	and let $e_1\leq e_2\leq \cdots \leq e_\ell$ be the degrees of a regular sequence given by $v_{\ell, t}$.
	Let $P$ be the ideal generated by $x_2^{e_2},\dots, x_\ell^{e_\ell}$.
	Then for all $k\in \{0,1,\dots, t\}$ we have $|\mathfrak{P}_{I,k}| = |P_{\alpha(I)+k}\cap\mathfrak{M}|$.
	Furthermore, for all $k\in \{1,2,\dots, t\}$ we have $|\uSdw_{\{2,3,\dots,n\}}(P_{\alpha(I)+k-1}\cap\mathfrak{M})| = |\uSdwB(\mathfrak{P}_{I,k-1}\cap\mathfrak{M})|$,
	and if $k$ is at least the first entry of $v$ then $|P_{\alpha(I)+k}\cap\mathfrak{M}| = \gamma_I(\alpha(I)+k)-1 + |\uSdw_{\{2,3,\dots,n\}}(P_{\alpha(I)+k-1}\cap\mathfrak{M})|$.
\end{lem}
\begin{proof}
	Since $I$ is $\LPPT(v_{\ell,t})$, 
	for all $k\in \{0,1,\dots, t\}$ we have the disjoint union
	\begin{align*}
		P_{\alpha(I)+k}\cap\mathfrak{M} &= \bigcup_{i=2}^{\gamma_I(\alpha(I)+k)} \{mx_i^{e_i}\bigm | m\in \mathfrak{M}_{n,\alpha(I)+k-e_i}\}\\
		&= \bigcup_{i=2}^{\gamma_I(\alpha(I)+k)} \left( \{mx_i^{e_i}\bigm | m\in \mathfrak{M}_{n,\alpha(I)+k-e_i} \mbox{ and } m\neq x_1^{\alpha(I)+k-e_i}\} \cup \{x_1^{\alpha(I)+k-e_i}x_i^{e_i}\}\right).
	\end{align*}
	From this we get $|\uSdw_{\{2,3,\dots,n\}}(P_{\alpha(I)+k-1}\cap\mathfrak{M})| = |\uSdwB(\mathfrak{P}_{I,k-1})|$.
	Also, if $k$ is at least the first entry of $v$ then $|P_{\alpha(I)+k}\cap\mathfrak{M}| = \gamma_I(\alpha(I)+k)-1 + |\uSdw_{\{2,3,\dots,n\}}(P_{\alpha(I)+k-1}\cap\mathfrak{M})|$,
	by the definition of $\uSdwB$ and since $\Gamma_I(\alpha(I)+k)$ only contains monomials that are not powers of $x_1$ and are not divisible by $x_3,x_4,\dots, x_n$.
	
	We prove that for all $k\in \{0,1,\dots, t\}$ we have $|\mathfrak{P}_{I,k}| = |P_{\alpha(I)+k}\cap\mathfrak{M}|$,
	by induction on $k$.
	The base cases where $k$ is at most the first entry of $v$ follows from the definition of $\LPPT(v_{\ell,t})$ and $\mathfrak{P}_{I,k}$.
	So suppose that $k$ is greater than the first entry of $v$ and that the claim holds for all $k'<k$.
	Well,
	\begin{align*}
		\mathfrak{P}_{I,k-1} &= \bigcup_{i=0}^{k-1} \uSdwB^{i}(\Gamma_I(\alpha(I)+k-1-i)).
	\end{align*}
	Hence,
	\begin{align*}
		\uSdwB(\mathfrak{P}_{I,k-1}) &= \uSdwB \left( \bigcup_{i=0}^{k-1} \uSdwB^{i}(\Gamma_I(\alpha(I)+k-1-i)) \right)\\
		&= \bigcup_{i=0}^{k-1} \uSdwB(\uSdw_{\{2,3,\dots, n\}}^{i}(\Gamma_I(\alpha(I)+k-1-i)))\\
		&= \bigcup_{i=0}^{k-1} \uSdwB^{i+1}(\Gamma_I(\alpha(I)+k - (i+1)))\\
		&= \bigcup_{i=1}^{k} \uSdwB^{i}(\Gamma_I(\alpha(I)+k - i)).
	\end{align*}
	Thus, $\mathfrak{P}_{I,k} = \Gamma_I(\alpha(I)+k) \cup \uSdwB(\mathfrak{P}_{I,k-1})$.
	So, $|\mathfrak{P}_{I,k}| = |\Gamma_I(\alpha(I)+k)| + |\uSdwB(\mathfrak{P}_{I,k-1})|$.
	
	We have $|P_{\alpha(I)+k}\cap\mathfrak{M}| = \gamma_I(\alpha(I)+k)-1 + |\uSdw_{\{2,3,\dots,n\}}(P_{\alpha(I)+k-1}\cap\mathfrak{M})|$,
	and $|\uSdw_{\{2,3,\dots,n\}}(P_{\alpha(I)+k-1}\cap\mathfrak{M})| = |\uSdwB(\mathfrak{P}_{I,k-1})|$.
	Therefore, the claim follows because $|\Gamma_I(\alpha(I)+k)| = \gamma_I(\alpha(I)+k)-1$.
\end{proof}

\section{Applications to Hilbert Functions and Graded Betti Numbers}\label{EGHResults}

\begin{thm}[Macaulay \cite{Macaulay}]\label{Macaulay}
	For any ideal, there exists a lex ideal with the same Hilbert function.
\end{thm}
 
\begin{thm}\label{EGHTheorem}
	Let $t\geq 0$, $\ell \geq 1$,
	and suppose that $I\subseteq R$ is an ideal of type $v_{\ell, t}\in \N^{\ell+ t+1}$.
	There exists an integer $D = D(v_{\ell,t})\in \N$,
	such that if $\alpha(I) \geq D$ then there is a lex plus powers ideal that has the same Hilbert function as $I$ in degrees $\leq \alpha(I)+t$.
	Furthermore, the lex ideal $L$ is generated from all the lex segments of size $|\mathfrak{L}_{I,k}|$ for all $k\in \{0,1,\dots, t\}$,
	and we take $P$ to be the same ideal as in Lemma \ref{PowersSegment}.
	Then our lex plus powers ideal is $L+P$.
	Also, if $\alpha(I)\geq D$,
	then $I$ is $\LPPT(v_{\ell,t})$.
\end{thm}
\begin{proof}
	By Theorem \ref{Monster2},
	there exists an integer $D = D(v_{\ell,t})\in \N$,
	such that if $\alpha(I) \geq D$ then $I$ is $\LPPT(v_{\ell, t})$.
	By Theorem \ref{LexSegment},
	for all $k\in \{0,1,\dots, t\}$ there exists a set $\mathfrak{A}_k\subseteq \mathfrak{M}_{n,\alpha(I)+k}$ that is strongly stable,
	and there exists a bijection $\sigma: \mathfrak{L}_{I,k} \rightarrow \mathfrak{A}_k$ such that for all $m\in \mathfrak{L}_{I,k}$ we have $\max(m) = \max(\sigma(m))$,
	where $k\in \{0,1,\dots, t\}$.
	Then for all $k\in \{0,1,\dots, t-1\}$ we have
	\begin{align*}
		|\uSdw(L_{\alpha(I)+k}\cap\mathfrak{M})| &= \sum_{m\in L_{\alpha(I)+k}} \Bez(m) \tag{By Lemma \ref{WeightsLemma}}\\
		&\leq \sum_{m\in \mathfrak{A}_{I,k}} \Bez(m) \tag{By Macaulay's Theorem \ref{Macaulay}}\\ 
		&= \sum_{m\in \mathfrak{L}_{I,k}} \Bez(m) \tag{By definition of $\sigma$ and $\Bez$}\\ 
		&= |\uSdwB(\mathfrak{L}_{I,k})| \tag{By definition of \Bez}\\ 
		&\leq |\mathfrak{A}_{I,k+1}| \tag{By Theorem \ref{LexSegment}}\\
		&= |L_{\alpha(I)+k+1}\cap\mathfrak{M}| \tag{By definition of $L$}.
	\end{align*}
	Hence $\uSdw(L_{\alpha(I)+k}) \subseteq L_{\alpha(I)+k+1}$ since the shadow of a lex segment is a lex segment, by Macaulay's Theorem \ref{Macaulay}.
	By Lemma \ref{PowersSegment}, for all $k\in \{0,1,\dots, t\}$ we have $|P_{\alpha(I)+k}\cap\mathfrak{M}| = |\mathfrak{P}_{I,k}|$.
	Let $V = \oplus_{k=0}^t L_{\alpha(I)+k}+ P_{\alpha(I)+k}$.
	Thus for all $k\in \{0,1,\dots, t\}$ we have
	\begin{align*}
		\dim I_{\alpha(I)+k} &= |\mathfrak{L}_{I,k}| + |\mathfrak{P}_{I,k}|\\ 
		&=|L_{\alpha(I)+k}\cap\mathfrak{M}|+|P_{\alpha(I)+k}\cap\mathfrak{M}|\\
		&=|(L_{\alpha(I)+k}\cup P_{\alpha(I)+k})\cap\mathfrak{M}|\\ 
		&= \dim V_{\alpha(I)+k}.
	\end{align*}
	Let $J$ be the ideal generated by $V$.
	Since for all $k\in \{0,1,\dots, t-1\}$ we have $\uSdw(L_{\alpha(I)+k}) \subseteq L_{\alpha(I)+k+1}$ and $\uSdw(P_{\alpha(I)+k}) \subseteq P_{\alpha(I)+k+1}$,
	we must have that $\dim J_{\alpha(I)+k} = \dim V_{\alpha(I)+k} = \dim I_{\alpha(I)+k}$ for all $k\in \{0,1,\dots, t\}$.
	Therefore, the proof is complete because $J$ is a lex plus powers ideal.
\end{proof}

\begin{rem}
	Note that we didn't use the full power of Theorem \ref{LexSegment} in the proof of Theorem \ref{EGHTheorem}.
	We could have used the $\uSdw(\mathfrak{A}_k) \subseteq \mathfrak{A}_{k+1}$ property together with the  Clements--Lindström Theorem \cite{ClementsLindstrom} to give an alternative proof.
	Due to this,
	for the cases that Theorem \ref{EGHTheorem} covers,
	the proof of Theorem \ref{EGHTheorem} gives a completely new proof of the Clements--Lindström Theorem.
\end{rem}

We now turn our attention to a result of Gotzmann,
which gives strong consequences when Macaulay's bound is sharp.

\begin{thm}[Gotzmann's Persistence Theorem \cite{Gotzmann}]\label{Gotzmann}
	Let $I\subseteq R$ be an ideal that is generated in degrees $\leq q$.
	If Macaulay's bound is sharp in degree $q$, then it is sharp in degree $q+k$ for all $k\geq 0$.
\end{thm}

\begin{thm}\label{EGHGotzmann}
	Let $t\geq 0$, $\ell \geq 1$,
	and suppose that $I\subseteq R$ is an ideal of minimal type $v_{\ell, t}\in \N^{\ell+ t+1}$.
	There exists an integer $D = D(v_{\ell,t})\in \N$,
	such that if $\alpha(I) \geq D$ then there exists a lex plus powers ideal $J$ with the same Hilbert function as $I$ in degrees $\leq \alpha(I)+t$.
	
	Furthermore,
	if $I$ is generated in degrees $\leq q$ with respect to $v_{\ell, t}$, $q\leq \alpha(I)+t-1$, and
	\begin{align*}
		\gamma_I(q+1)-\gamma_I(q)+\dim \uSdw(J_q) = \dim J_{q+1},
	\end{align*}
	then for all $i\in \{0,1,\dots, \alpha(I)+t-q-1\}$ we have 
	\begin{align*}
		\gamma_I(q+i+1)-\gamma_I(q+i)+\dim \uSdw(J_{q+i}) = \dim J_{q+i+1}.
	\end{align*}
	Furthermore, for all $i\in \{0,1,\dots, \alpha(I)+t-q-1\}$,
	if $q+i\leq \alpha(I)-1$ then there are exactly $\gamma_I(q+i+1)$ minimal generators in degree $q+i+1$ of $\gin(I)$,
	otherwise there are exactly $\gamma_I(q+i+1)-1$ minimal generators in degree $q+i+1$ of $\gin(I)$,
	and they are not divisible by $x_3,x_4,\dots, x_n$.
\end{thm}
\begin{proof}
	By Theorem \ref{EGHTheorem} there exists an integer $D = D(v_{\ell,t})\in \N$,
	such that if $\alpha(I) \geq D$ then there exists a lex plus powers ideal $J=L+P$ with the same Hilbert function as $I$,
	where $L$ and $P$ are defined as in Theorem \ref{EGHTheorem}.
	Also, $I$ is $\LPPT(v_{\ell,t})$.
	The claim clearly holds when $q\leq \alpha(I)-1$,
	so we suppose that $q\geq \alpha(I)$.
	We prove the claim by induction on $i$.
	First, we handle the base case $i=0$.
	From the assumptions we have $\gamma_I(q+1)-\gamma_I(q)+\dim \uSdw(J_q) = \dim J_{q+1}$,
	and we just need to prove that there are exactly $\gamma_I(q+1)-1$ minimal generators in degree $q+1$ of $\gin(I)$,
	and that they are not divisible by $x_3,x_4,\dots, x_n$.
	By Theorem \ref{EGHTheorem},
	$|L_{\alpha(I)+k}\cap\mathfrak{M}| = |\mathfrak{L}_{I,k}|$,
	for all $k\in \{0,1\dots, t\}$.
	By Lemma \ref{PowersSegment},
	for all $k\in \{0,1,\dots, t\}$ we have $|\mathfrak{P}_{I,k}| = |P_{\alpha(I)+k}\cap\mathfrak{M}|$.
	Also, for all $k\in \{1,2,\dots, t\}$ we have $|P_{\alpha(I)+k}\cap\mathfrak{M}| = \gamma_I(\alpha(I)+k)-1 + |\uSdw_{\{2,3,\dots,n\}}(P_{\alpha(I)+k-1}\cap\mathfrak{M})|$ and $|\uSdw_{\{2,3,\dots,n\}}(P_{\alpha(I)+k-1}\cap\mathfrak{M})| = |\uSdwB(\mathfrak{P}_{I,k-1})|$.
	Thus,
	we have that $\uSdw(\gin(I_q)) = \gin(I_{q+1})\setminus \Gamma_I(q+1)$,
	since $\gamma_I(q+1)-\gamma_I(q)+\dim \uSdw(J_q) = \dim J_{q+1}$.
	Hence, $\gin(I)$ has exactly $\gamma_I(q+1)-1$ minimal generators in degree $q+1$,
	and they are not divisible by $x_3,x_4,\dots, x_n$.
	
	So suppose that $i\geq 1$ and that the claim holds for all $i'<i$.
	Let $W$ be the ideal generated by $I_{q+i-1}$,
	and let $B$ be a reduced basis for $W_{q+i-1}$.
	Consider a Borel shadow plus expanders basis $E\cup \uSdwB(B)$ for $W_{q+i}$.
	Then $|E| \leq \gamma_I(q+i)-1$ by the inductive hypothesis.
	By Theorem \ref{GeneralGreen},
	for all $k\geq 1$ we have that $\uSdw^{k-1}(E)\cup \uSdwB^k(B)$ is a spanning set for $W_{q+i+k-1}$.
	In particular, $\uSdw(E)\cup \uSdwB^2(B)$ is a spanning set for $W_{q+i+1}$.
	
	Let $f_1,\dots f_\ell$ be a regular sequence given by the type $v_{\ell,t}$ condition.
	Since $I$ is generated in degrees $\leq q$ with respect to $v_{\ell, t}$,
	there are sets $A=\{f_{r_1},\dots, f_{r_s}\}$ and $C=\{f_{c_1},\dots, f_{c_u}\}$ with $r_1\leq \cdots \leq r_s\leq c_1\cdots \leq c_u$ and $s,u\geq 0$,
	such that $W_{q+i}+ \Span(A) = I_{q+i}$ and $W_{q+i+1}+ \uSdw(\Span(A)) + \Span(C) = I_{q+i+1}$.
	Note that we could have $A=C=\emptyset$ if none of the polynomials $f_1,\dots, f_\ell$ have degree $q+i$ or $q+i+1$.
	Also, $|E| = \gamma_I(q+i)-1-s$,
	and the polynomials in $A$ contribute $s$ minimal generators of degree $q+i$ of $\gin(I_{q+i})$,
	and these minimal generators are not divisible by $x_3,x_4,\dots, x_n$.
	Thus,
	\begin{align*}
		\dim I_{q+i+1} &= \dim (W_{q+i+1}+ \uSdw(\Span(A)) + \Span(C))\\
		&\leq |\uSdw(E)| + |\uSdwB^2(B)| + n|A| + |C|\\
		&\leq \dim(\uSdwB(\gin(I_{q+i}))) + |E|+|A|+|C|\\
		&= \dim(\uSdwB(\gin(I_{q+i}))) + \gamma_I(q+i)-1-s + s + u\\
		&= \dim(\uSdwB(\gin(I_{q+i}))) + \gamma_I(q+i+1)-1.
	\end{align*}
	Since $I$ is $\LPPT(v_{\ell, t})$ we must have
	\begin{align*}
		\dim I_{q+i+1} \geq \dim(\uSdwB(\gin(I_{q+i}))) + \gamma_I(q+i+1)-1.
	\end{align*}
	Hence,
	\begin{align*}
		\dim I_{q+i+1} = \dim(\uSdwB(\gin(I_{q+i}))) + \gamma_I(q+i+1)-1.
	\end{align*}
	Therefore, $\gamma_I(q+i+1)-\gamma_I(q+i)+\dim \uSdw(J_{q+i}) = \dim J_{q+i+1}$,
	and there are exactly $\gamma_I(q+i+1)-1$ minimal generators in degree $q+i+1$ of $\gin(I)$ which are not divisible by $x_3,x_4,\dots, x_n$.
\end{proof}

Next,
we turn to another classical result,
Green's Hyperplane Restriction Theorem \ref{GreenHyper},
which we state as Theorem \ref{GreenHyper}.

\begin{thm}[Green \cite{GreenHyper}]\label{GreenHyper}
	Let $I\subseteq R$ be an ideal and $L$ be the lex ideal with the same Hilbert function.
	If $q\in \{n,n-1,\dots, 3\}$ and $f_n,f_{n-1},\dots, f_q$ are general linear forms then for all $k\in \N$ we have
	\begin{align*}
		\dim \left(R/(L,f_n,\dots, f_q)\right)_k \geq \dim \left(R/(I,f_n,\dots, f_q)\right)_k.
	\end{align*}
\end{thm}

\begin{rem}
	Gasharov proved that Green's Hyperplane Restriction Theorem holds for fields with positive characteristic \cite{Gasharov}.
\end{rem}

We will generalize this result by replacing the lex ideal with the lex plus powers ideal.
In order to prove a generalization of Green's Hyperplane Restriction Theorem,
we need the following lemma.

\begin{lem}\label{OptimalSameMax}
	Suppose that $L,A\subseteq \mathfrak{M}_{n,k}$ such that $|L|=|A|$, $L$ is a lex segment, and $A$ is strongly stable.
	If $|\uSdw(A)| = |\uSdw(L)|$ then there exists a bijection $\sigma: L \rightarrow A$ such that $\max(m) = \max(\sigma(m))$.
\end{lem}
\begin{proof}
	During this proof,
	$\Bez = \Bez_n$ and $\uSdw = \uSdw_{\{ 1,2,\dots, n \}}$.
	We prove the claim by induction on $n\geq 2$.
	If $n=2$ then $A=L$ and the claim holds.
	So suppose that $n\geq 3$ and the claim holds for all $n'<n$.
	Let $A'$ and $L'$ be all the monomials in $A$ and $L$ respectively,
	that are not divisible by $x_n$.
	We show that $|A'| = |L'|$.
	Assume to the contrary that $|A'| \neq |L'|$.
	
	By Green's Theorem $|A'|\geq |L'|$.
	So, $|A'| > |L'|$.
	Let $A''$ be the last $|L'|$ monomials in the lex order in $A'$.
	Then $A''$ is strongly stable.
	By Macaulay's Theorem $|\uSdw(L')| \leq |\uSdw(A'')| <  |\uSdw(A')|$.
	By Lemma \ref{WeightsLemma} we have
	\begin{align*}
		\sum_{m\in A'} \Bez(m) - \sum_{m\in L'} \Bez(m)
		&= \sum_{m\in A'\setminus A''} \Bez(m) + \sum_{m\in A''} \Bez(m) - \sum_{m\in L'} \Bez(m)\\
		&= \sum_{m\in A'\setminus A''} \Bez(m) + |\uSdw(A'')| - |\uSdw(L')|\\
		&\geq \sum_{m\in A'\setminus A''}  \Bez(m).
	\end{align*}
	Thus
	\begin{align*}
		|\uSdw(A)| - |\uSdw(L)| 
		&= \sum_{m\in A} \Bez(m) -  \sum_{m\in L} \Bez(m) \tag{By Lemma \ref{WeightsLemma}}\\
		&= \sum_{m\in A\setminus A'} \Bez(m) + \sum_{m\in A'} \Bez(m) - \sum_{m\in L\setminus L'} \Bez(m) - \sum_{m\in L'} \Bez(m)\\
		&\geq \sum_{m\in A\setminus A'} \Bez(m) - \sum_{m\in L\setminus L'} \Bez(m) + \sum_{m\in A'\setminus A''}  \Bez(m)\\
		&= \sum_{i=1}^{|A\setminus A'|} 1 - \sum_{i=1}^{|L\setminus L'|} 1 + \sum_{m\in A'\setminus A''}  \Bez(m) \tag{By definition of $\Bez$}\\
		&= - \sum_{i=1}^{|L\setminus L'|-|A\setminus A'|} 1 + \sum_{m\in A'\setminus A''}  \Bez(m) \tag{Since $|L\setminus L'|>|A\setminus A'|$}\\
		&\geq - \sum_{i=1}^{|L\setminus L'|-|A\setminus A'|} 1 + \sum_{m\in A'\setminus A''}  2 \tag{By definition of $\Bez$}\\
		&= \sum_{m\in A'\setminus A''}  1 \tag{Since $|L\setminus L'|-|A\setminus A'| = |A'\setminus A''|$}\\
		&>0. \tag{Since $|A'|>|L'|$}
	\end{align*}
	Hence, $|\uSdw(A)| > |\uSdw(L)|$, a contradiction.
	So, $|A'|=|L'|$,
	and from here we must also have that $|A\setminus A'| = |L\setminus L'|$.
	Let $N=\{1,2,\dots, n-1\}$.
	If $|\uSdw_N(A')| = |\uSdw_N(L')|$ then we can use the inductive hypothesis.
	We must have that $|\uSdw_N(A')| = |\uSdw_N(L')|$ is true from the definition of $\Bez$ and Lemma \ref{WeightsLemma},
	since $|A\setminus A'| = |L\setminus L'|$.
	Therefore, claim now follows from the inductive hypothesis.
\end{proof}


\begin{thm}\label{EGHGreenHyper}
	Let $t\geq 0$, $\ell \geq 1$,
	and suppose that $I\subseteq R$ is an ideal of type $v_{\ell, t}\in \N^{\ell+t+1}$.
	There exists an integer $D = D(v_{\ell,t})\in \N$,
	such that if $\alpha(I) \geq D$ then there exists a lex plus powers ideal $J$ with the same Hilbert function as $I$ in degrees $\leq \alpha(I)+t$.
	
	Furthermore, if $q\in \{n,n-1,\dots, 3\}$ and $f_n,f_{n-1},\dots, f_q$ are general linear forms then for all $k\in \{0,1,\dots, t-1\}$ we have
	\begin{align*}
		\dim \left(R/(J,f_n,\dots, f_q)\right)_{\alpha(I)+k} \geq \dim \left(R/(I,f_n,\dots, f_q)\right)_{\alpha(I)+k}.
	\end{align*}
\end{thm}
\begin{proof}
	By Theorem \ref{EGHTheorem} there exists an integer $D_1 = D_1(v_{\ell,t})\in \N$,
	such that if $\alpha(I) \geq D_1$ then there exists a lex plus powers ideal $J=L+P$ with the same Hilbert function as $I$,
	where $L$ and $P$ are defined as in Theorem \ref{EGHTheorem}.
	Also, $I$ is $\LPPT(v_{\ell,t})$.
	The ideal $J$ is of some type $w_{\ell,t}\in \N^{\ell+t+1}$,
	where the first $\ell$ corresponding entries of $w_{\ell,t}$ and $v_{\ell, t}$ are equal.
	We have that $v=v_{\ell,t} = (a_1,\dots, a_\ell, b_0,\dots, b_t)$, and let
	\begin{align*}
		V = \{(a_1,\dots, a_\ell, c_0,\dots, c_t) \bigm | c_k \leq \MB_v(k)\}.
	\end{align*}
	Then $w_{\ell, t}\in V$ because for all $k\in \{0,1,\dots, t\}$ we have $\dim I_{\alpha(I)+k} \leq \MB_v(k)$.
	For every $u\in V$,
	by Theorem \ref{Monster2},
	there exists an integer $D_u = D_u(v_{\ell,t}) \in \N$,
	such that for any ideal $Q$ of type $u$ with $\alpha(Q) \geq D_u$,
	we have that $Q$ is $\LPPT(u)$.
	Let $D_2 = D_2(v_{\ell, t}) = \max_{u\in V} D_u$.
	Hence if $\alpha(J) \geq D_2$ then $J$ is $\LPPT(w_{\ell,t})$.
	Set $D= \max\{D_1,D_2\}$ and assume $\alpha(I) = \alpha(J) \geq D$.
	
	By Lemma \ref{Conca}, for all $k\in \{0,1,\dots, t-1\}$ we have to prove
	\begin{align*}
		\dim \left(R/(\gin(J),x_n,\dots, x_q)\right)_{\alpha(I)+k} \geq \dim \left(R/(\gin(I),x_n,\dots, x_q)\right)_{\alpha(I)+k}.
	\end{align*}
	By Theorem \ref{LexSegment},
	for all $k\in \{0,1,\dots, t\}$ there exists a set $\mathfrak{A}_{I,k}\subseteq \mathfrak{M}_{n,\alpha(I)+k}$ that is strongly stable,
	and there exists a bijection $\sigma_I: \mathfrak{L}_{I,k} \rightarrow \mathfrak{A}_{I,k}$ such that for all $m\in \mathfrak{L}_{I,k}$ we have $\max(m) = \max(\sigma_I(m))$.
	Also, by Theorem \ref{LexSegment},
	for all $k\in \{0,1,\dots, t\}$ there exists a set $\mathfrak{A}_{J,k}\subseteq \mathfrak{M}_{n,\alpha(I)+k}$ that is strongly stable,
	and there exists a bijection $\sigma_J: \mathfrak{L}_{J,k} \rightarrow \mathfrak{A}_{J,k}$ such that for all $m\in \mathfrak{L}_{J,k}$ we have $\max(m) = \max(\sigma_J(m))$.
	For all $k\in \{0,1,\dots, t\}$ we have $|L_{\alpha(I)+k}\cap\mathfrak{M}| = |\mathfrak{L}_{I,k}| = |\mathfrak{A}_{I,k}|$ by definition of $L$ and the bijection $\sigma_I$.
	By Lemma \ref{PowersSegment} for all $k\in \{0,1\dots, t\}$ we have $|P_{\alpha(I)+k}\cap\mathfrak{M}| = |\mathfrak{P}_{I,k}|$.
	Hence, we must have $|\mathfrak{A}_{J,k}| = |\mathfrak{L}_{J,k}| = |L_{\alpha(I)+k}\cap\mathfrak{M}|$ for all $k\in \{0,1,\dots, t\}$.
	
	Next, we show that $|\uSdw(\mathfrak{A}_{J,k})| = |\uSdw(L_{\alpha(I)+k}\cap\mathfrak{M})|$ for all $k\in \{0,1,\dots, t-1\}$.
	If $k$ is less than the first entry of $v$ then this clearly holds.
	Assume to the contrary $|\uSdw(\mathfrak{A}_{J,k})| > |\uSdw(L_{\alpha(I)+k}\cap\mathfrak{M})|$ for some $k\in \{0,1,\dots, t-1\}$ that is at least the first entry in $v$.
	Then we get a contradiction,
	\begin{align*}
		\dim \uSdw (J_{\alpha(I)+k}) &= \gamma_I(\alpha(I)+k) -1 + \dim \uSdw(\gin(J_{\alpha(I)+k}))\\
		&= \gamma_I(\alpha(I)+k) -1+|\uSdwB(\mathfrak{L}_{J,k})| + |\uSdwB(\mathfrak{P}_{J,k})|\\
		&= |\uSdw(\mathfrak{A}_{J,k})| + |\uSdwB(\mathfrak{P}_{J,k})|+\gamma_I(\alpha(I)+k) -1\\
		&= |\uSdw(\mathfrak{A}_{J,k})| + |\uSdw(P_{\alpha(I)+k}\cap\mathfrak{M})|\\
		&> |\uSdw(L_{\alpha(I)+k}\cap\mathfrak{M})| + |\uSdw(P_{\alpha(I)+k}\cap\mathfrak{M})|\\
		&= \dim \uSdw (J_{\alpha(I)+k}).
	\end{align*}
	Thus, $|\uSdw(\mathfrak{A}_{J,k})| = |\uSdw(L_{\alpha(I)+k}\cap\mathfrak{M})|$ for all $k\in \{0,1,\dots, t-1\}$.
	Hence, 
	by Lemma \ref{OptimalSameMax},
	for all $k\in \{0,1,\dots, t-1\}$ there exists a bijection $\sigma_k: \mathfrak{A}_{J,k} \rightarrow L_{\alpha(I)+k}\cap\mathfrak{M}$ such that for all $m\in \mathfrak{A}_{J,k}$ we have $\max(m)= \max(\sigma_k(m))$.
	
	Since $w_{\ell,t}$ and $v_{\ell, t}$ agree on the first $\ell$ entries,
	for all $k\in \{0,1,\dots, t-1\}$ there exists a bijection $\tau_k: \mathfrak{P}_{J,k} \rightarrow \mathfrak{P}_{I,k}$ such that for all $m\in \mathfrak{P}_{J,k}$ we have $\max(m)= \max(\tau_k(m))$.
	So, in order to prove our claim,
	we just need to show that $\mathfrak{A}_{J,k}$ has no more monomials that are not divisible by $x_q,x_{q+1},\dots, x_n$ than $\mathfrak{A}_{I,k}$.
	This follows from the bijection $\sigma_k$ and Green's Hyperplane Restriction Theorem \ref{GreenHyper}.
	Therefore, the claim is proved.
\end{proof}

Our final application in this section is to graded Betti numbers.

\begin{thm}[Bigatti--Hulett--Pardue]\label{BettiLex}
	Suppose that $I\subseteq R$ is an ideal,
	and let $L$ be the lex ideal with the same Hilbert function.
	Then we have an inequality of Betti numbers $b_{p,p+q}(I) \leq b_{p,p+q}(L)$.
\end{thm}

\begin{lem}\label{BettiLemma}
	Suppose that $M_1,M_2\subseteq K[x_1,\dots, x_n]$ are strongly stable sets of monomials of degree $d$ and $d+1$,
	such that $\uSdw(M_1)\subseteq M_2$.
	Let $L_1$ and $L_2$ be the lex segments of size $|M_1|$ and $|M_2|$ in degrees $d$ and $d+1$ respectively.
	If $|\uSdw(M_1)| = |\uSdw(L_1)|$ and $|\uSdw(M_2)| = |\uSdw(L_2)|$ then for all $h\in \{1,2,\dots, n\}$ we have
	\begin{align*}
		|\{m\in M_2\setminus \uSdw(M_1) \bigm | \max(m) = h\}| = |\{m\in L_2\setminus \uSdw(L_1) \bigm | \max(m) = h\}|.
	\end{align*}
\end{lem}
\begin{proof}
	By Lemma \ref{OptimalSameMax} there exist functions $\sigma_1: L_1 \rightarrow M_1$ and $\sigma_2: L_2 \rightarrow M_2$,
	such that for all $m\in L_1$ we have $\max(m) = \max(\sigma_1(m))$,
	and for all $m\in L_2$ we have $\max(m) = \max(\sigma_2(m))$.
	By Lemma \ref{BorelShadowUnique} for all $m\in \uSdw(L_1)$ there exist unique $i\in \{1,\dots, n\}$ and $m^\ast \in L_1$ such that $m=x_im^\ast$.
	So, we define the function $\tau: \uSdw(L_1) \rightarrow \uSdw(M_1)$ such that for all $m\in \uSdw(L_1)$ with unique decomposition $m=x_im^\ast$ we have $\tau(m) = x_i\sigma_1(m^\ast)$.
	Clearly, $\tau$ is a bijection and for all $m\in \uSdw(L_1)$ we have $\max(m) = \max(\tau(m))$.
	Therefore, the claim follows by the definitions of $\tau$ and $\sigma_2$.
\end{proof}

\begin{thm}\label{BettiEGH}
	Let $t\geq 0$, $\ell \geq 1$,
	and suppose that $I\subseteq R$ is an ideal of type $v_{\ell, t}\in \N^{\ell+t+1}$.
	There exists an integer $D = D(v_{\ell,t})\in \N$,
	such that if $\alpha(I) \geq D$ then there exists a lex plus powers ideal $J$ with the same Hilbert function as $I$ in degrees $\leq \alpha(I)+t$.
	Furthermore, we have the inequality of Betti numbers $b_{p,p+q}(I) \leq b_{p,p+q}(\gin(J))$ for $q\leq \alpha(I)+t-1$.
\end{thm}
\begin{proof}
	By Theorem \ref{EGHTheorem} there exists an integer $D_1 = D_1(v_{\ell,t})\in \N$,
	such that if $\alpha(I) \geq D_1$ then there exists a lex plus powers ideal $J=L+P$ with the same Hilbert function as $I$,
	where $L$ and $P$ are defined as in Theorem \ref{EGHTheorem}.
	Also, $I$ is $\LPPT(v_{\ell,t})$.
	The ideal $J$ is of some type $w_{\ell,t}\in \N^{\ell+t+1}$,
	where the first $\ell$ corresponding entries of $w_{\ell,t}$ and $v_{\ell, t}$ are equal.
	By an argument similar to the one in the proof of Theorem \ref{EGHGreenHyper},
	there exists an integer $D_2 = D_2(v_{\ell,t})\in \N$ such that if $\alpha(J) \geq D_2$ then $J$ is $\LPPT(w_{\ell,t})$.
	Let $D= \max\{D_1,D_2\}$ and assume $\alpha(I) = \alpha(J) \geq D$.
	
	Betti numbers can only increase when taking an initial ideal.
	Hence, we just need to show that $b_{p,p+q}(\gin(I)) \leq b_{p,p+q}(\gin(J))$.
	By a result of Eliahou and Kervaire \cite{EliahouKervaire} we have
	\begin{align*}
		b_{p,p+q}(\gin(I))  = \sum_{\substack{ m\in\gin(I) \\ m \text{ is a minimal generator } \\ \deg(m)=q}} \binom{\max(m)-1}{p},\\
		 b_{p,p+q}(\gin(J)) = \sum_{\substack{ m\in\gin(J) \\ m \text{ is a minimal generator } \\ \deg(m)=q}} \binom{\max(m)-1}{p}.
	\end{align*}
	Thus, we need to show
	\begin{align*}
		\sum_{\substack{ m\in\gin(I) \\ m \text{ is a minimal generator } \\ \deg(m)=q}} \binom{\max(m)-1}{p} \leq  \sum_{\substack{ m\in\gin(J) \\ m \text{ is a minimal generator } \\ \deg(m)=q}} \binom{\max(m)-1}{p}.
	\end{align*}
	Well, for $k\in \{0,1,\dots, t\}$ the set of monomials of degree $\alpha(I)+k$ in $\gin(I)$ is equal to the disjoint union $\mathfrak{P}_{I,k}\cup \mathfrak{L}_{I,k}$.
	Also, for $k\in \{0,1,\dots, t\}$ the set of monomials of degree $\alpha(I)+k$ in $\gin(J)$ is equal to the disjoint union $\mathfrak{P}_{J,k}\cup \mathfrak{L}_{J,k}$.
	Since $v_{\ell, t}$ and $w_{\ell, t}$ agree in the first $\ell $ corresponding entries, we have
	\begin{align*}
		\sum_{\substack{m\in \mathfrak{P}_{I,q-\alpha(I)} \subseteq \gin(I) \\ m \text{ is a minimal generator }\\ \deg(m)=q}} \binom{\max(m)-1}{p} =  \sum_{\substack{m\in\mathfrak{P}_{J,q-\alpha(I)}\subseteq \gin(J) \\ m \text{ is a minimal generator }\\ \deg(m)=q}} \binom{\max(m)-1}{p}.
	\end{align*}
	Hence, we need to show that
	\begin{align*}
		\sum_{\substack{ m\in \mathfrak{L}_{I,q-\alpha(I)}\subseteq \gin(I)\\ m \text{ is a minimal generator  }\\ \deg(m)=q}} \binom{\max(m)-1}{p} \leq  \sum_{\substack{m\in\mathfrak{L}_{J,q-\alpha(I)}\subseteq \gin(J)\\ m \text{ is a minimal generator }\\ \deg(m)=q}} \binom{\max(m)-1}{p}.
	\end{align*}
	By Theorem \ref{LexSegment}, 
	for all $k\in \{0,1,\dots, t\}$,
	there exists a set $\mathfrak{A}_{I,k}\subseteq \mathfrak{M}_{n,\alpha(I)+k}$ that is strongly stable,
	and there exists a bijection $\sigma_I: \mathfrak{L}_{I,k} \rightarrow \mathfrak{A}_{I,k}$ such that for all $m\in \mathfrak{L}_{I,k}$ we have $\max(m) = \max(\sigma_I(m))$.
	Furthermore, for all $k\in \{0,1,\dots, t-1\}$ we have $\uSdw(\mathfrak{A}_{I,k}) \subseteq \mathfrak{A}_{I,k+1}$,
	and for all $h\in \{1,2,3,\dots, n\}$  we have
	\begin{align*}
		\left\lbrace m\in \mathfrak{L}_{I,k+1}\setminus \uSdwB(\mathfrak{L}_{I,k}) \bigm | \max(m) = h\right\rbrace| 
		&= |\left\lbrace m\in \mathfrak{A}_{I,k+1}\setminus \uSdw(\mathfrak{A}_{I,k}) \bigm | \max(m)=h \right\rbrace|.
	\end{align*}
	Also, by Theorem \ref{LexSegment} there exist $\mathfrak{A}_{J,k}$ and $\sigma_J:\mathfrak{L}_{J,k} \rightarrow \mathfrak{A}_{J,k}$,
	with the same properties.
	By a similar argument to the proof of Theorem \ref{EGHGreenHyper} for all $k\in \{0,1,\dots, t-1\}$ we have
	\begin{align*}
		|\mathfrak{A}_{J,k}| &= |\mathfrak{L}_{J,k}| = |L_{\alpha(I)+k}\cap\mathfrak{M}| = |\mathfrak{L}_{I,k}| = |\mathfrak{A}_{I,k}|,\\
		|\uSdw(\mathfrak{A}_{J,k})| &= |\uSdw(L_{\alpha(I)+k})|.
	\end{align*}
	Hence, by Lemma \ref{BettiLemma},
	for all $k\in \{0,1,\dots, t-2\}$ and $h\in \{1,2,\dots, n\}$ we have 
	\begin{align*}
		\left\lbrace m\in L_{\alpha(I)+k+1}\cap\mathfrak{M}\setminus \uSdw(L_{\alpha(I)+k}) \bigm | \max(m) = h\right\rbrace| 
		= |\left\lbrace m\in \mathfrak{A}_{J,k+1}\setminus \uSdw(\mathfrak{A}_{J,k}) \bigm | \max(m)=h \right\rbrace|.
	\end{align*}
	Thus, by the above and Lemma \ref{OptimalSameMax} we have
	\begin{align*}
		\sum_{\substack{m\in\mathfrak{L}_{J,q-\alpha(I)}\subseteq \gin(J)\\ m \text{ is a minimal generator }\\ \deg(m)=q}} \binom{\max(m)-1}{p} =
		\sum_{\substack{m\in L_{q}\cap\mathfrak{M}\\ m \text{ is a minimal generator }\\ \deg(m)=q}} \binom{\max(m)-1}{p}.
	\end{align*}
	Let $A$ be the ideal generated by $\bigcup_{i=0}^t \mathfrak{A}_{I,i}$.
	Then we have
	\begin{align*}
		\sum_{\substack{ m\in \mathfrak{L}_{I,q-\alpha(I)}\subseteq \gin(I)\\ m \text{ is a minimal generator  }\\ \deg(m)=q}} \binom{\max(m)-1}{p} 
		= \sum_{\substack{ m\in  A\cap\mathfrak{M}\\ m \text{ is a minimal generator  }\\ \deg(m)=q}} \binom{\max(m)-1}{p}, 
	\end{align*}
	since 
	for all $h\in \{1,2,3,\dots, n\}$  we have
	\begin{align*}
		\left\lbrace m\in \mathfrak{L}_{I,k+1}\setminus \uSdwB(\mathfrak{L}_{I,k}) \bigm | \max(m) = h\right\rbrace| 
		&= |\left\lbrace m\in \mathfrak{A}_{I,k+1}\setminus \uSdw(\mathfrak{A}_{I,k}) \bigm | \max(m)=h \right\rbrace|.
	\end{align*}
	Hence, we need to show that
	\begin{align*}
		\sum_{\substack{ m\in  A\cap\mathfrak{M}\\ m \text{ is a minimal generator  }\\ \deg(m)=q}} \binom{\max(m)-1}{p}
		&\leq \sum_{\substack{m\in L_{q}\cap\mathfrak{M}\\ m \text{ is a minimal generator }\\ \deg(m)=q}} \binom{\max(m)-1}{p}.
	\end{align*}
	Therefore, the claim now follows by the result of Eliahou and Kervaire \cite{EliahouKervaire} and the Bigatti--Hulett--Pardue Theorem \ref{BettiLex}.
\end{proof}


\section{Translating from Ideals to Shadows}\label{TranslationSection}

The reader might find it helpful to recall the definitions in Section \ref{CombinatoricsMini}.
We develop a dictionary between algebra and combinatorics that shows how Theorem \ref{BFGG} follows from Theorem \ref{IntroEGHGotzmann}.
If we consider a monomial $x_1^{a_1}\cdots x_n^{a_n}\in R$ then we can identify it with $(a_1,\dots, a_n)\in \N^n$.
The monomials in $R$ form a poset under division and this partial order corresponds to the one for $\N^n$.
An {\it upset} is a set in a poset such that whenever an element is in our set,
any element larger than it is in our set.
If we have a monomial ideal then the monomials in it form an upset due to closure under multiplication of the variables.
Thus, any statement about Hilbert functions of monomial ideals is also a statement about the sizes of levels of upsets in $\N^n$.

Suppose that $\ell\leq n$ and we have $2\leq e_1\leq \cdots \leq e_\ell$.
Let $P$ be the ideal generated by $x_1^{e_1},\dots, x_\ell^{e_\ell}$ and let $Q$ be the corresponding upset.
Then $\N^n \setminus Q = \{0,1,\dots e_1-1\}\times \cdots \{0,1,\dots, e_\ell -1\} \times \N^{n-\ell}$.
Set $e= (e_1,\dots, e_\ell)$ and $G_e = \N^n \setminus Q$.
Note that the construction of $G_e$ corresponds to constructing the quotient $R/P$.
So if we have a monomial ideal $M$ and we make a statement about the Hilbert function of $M+P$,
this corresponds to a statement about the sizes of levels of upsets in $G_e$.
The lexicographic order on the monomials of $R$ translates to the lexicographic order on $\N^n$.
If we have a lex ideal $L$ and we consider the LPP ideal $L+P$,
then we get a corresponding upset in $G_e$,
which we call a {\it lex} upset.
Similarly, if we have a lex segment of degree $d$ in $R$,
this corresponds to a lex segment of level $d$ in $\N^n$,
and then to a lex segment of level $d$ in $G_e$.

The Clements--Lindström Theorem says that for every upset $U\subseteq G_e$,
there exists a lex upset $L\subseteq G_e$ such that the corresponding levels of $U$ and $L$ have the same size.
The Kruskal--Katona Theorem says the same, but for the case $G_e=H_n$.
These theorems are often not stated like this in the combinatorics literature, 
but stated by using shadows,
like we did in Section \ref{CombinatoricsMini}.

Combining the translation we have done in this section and Theorem \ref{IntroEGHGotzmann} we get Theorem \ref{BFGG}.


\section{Applications to Lefschetz Properties}\label{LefschetzApplicationsSection}

In this section we apply the techniques to conclude something about the WLP.
Lemma \ref{SWLPNewGenerators} and Corollary \ref{SWLPequiv} allow us to use the structure of the generic initial ideal from previous sections.

\begin{lem}\label{SWLPNewGenerators}
	Let $I\subseteq R$ be a strongly stable monomial ideal and let $H = (x_n,x_{n-1},\dots, x_{n-i+1})$ for some $i\in \{0,1,\dots, n-2\}$.
	Suppose that there is a monomial in $K[x_1,\dots, x_{n-i-1}]\setminus I_{d+1}$ of degree $d+1$ that is not divisible by $x_{n-i}$.
	Then $x_{n-i}$ is a WLE on $R/(I+H)$ in degree $d$ iff all minimal monomial generators that belong to $K[x_1,\dots, x_{n-i}]$ in degree $d+1$ are not divisible by $x_{n-i}$.
\end{lem}
\begin{proof}
	We prove the forward direction first.
	Assume to the contrary that there is some monomial $m\in I_{d+1}$ that is a minimal generator in degree $d+1$ and is divisible by $x_{n-i}$.
	Then multiplication by $x_{n-i}$ is not injective because $m/x_{n-i}\not\in I_d$ and $m\in I_d$.
	However, multiplication by $x_{n-i}$ is not surjective because there is some monomial not in $I_{d+1}$ that is not divisible by $x_{n-i}$.
	Thus, we have a contradiction with the assumption that $x_{n-i}$ is a weak Lefschetz element on $R/(I+H)$.
	Hence, all minimal monomial generators in degree $d+1$ are not divisible by $x_{n-i}$.
	
	Now for the backwards direction.
	Assume that all minimal monomial generators in degree $d+1$ are not divisible by $x_{n-i}$.
	Then multiplication by $x_{n-i}$ from degree $d$ to $d+1$ is injective because $I$ is strongly stable.
	Therefore, $x_{n-i}$ is a WLE on $R/(I+H)$ in degree $d$.
\end{proof}

\begin{cor}\label{SWLPequiv}
	Suppose $I\subseteq R$ is an ideal and let $g_1,\dots, g_n$ be general linear forms.
	Then for all $i\in \{0,1,\dots, n-2\}$ we have that $g_{n-i}$ is a WLE on $R/(I+(g_n,g_{n-1},\dots, g_{n-i+1}))$ iff $x_{n-i}$ is a WLE on $R/(\gin(I) +(x_n,x_{n-1},\dots,x_{n-i+1}))$.
\end{cor}
\begin{proof}
	Let $J = I+(g_n,g_{n-1},\dots, g_{n-i+1})$ and $H=(x_n,x_{n-1},\dots, x_{n-i+1})$.
	By Lemma \ref{Conca} we have $\dim(R/(J, g_{n-i}))_{k} = \dim(R/(\gin(I)+H,x_{n-i}))_{k}$ for all $k\in \N$.
	Therefore, the claim follows by Proposition \ref{WLPequiv}.
\end{proof}

Theorem \ref{WLPCITheorem} states that if we fix the difference between the degrees of a complete intersection,
then that complete intersection satisfies the WLP,
and it satisfies the WLP even after we include general linear forms with it,
as long as we make the minimal generating degree large enough.

\begin{thm}\label{WLPCITheorem}
	Let $t\geq 0$, $\ell\geq 1$,
	and suppose $I\subseteq R$ is a complete intersection of minimal type $v_{\ell, t}\in \N^{\ell+t+1}$.
	Let $g_3,g_4,\dots, g_n$ be general linear forms.
	There exists an integer $D = D(v_{\ell, t})\in \N$ such that if $\alpha(I) \geq D$,
	then for all $i\in \{0,1,\dots, n-3\}$ and all $k\in \{0,1,\dots, t\}$ we have that $g_{n-i}$ is a WLE in degree $\alpha(I)-1+k$ on $R/(I+(g_n,g_{n-1},\dots, g_{n-i+1}))$. 
\end{thm}
\begin{proof}
	By Theorem \ref{CIGinStructure} there exists an integer $D'$ such that if $\alpha(I)\geq D'$,
	then all the minimal monomial generators in degrees $\leq \alpha(I)+t$ are not divisible by $x_3,x_4,\dots, x_n$.
	Let $D=D'+1$.
	By Corollary \ref{SWLPequiv} we need to check that $x_{n-i}$ is a WLE in degree $\alpha(I)-1+k$ on $R/(\gin(I) +(x_n,x_{n-1},\dots,x_{n-i+1}))$.
	By Lemma \ref{SWLPNewGenerators} we have that $x_{n-i}$ is a WLE in degree $\alpha(I)-1+k$ on $R/(\gin(I) +(x_n,x_{n-1},\dots,x_{n-i+1}))$ iff all minimal monomial generators in degree $\alpha(I)+k$ are not divisible by $x_{n-i}$.
	Therefore, the claim holds by the existence of $D$.
\end{proof}

Theorem \ref{WLPCIPersistence} generalizes Theorem \ref{CIGinStructure},
from the case of a complete intersection,
to the case when the EGH Bound is sharp.

\begin{thm}\label{WLPCIPersistence}
	Let $t\geq 0$, $\ell \geq 1$,
	and suppose that $I\subseteq R$ is an ideal of minimal type $v_{\ell, t}\in \N^{\ell + t+1}$.
	There exists an integer $D = D(v_{\ell,t})\in \N$,
	such that if $\alpha(I) \geq D$ then there exists a lex plus powers ideal $J$ with the same Hilbert function as $I$ in degrees $\leq \alpha(I)+t$.
	
	Furthermore, if $I$ is generated in degrees $\leq q$ with respect to $v_{\ell, t}$,
	\begin{align*}
		\gamma_I(q+1)-\gamma_I(q)+\dim \uSdw(J_q) = \dim J_{q+1},
	\end{align*}
	and let $g_3,g_4,\dots, g_n$ be general linear forms,
	then for all $i\in \{0,1,\dots, n-3\}$ and all $w\in \{0,1,\dots, \alpha(I)+t-q-1\}$ we have that $g_{n-i}$ is a WLE in degree $q+w$ on $R/(I+(g_n,g_{n-1},\dots, g_{n-i+1}))$.
\end{thm}
\begin{proof}
	By Theorem \ref{EGHGotzmann} there exists an integer $D = D(v_{\ell,t})\in \N$,
	such that if $\alpha(I) \geq D$ then there exists a lex plus powers ideal $J$ with the same Hilbert function as $I$.
	Furthermore,
	the minimal generators in degree $q+w+1$ of $\gin(I)$ are not divisible by $x_3,x_4,\dots, x_n$.
	By Corollary \ref{SWLPequiv} we need to check that $x_{n-i}$ is a WLE in degree $q+w$ on $R/(\gin(I) +(x_n,x_{n-1},\dots,x_{n-i+1}))$.
	By Lemma \ref{SWLPNewGenerators} we have that $x_{n-i}$ is a WLE in degree $q+w$ on $R/(\gin(I) +(x_n,x_{n-1},\dots,x_{n-i+1}))$ iff all minimal monomial generators in degree $q+w+1$ are not divisible by $x_{n-i}$.
	Therefore, the claim holds.
\end{proof}

\begin{rem}
	Theorem \ref{WLPCITheorem} is a direct corollary of Theorem \ref{WLPCIPersistence},
	since the EGH Bound is sharp for each complete intersection in every degree.
	However, the proof of Theorem \ref{WLPCIPersistence} produces a larger $D$ compared to the proof of Theorem \ref{WLPCITheorem}.
	We have not optimized either of these proofs to get the best possible $D$ in this paper.
\end{rem}

\section{A Few Conjectures}\label{ConjecturesSection}
One can make many conjectures by trying to relax the asymptotic conditions in any of the results of this paper.
Instead of stating all of these,
we state conjectures that we think will yield the most progress on the topics discussed,
and potentially others.

\begin{con}\label{ginUnique}
	Let $I$ and $J$ be ideals generated by regular sequences $f_1,\dots, f_\ell$ and $g_1,\dots, g_\ell$ respectively.
	If for all $i\in \{1,\dots, \ell\}$ we have $\deg(f_i) = \deg(g_i)$ ,
	then $\gin(I) = \gin(J)$.
\end{con}

Conjecture \ref{ginUnique} implies the strong Lefschetz property for complete intersections,
because Stanley's Theorem \cite{Stanley} shows that monomial complete intersections satisfy the strong Lefschetz property.
Conjecture \ref{ginRevlex} implies Conjecture \ref{ginUnique}.

\begin{con}\label{ginRevlex}
	If $I$ is a complete intersection then $\gin(I)$ is an almost revlex ideal.
\end{con}


\section{Acknowledgments}\label{Thanks}
This paper was written while the author was supported by a Kenna Postdoctoral Fellowship at the University of Notre Dame.
The author had many conversations about the topics in this paper with Juan Migliore for almost $3$ years.
During $2$ of these $3$ years we had in-person meetings almost every week.
The paper would not have been possible without his mathematical expertise and moral support.

Michael Stillman made the author aware of \cite{FloystadStillman} during a visit at the Fields Institute for the Apprenticeship Program in Commutative Algebra.
This turned out to be a very useful tool for the results in the paper.
Funding for this trip was provided by the University of Notre Dame and the Fields Institute.

Alexandra Seceleanu read an early draft and provided many helpful comments that significantly improved the presentation of the paper.
The author had many discussion with her about the topics in the paper.
These discussion allowed the author to fix gaps in some proofs.

The author had discussions about generic initial ideals Ritvik Ramkumar.
These discussions helped the author understand the literature on generic initial ideals better.

Emanuela Marangone made the author aware of \cite{HarimaWachi, PalezzatoTorielli}.
The author was not aware that $k$-Lefschetz properties were studied in the literature before that.

Eric Riedl read an earlier draft of this paper and provided many comments for the abstract and introduction.
This significantly improved the presentation of the paper.


\bibliographystyle{acm}
\bibliography{EGH-Lefschetz-Gins}
	
\end{document}